\documentclass[letterpaper,reqno,11pt]{amsart}
\usepackage[utf8]{inputenc}
\usepackage[T1]{fontenc}
\usepackage{amsmath}
\usepackage{amssymb}
\usepackage{graphicx}
\usepackage[counterclockwise]{rotating}
\usepackage{mathtools}
\usepackage{wasysym}
\usepackage{amsfonts}
\usepackage{xy}
\usepackage{esint}
\usepackage{amsthm}

\usepackage{xfrac}
\usepackage[dvipsnames]{xcolor}
\usepackage[maxcitenames=5,maxbibnames=5,
backend=bibtex,
style=alphabetic,
]{biblatex}
\AtEveryBibitem{%
	\clearfield{doi}%
	\clearfield{issn}%
	\clearfield{url}
}
\DeclareFieldFormat[article,periodical]{volume}{\mkbibbold{#1}}
\renewbibmacro{in:}{%
	\ifentrytype{article}
	{}
	{\bibstring{in}%
		\printunit{\intitlepunct}}}

\usepackage{xurl}
\usepackage[urlcolor=Magenta, 
linkcolor=red, 
citecolor=Green,
colorlinks=true,
linktocpage=true]{hyperref}
\usepackage{dirtytalk}
\usepackage{enumitem}
\usepackage{tikz}
\usepackage{tikz-cd}
\tikzstyle{lien}=[->,>=stealth,rounded corners=5pt,thick]

\usepackage{geometry}
\usepackage{dirtytalk}

\usepackage{cleveref}

\usepackage{graphicx}
\usepackage{subcaption}
\usepackage{caption}

\crefformat{section}{\S#2#1#3} 
\crefformat{subsection}{\S#2#1#3}
\crefformat{subsubsection}{\S#2#1#3}

\newtheorem{thm}{Theorem}[section]

\newtheorem{conj}[thm]{Conjecture}
\newtheorem{prop}[thm]{Proposition}
\newtheorem{lem}[thm]{Lemma}
\newtheorem{cor}[thm]{Corollary}

\theoremstyle{definition}
\newtheorem{rem}[thm]{Remark}
\newtheorem{defi}[thm]{Definition}
\newtheorem{ex}[thm]{Example}

\numberwithin{equation}{section}

\author{Julien Moy}
\address{Laboratoire de Mathématiques d'Orsay, Universit{\'e} Paris-Saclay, 91405 Orsay cedex, France}
\email{julien.moy@universite-paris-saclay.fr}
\title[Pollicott--Ruelle resonances on Anosov surface covers]{Spectral gap for Pollicott--Ruelle resonances on random coverings of Anosov surfaces}

\begin{document}
	\maketitle
	\begin{abstract}
		Let $(M,g)$ be a closed Riemannian surface with Anosov geodesic flow. We prove the existence of a spectral gap for Pollicott--Ruelle resonances on random finite coverings of $M$ in the limit of large degree, which is expected to be optimal.
		
		The proof combines the recent strong convergence results of Magee, Puder and van Handel for permutation representations of surface groups with an analysis of the spherical mean operator on the universal cover of $M$.
		
	\end{abstract}
	
	\tableofcontents 
	\section{Introduction}
	
	\subsection{Motivation} 
	
	Let $(M,g)$ be a smooth connected Riemannian manifold, without boundary. Let $\mathcal M=SM$ denote the unit tangent bundle of $M$, and let $\varphi_t:\mathcal M\to \mathcal M$ denote the geodesic flow. This flow preserves a natural smooth probability measure on $\mathcal M$, known as the \emph{Liouville measure}, which is the normalized lift of the Riemannian volume to $SM$. The Liouville measure is denoted by $\mu_{\rm L}$. To study the large time behavior of the flow, one introduces the correlation function
	\[\rho_{f,g}(t):=\int_{\mathcal M}(f\circ \varphi_{-t}) g \mathrm{d}\mu_{\rm L},\]
	for $f$ and $g$ sufficiently regular on $\mathcal M$, and $\int g\mathrm{d}\mu_{\rm L}=0$. From now on, we assume that the geodesic flow is \emph{Anosov}, or \emph{uniformly hyperbolic}, which roughly means that it is strongly chaotic, \emph{i.e.} sensitive to perturbations of initial conditions. For example, the geodesic flow is Anosov whenever $(M,g)$ has negative sectional curvatures. When $(\varphi_t)$ is Anosov, it is known that the flow is \emph{mixing} with respect to the Liouville measure, meaning that one has
	\begin{equation} \label{eq: convergence correlations} \rho_{f,g}(t)\underset{t\to +\infty}{\longrightarrow} 0.\end{equation}
	Rather than studying $\rho_{f,g}(t)$ directly, it is useful to consider its Laplace transform
	\[\int_0^{+\infty} \mathrm{e}^{-zt}\rho_{f,g}(t)=\int_{\mathcal M}\int_0^{+\infty} \mathrm{e}^{-zt}(f\circ \varphi_{-t})g\mathrm{d}t\mathrm{d}\mu_{\rm L},\]
	which is well defined and holomorphic in the half-plane $\operatorname{Re}(z)>0$. The heuristic is that the long time behavior of $\rho_{f,g}$ should be governed by the poles of the meromorphic continuation of its Laplace transform in the negative half-plane, whose existence follows from Theorem~\ref{thm: BL FS} below. Note that the meromorphic continuation of the Laplace transform of $\rho_{f,g}$ in a small strip left to the imaginary axis was already obtained by Pollicott \cite{Pollicott85} and Ruelle \cite{Ruelle87Meromorphic} by using symbolic dynamics. The poles are known as \emph{Pollicott--Ruelle resonances}.
	
	We are interested in the rate of decay of correlations; since the presence of resonances close to the imaginary axis obstructs the rapid decay of $\rho_{f,g}$, a natural problem is therefore to establish a \emph{spectral gap}, meaning a resonance-free strip left to the imaginary axis. \medskip
	
	For $f\in C^\infty(\mathcal M)$, the function $f\circ \varphi_{-t}$ is the solution at time $t$ of the following ODE:
	\[\partial_t u(t,x)=-X u(t,x), \qquad u(0,x)=f(x),\]
	where $X$ is the geodesic vector field. Thus, we can write formally $\mathrm{e}^{-tX}f:=f\circ \varphi_{-t}$. This motivates the definition, for $\operatorname{Re}(z)>0$, of the \emph{resolvent operator}
	\begin{equation} \label{eq: def resolvent}(X+z)^{-1}f:=\int_0^{+\infty} \mathrm{e}^{-zt}(f\circ \varphi_{-t})\mathrm{d}t,\end{equation}
	which is well defined as an operator $C^\infty(\mathcal M)\to \mathcal D'(\mathcal M)$. The Laplace transform of $\rho_{f,g}$ can then be expressed as $\langle (X+z)^{-1}f,g\rangle_{\mathcal D'(\mathcal M),C^\infty(\mathcal M)}$.\medskip
	
	A functional analytic approach to the study of the operator $(X+z)$, which is built on the use of anisotropic Banach spaces of distributions that are well adapted to the dynamics, was developed by Liverani \cite{Liverani2004}, Butterley--Liverani \cite{Butterley2007} and Faure--Sj\"ostrand \cite{Faure2011}, following earlier works on Anosov diffeomorphisms (non-exhaustively, see \cite{BKL2002,GouezelLiverani2006,BaladiTsujii2007}). It led to the following result:
	\begin{thm}[Butterley--Liverani, Faure--Sj\"ostrand] \label{thm: BL FS} The family of operators $(X+z)^{-1}:C^\infty(\mathcal M)\to \mathcal D'(\mathcal M)$, initially defined and holomorphic in $\{\operatorname{Re}(z)>0\}$, extends meromorphically to $z\in \mathbb{C}$. Its poles have finite rank and are called \emph{Pollicott--Ruelle resonances}. \end{thm}
	
	The set of Pollicott--Ruelle resonances of $M$ is denoted by $\operatorname{Res}(X_{M})$. If $z_0$ is a resonance, one can define a \emph{spectral projector} $\Pi_{z_0}$ of finite rank by integrating the resolvent along a small circle centered at $z_0$. We call \emph{multiplicity} of $z_0$ the rank of this projector.
	
	Mixing properties of $\varphi_t$ usually follow from the existence of a resonance-free region in $\mathbb{C}$, together with resolvent estimates that come with it. The following result was proved by Tsujii \cite{Tsujii_2010}, and by Nonnenmacher--Zworski \cite{nonnenmacher2015decay} in a more general situation (see also Faure--Tsujii \cite{faure2024micro} for a more precise description of the spectrum):
	\begin{thm}[Tsujii, Nonnenmacher--Zworski] \label{thm: tsujii essential gap} Let $\gamma_0$ be the minimal exponential growth rate of the \emph{unstable Jacobian} defined in \eqref{eq: defi gamma 0 higher dim}. Then, for any $\varepsilon>0$, there are finitely many resonances in the half space $\operatorname{Re}(z)>-\frac{\gamma_{0}}{2}+\varepsilon$.\end{thm}
	Since $0$ is the only resonance lying on the imaginary axis---this fact is equivalent to weak mixing---, a byproduct of Theorem~\ref{thm: tsujii essential gap} (together with a resolvent bound) is the following estimate on the decay of correlations, which recovers results of Dolgopyat \cite{Dolgopyat1998} and Liverani \cite{Liverani2004}.
	\begin{cor}[Exponential decay of correlations] There is a constant $\kappa>0$ such that for any $f,g\in C^\infty(\mathcal M)$, with $\int g=0$:
		\[\rho_{f,g}(t)=\mathcal O_{f,g}(\mathrm{e}^{-\kappa t}).\]
	\end{cor} 
	One can be more precise and give an expansion of $\rho_{f,g}(t)$ with remainder $\mathcal O(\mathrm{e}^{-(\gamma_0/2-\varepsilon)t})$ that involves resonances of real part $>-\gamma_0/2$, see \emph{e.g.} \cite[Corollary 5]{nonnenmacher2015decay}.
	\subsubsection{Spectral properties of large random surfaces}
		In recent years, growing attention has been directed towards the spectral properties of large random objects. In particular, some effort has been devoted to understanding the spectrum of the Laplace--Beltrami operator on random hyperbolic surfaces of large genus. We highlight three models of interest, see \S\ref{subsec: related work} for references.
	\begin{itemize}
		\item The Weil--Petersson model of random hyperbolic surfaces of genus $g$.
		\item The Brooks--Makover model of random hyperbolic surfaces obtained by gluing $2n$ ideal hyperbolic triangles in a certain way, introduced in \cite{BrooksMakover}.
		\item The model of random covers of degree $n$ of a fixed negatively curved surface, introduced by Magee, Naud and Puder \cite{MNP}.
	\end{itemize}
	A central question for these models concerns the existence of a probabilistic \emph{spectral gap} for the Laplace--Beltrami operator, that is a uniform lower bound on the first nontrivial eigenvalue, that holds asymptotically almost surely in the limit of large genus. Optimal spectral gaps for the Laplacian have now been obtained by Anantharaman--Monk \cite{AnMo3} for the Weil--Petersson model (see also \cite{hideWeilPol}), by Magee--Puder--van Handel \cite{magee2025strongconvergenceuniformlyrandom} for the random cover model, and by Shen--Wu \cite{shen2025nearlyoptimalspectralgaps} for the Brooks--Makover model---see \S \ref{subsec: related work} for prior results. These give some explicit rate of convergence to equilibrium for the heat flow.\medskip
	
	In constant curvature, there is an explicit correspondence between the Laplace and the Pollicott--Ruelle spectra. In particular, a spectral gap for the Laplacian also yields a spectral gap for the Pollicott--Ruelle spectrum, leading to a quantitative rate of decay for correlations of the geodesic flow. However, this correspondence breaks down in variable curvature. It is thus desirable to understand whether one can establish quantitative decay rates for dynamical correlations on a model of random surfaces of variable negative curvature, or satisfying the weaker assumption that the geodesic flow is Anosov.
	
	\subsection{Main results} Let $(M,g)$ be a closed manifold. To state our results, we need to define the topological pressure of a potential $F\in C(SM)$. We will review the various definitions of the pressure in \S \ref{subsec: pressure}. For now, we only mention that if $\psi_u$ is the unstable Jacobian, defined in \eqref{eq: def unstable jac}, and $q\in \mathbb R$, then the pressure of $-q\psi_u$ is given by (see \cite{parrypressure})
	\[\operatorname{Pr}(-q\psi_u)=\underset{T\to +\infty}{\lim} \frac 1T\log \sum_{\ell(\gamma)\in [T,T+1]} \frac{1}{|\det(I-P_\gamma)|^q}, \]
	where the sum runs over periodic orbits $\gamma$ of the geodesic flow, $\ell(\gamma)$ denotes the length of an orbit and $P_\gamma$ denotes the linearized Poincar\'e map of the orbit. Note that $\operatorname{Pr}(0)=h_{\rm top}$ is the topological entropy of the flow, and $\operatorname{Pr}(-\psi^u)=0$. Moreover, when $M$ has nonconstant curvature, the function $q\mapsto \operatorname{Pr}(-q\psi^u)$ is strictly convex. We define
	\[\delta_0:=\frac{\operatorname{Pr}(-2\psi^u)}{2}.\]
	Note that we can bound
	\begin{equation} \label{eq: bound delta_0 nana}-\frac{h_{\rm top}}{2}\le \delta_0\le -\frac{\gamma_0}{2},\end{equation}
	where $\gamma_0$ is the minimal rate of expansion of the flow appearing in Theorem \ref{thm: tsujii essential gap}. Moreover, these inequalities are strict when $(M,g)$ is a surface with nonconstant negative curvature, see Proposition \ref{prop: jaco coho const} in the Appendix for a more general result.

	\subsubsection{Almost sure spectral gap for uniformly random Anosov surface covers} We now describe the model of random covers. Let $M$ be a closed Anosov manifold with fundamental group $\Gamma$. The universal cover of $M$ is denoted by $\widetilde M$. Then $\Gamma$ acts by isometries on $\widetilde M$, so that $M$ identifies with the quotient 
	\[M=\Gamma\backslash \widetilde M.\]
	We denote by $S_n$ the symmetric group over $[n]:=\{1,\ldots,n\}$. Let $\phi_n:\Gamma\to S_n$ be a group homomorphism; then, $\Gamma$ acts on the product $\widetilde M\times [n]$ by $\gamma.(x,i)=(\gamma x,\phi_n[\gamma](i))$. One defines a covering of degree $n$ of $M$ associated with $\phi_n$ by setting 
	\[M_n:=\Gamma\backslash (\widetilde M\times [n]).\]
	Since $\Gamma$ is finitely generated, the set $\operatorname{Hom}(\Gamma,S_n)$ is finite, hence we can endow it with the uniform probability measure. This gives a model of \emph{uniformly random Riemannian coverings} of degree $n$ of $M$. 
	\begin{rem}Two homomorphisms $\Gamma\to S_n$ that are conjugated by an interior automorphism of $S_n$ give rise to isometric Riemannian coverings of $M$; we do not take the quotient of $\operatorname{Hom}(\Gamma,S_n)$ by this relation however.\end{rem}

	\begin{defi}\label{defi: a.a.s} Let $(\Omega_n,\mathbb P_n)$ be a sequence of discrete probability spaces. A sequence of events $\{A_n\}_{n=1}^{+\infty}$ such that $A_n\subset \Omega_n$ occurs \emph{asymptotically almost surely}, or \emph{a.a.s.}, if 
	\[\mathbb P_n(A_n)\underset{n\to +\infty}{\longrightarrow} 1.\]\end{defi}
	
	If $M'\to M$ is a finite Riemannian covering, then any Pollicott--Ruelle resonance of $M$ is also a resonance of $M'$ with greater or equal multiplicity, because of the embedding $C^\infty(SM)\subset C^\infty(SM')$. Thus, we are only interest in \emph{new} resonances.
	\begin{defi}Let $M'\to M$ be a finite Riemannian covering. We say that $z_0\in \mathbb C$ is a \emph{new} resonance if it appears in $\operatorname{Res}(X_{M'})$ with higher multiplicity than in $\operatorname{Res}(X_M)$.
		
	\end{defi}Our main result is as follows. Recall that $\delta_0=\frac 12\operatorname{Pr}(-2\psi^u)$ where $\psi^u$ is the unstable Jacobian.

	\begin{thm}\label{thm: spectral gap} Let $M$ be a closed Anosov surface. Fix a compact set $\mathcal K\subset \{z\in\mathbb{C}~:~ \operatorname{Re}(z)> \delta_0\}$. If $\phi_n\in \operatorname{Hom}(\Gamma,S_n)$ is picked uniformly at random and $M_n\to M$ is the associated covering, then asymptotically almost surely, $X_{M_n}$ has no new resonance in $\mathcal K$.\end{thm}
	
We discuss the optimality of this result in \S \ref{subsec: optimality}.	
	\subsection{Strong convergence of representations and spectral gaps} Instead of working with coverings of $M$, it is more convenient to change perspective and work with unitary flat bundles over $M$ associated with permutation representations of the fundamental group. Let $M_n\to M$ be a Riemannian covering of degree $n$, associated with a homomorphism $\phi_n:\Gamma\to S_n$. There is a natural splitting
	\[C^\infty(M_n)=C^\infty(M)\oplus C^\infty_{\rm new}(M_n).\]
	The first factor corresponds to lifts of functions from $M$ to $M_n$, and the factor $C^\infty_{\rm new}(M_n)$ consists of functions that have zero average over all the fibers of the projection $M_n\to M$.
	
	Elements of $C^\infty(M_n)$ can be viewed as functions $f:\widetilde M\times [n]\to \mathbb C$ satisfying the invariance property
	\[\forall \gamma\in \Gamma, \quad f(x,i)=f(\gamma x,\phi_n[\gamma](i)).\]
	In turn, such functions can be viewed as functions $s:\widetilde M\to \mathbb C^n$ satisfying the equivariance property $s(\gamma x)=\pi_n(\gamma)s(x)$, where $\pi_n(\gamma)$ is just the permutation matrix associated with $\phi_n[\gamma]$. One defines a flat bundle over $M$ by 
	\[E_{\pi_n}:=\Gamma\backslash (S\widetilde M\times \mathbb C^n), \qquad \gamma.(x,v)=(\gamma x,\pi_n(\gamma)v).\]
	Then we have an identification $C^\infty(M_n)\simeq C^\infty(M,E_{\pi_n})$. Let $(\operatorname{std}_{n-1},V_n^0)$ be the standard irreducible $(n-1)$-dimensional representation of $S_n$, which is the restriction of the natural action to the space of vectors with zero average:
	\[V_n^0=\Big\{(x_1,\ldots,x_n)\in \mathbb C^n~:~\sum x_i=0\Big\}.\]
	The bundle $E_{\pi_n}$ has rank $n$; by removing the trivial part corresponding to functions in $C^\infty(M_n)$ that are constant over the fibers, one obtains a subbundle $E_{\rho_n}$ of rank $n-1$ corresponding to the representation of $\Gamma$ defined by \begin{equation} \label{eq: def rho_n}\rho_n:=\operatorname{std}_{n-1}\circ \phi_n.\end{equation}
	In other words, the subspace $C^\infty_{\rm new}(M_n)\subset C^\infty(M_n)$ identifies with the subspace $C^\infty(M,E_{\rho_n})\subset C^\infty(M,E_{\pi_n})$.
	
	Letting $\pi:SM\to M$ denote the projection, we can lift $E_{\rho_n}$ to a bundle $\mathcal E_{\rho_n}:=\pi^*E_{\rho_n}$ over $SM$. Similarly, we have an identification
	\[C^\infty_{\rm new}(SM_n)\simeq C^\infty(SM,\mathcal E_{\rho_n}).\]

	
	If $\phi_n$ is picked uniformly at random in $\operatorname{Hom}(\Gamma,S_n)$, we obtain a random representation $(\rho_n,V_n^0)$ defined by \eqref{eq: def rho_n}.
	We denote by $\mathbf X_{\mathcal E_{\rho_n}}$ the Lie derivative of $X$ acting on smooth sections of $\mathcal E_{\rho_n}$, and let $\operatorname{Res}(\mathbf X_{\mathcal E_{\rho_n}})$ be the set of Pollicott--Ruelle resonances of the operator $\mathbf X_{\mathcal E_{\rho_n}}$, which are defined as the poles of the meromorphic extension
	\[(\mathbf X_{\mathcal E_{\rho_n}}+z)^{-1}:C^\infty(SM,\mathcal E_{\rho_n})\to \mathcal D'(SM,\mathcal E_{\rho_n}).\]
	Then, one has
	\begin{equation}\operatorname{Res}(X_{M_n})=\operatorname{Res}(X_M)\cup \operatorname{Res}(\mathbf X_{\mathcal E_{\rho_n}}), \end{equation}
	where the union is taken with multiplicities. In this reformulation, one can forget that we are working on coverings $M_n\to M$ and instead work with unitary flat bundles over $M$. In particular, we can work with flat bundles associated with unitary representations of $\Gamma$ that need not be permutation representations of the form \eqref{eq: def rho_n}. \medskip
	
	 We recall the definition of strong convergence of representations.
	  \begin{defi} Let $(\lambda_{\Gamma},\ell^2(\Gamma))$ denote the regular representation of $\Gamma$. A sequence of unitary representations $(\rho_n,V_n)_{n\ge 0}$ of $\Gamma$ \emph{strongly converges} to $(\lambda_{\Gamma},\ell^2(\Gamma))$ if the following holds. For any $w\in\mathbb C[\Gamma]$, for any $\varepsilon>0$, we have for $n$ large enough (depending on $w,\varepsilon$):
	\begin{equation} \label{eq: def rhon z str c}\big| \|\rho_n(w)\|_{V_n}- \|\lambda_{\Gamma}(w)\|_{\ell^2(\Gamma)}\big| <\varepsilon ,\end{equation}
	 where the norms are operator norms. Here we view $w\in \mathbb C[\Gamma]$ as a linear combination of group elements and $\rho_n(w)$ acts on $V_n$ by extending $\rho_n$ by linearity. A sequence of random unitary representations $(\rho_n,V_n)$ is said to strongly converges \emph{asymptotically almost surely} to $(\lambda_\Gamma,\ell^2(\Gamma))$ if for any $w\in \mathbb C[\Gamma]$ and $\varepsilon>0$, \eqref{eq: def rhon z str c} holds a.a.s. in the sense of Definition \ref{defi: a.a.s}. \end{defi}
	 
	 \begin{rem}There is also a notion of \emph{weak convergence}, which replaces convergence of operator norms in \eqref{eq: def rhon z str c} by convergence of the normalized traces. In our setting, strong convergence implies weak convergence \cite[Lemma 6.1]{louder2022strongly}. For a sequence of unitary representations $(\rho_n)$ associated with a sequence of Riemannian coverings $(M_n\to M)$, the weak convergence of $(\rho_n)$ is equivalent to the convergence in the sense of Benjamini--Schramm of $(M_n)$ to the universal cover $\widetilde M$. 
	 \end{rem}
	 
	 The link with the model of random covers is given by a recent result of Magee, Puder and van Handel \cite{magee2025strongconvergenceuniformlyrandom}, establishing a.a.s. strong convergence for uniformly random permutation representations of surface groups.
	
	\begin{thm}[Magee--Puder--van Handel]\label{thm: mpvh strong con} Let $\Gamma$ be a surface group (in genus $\ge 2$). Then, if $\phi_n\in \operatorname{Hom}(\Gamma,S_n)$ is picked uniformly at random, the sequence of random representations $(\rho_n,V_n)$ defined by \eqref{eq: def rho_n} strongly converges a.a.s. to $(\lambda_\Gamma,\ell^2(\Gamma))$. \end{thm}
	Thus, Theorem~\ref{thm: spectral gap} is a special case of the following result.
	\begin{thm}\label{thm: extension resolvent}Let $(M,g)$ be a closed Anosov manifold of dimension $\ge 2$, with fundamental group $\Gamma$. Let $\{(\rho_n,V_n)\}_{n=1}^{+\infty}$ be a sequence of random unitary representations of $\Gamma$. Assume that $(\rho_n,V_n)$ strongly converges to $(\lambda_\Gamma,\ell^2(\Gamma))$ a.a.s. Then, for any compact subset $\mathcal K\subset \{\operatorname{Re}(z)>\delta_0\}$, a.a.s., the resolvent 
		\[(\mathbf X_{\mathcal E_{\rho_n}}+z)^{-1}:C^\infty(\mathcal M,\mathcal E_{\rho_n})\to \mathcal D'(\mathcal M,\mathcal E_{\rho_n})\]
		is holomorphic for $z\in \mathcal K$.\end{thm}
	\begin{rem}Theorem~\ref{thm: extension resolvent} includes the case of deterministic sequences of representations $\{\rho_n\}_{n=1}^{+\infty}$ which strongly converge to the regular representation, by considering probability measures supported on a single point.
	\end{rem}
	
	Theorem \ref{thm: spectral gap} is obtained by combining Theorem \ref{thm: mpvh strong con} with Theorem~\ref{thm: extension resolvent}. It was shown by Magee and Thomas \cite{magee2023stronglyconvergentunitaryrepresentations} that the fundamental group of any hyperbolic $3$-manifold admits at least one strongly converging sequence of unitary representations (which need not be permutation representations), which gives another application of Theorem \ref{thm: extension resolvent}. 
	
	No examples of strongly convergent sequences of unitary representations of $\Gamma$ are known in higher dimension. We chose to state the result in its present form because in the future strong convergence may be obtained for new sequences of random or deterministic unitary representations of $\Gamma$. Moreover, we believe that the results of \S\ref{sec: spherical mean} about the spherical mean operator on $\widetilde M$ are interesting in themselves, thus deserve to be treated in any dimension.

	\subsection{High frequency spectral gap} 	Since the value of $\gamma_0$ in Theorem \ref{thm: tsujii essential gap} is independent of the covering $M_n\to M$, Theorem~\ref{thm: tsujii essential gap} tells us that for any cover $M_n\to M$, there are finitely many resonances in the half-space $\{\operatorname{Re}(z)>-\gamma_0/2+\varepsilon\}$. However Theorem~\ref{thm: tsujii essential gap} does not give the localization of these resonances. In particular, if $M_n\to M$ is a sequence of coverings, it does not rule out the existence of a sequence of resonances $z_n\in \operatorname{Res}(X_{M_n})$ such that $\operatorname{Re}(z_n)>-\gamma_0+\varepsilon$ and $|\operatorname{Im}(z_n)|\to +\infty$. In view of Theorem~\ref{thm: tsujii essential gap} and Theorem~\ref{thm: spectral gap}, it is still natural to formulate the following conjecture.
	\begin{conj}\label{conj: uniform gap} Let $M$ be a closed Anosov surface. For any $\varepsilon>0$, with probability tending to $1$ as $n\to +\infty$, a uniformly random cover $M_n\to M$ has no new resonance in the half space $\{\operatorname{Re}(z)>-\frac{\gamma_0}{2}+\varepsilon\}$.
	\end{conj}
	
	Conjecture~\ref{conj: uniform gap} would follow from the following \emph{deterministic} result (combined with Theorem~\ref{thm: spectral gap}).
	\begin{conj}\label{conj: uniform asymptotic gap} For any $\varepsilon>0$, there is a constant $C_\varepsilon$ such that for any unitary flat bundle $\mathcal E \to SM$ of finite rank, the Ruelle operator $\mathbf X_{\mathcal E}$ has no resonance in 
		\[\mathcal S_\varepsilon:=\Big\{z\in\mathbb{C}, \ |\operatorname{Im}(z)|\ge C_\varepsilon \ \text{and} \ \operatorname{Re}(z)\ge -\frac{\gamma_0}{2}+\varepsilon\Big\}.\]\end{conj}
		\begin{rem}Here we do not assume that $\mathcal E$ is the lift of a bundle over $M$.
		\end{rem}
	Conjecture~\ref{conj: uniform asymptotic gap} will be investigated in future work. It should follow by running again the proof of \cite[Theorem 4]{nonnenmacher2015decay}. As it is already hinted in Appendix~\ref{appendix}, when quantizing scalar functions $a\in C^\infty(T^*\mathcal M)$ into operators $\operatorname{Op}_h^{\mathcal E}(a)$ acting on $C^\infty(\mathcal M,\mathcal E)$, we obtain a calculus with errors that are uniform with respect to the bundle, essentially because the transition maps of the bundle are constant unitary matrices. As a motivation, we point out that an analogue of Conjecture~\ref{conj: uniform asymptotic gap} for resonances of the Laplacian on Schottky surfaces was proved by Magee and Naud (with a nonexplicit uniform gap), see \cite[Theorem 1.10]{magee2020explicit}.
	\begin{rem}A stronger version of Conjecture~\ref{conj: uniform asymptotic gap} should be expected, namely that the resonance free region comes with a uniform resolvent bound with respect to $\mathcal E$, \emph{i.e.} that there is a small $h_0>0$ such that for any unitary flat bundle $\mathcal E\to \mathcal M$:
		\[ \forall z\in \mathcal S_\varepsilon,\qquad \|(\mathbf X_{\mathcal E}+z)^{-1}\|_{\mathcal H_{h_0}^s(\mathcal M,\mathcal E)}\le C_\varepsilon(1+|\operatorname{Im}(z)|^{N_0}).\]
		Here $\mathcal H_h^s(\mathcal M,\mathcal E)$ are the anisotropic Sobolev spaces introduced by Faure--Sj\"ostrand, see \S \ref{sec: aniso}. When $M$ is a closed Anosov surface, this conjecture, combined with Theorem~\ref{thm: extension resolvent}, would imply that there is a uniform constant $\kappa>0$ and an integer $k\ge 0$ such that a.a.s., one has for any $f,g\in H^k(SM_n)$,
		\begin{equation} \label{eq: decay cor strong conv} \int_{SM_n}(f\circ \varphi_t)g\mathrm{dvol}=\int_{SM_n} f\mathrm{dvol}\int_{SM_n}g\mathrm{d}\mu_{\rm L} +\mathcal O(\mathrm{e}^{-\kappa t}\|f\|_{H^k(SM_n)}\|g\|_{H^k(SM_n)}), \end{equation}
		the implied constant in the $\mathcal O(\cdot)$ being independent of $n$. This result might be a little surprising at first sight. Indeed, it is known that uniformly random covers $M_n\to M$ converge in the sense of Benjamini--Schramm to the universal cover $\widetilde M$. In particular, one can find balls of radius $R_n\to +\infty$ in $M_n$ that are diffeomorphic to balls of the same radius in $\widetilde M$, suggesting that no mixing should occur in times $t\le R_n$. The point lies in the fact that correlations measure not only mixing but also dispersion. Mixing occurs when an initial cloud of points of mass $\sim 1$ is uniformly spread over the whole manifold. To measure this, take smooth functions $f,g$ supported in balls of radius $1$ and normalized so that $\int_{SM_n} f\mathrm{dvol}=\int_{SM_n} g\mathrm{dvol}=1$. We may even assume that $f,g$ have uniformly bounded $H^k$ norms. Then, the first term on the right-hand side of \eqref{eq: decay cor strong conv} is of order $\sim n^{-1}$, because
		\[\int_{SM_n} g\mathrm{d}\mu_{\rm L}=\frac{1}{\operatorname{vol}(SM_n)}\int_{SM_n} f\mathrm{dvol}\sim \frac 1n, \]
		while the error is of order $\mathcal O(\mathrm{e}^{-\kappa t})$. Mixing emerges when the error is negligible compared to the leading term, i.e. when
		\[t\gtrsim \frac{\log n}{\kappa}\]
		Thus, we have to reach times of order $\sim \log n$ to observe uniform spreading of a cloud of point of initial size $\sim 1$. This reflects the fact that the diameter of a uniformly random cover $M_n\to M$ grows like $\sim \log n$, a fact that can be deduced from \eqref{eq: decay cor strong conv} but can also be proved using the Cheeger constant to obtain an upper bound on the diameter from a lower bound on the first nontrivial eigenvalue of the Laplacian---which is known by \cite{magee2025strongconvergenceuniformlyrandom,hide2025spectralgap}, see \cite{MageeLetter} for example.
		
		For times $t\lesssim \log n$, the dynamics in regions where the injectivity radius is large can be understood by lifting it to the universal cover. A cloud of initial mass $1$ supported in a ball of radius $1$ spreads over a region of diameter $\sim t$, and puts a mass $\lesssim \mathrm{e}^{-\gamma_0 t}$ in any ball of radius $1$. Integrating against a $H^k$ function then produces a small contribution. This is a regime of \emph{dispersion} but not full mixing.
		\end{rem}

	\subsection{The case of hyperbolic surfaces} \label{subsec: case of hyperbolic surfaces} In the case where $M$ has constant negative curvature $-1$, we know that Conjecture~\ref{conj: uniform gap} holds true. This is due to the correspondence between eigenvalues of the Laplacian and Pollicott--Ruelle resonances, and follows from a recent breakthrough of Magee--Puder--van Handel \cite{magee2025strongconvergenceuniformlyrandom}. Indeed, assume that $M$ is a closed connected hyperbolic surface. Then the Laplace--Beltrami operator $\Delta_M$ has discrete spectrum
	\[\operatorname{Spec}(\Delta_M)=\{0=\lambda_0< \lambda_1\le \ldots \le \lambda_n\to +\infty\}.\]
	By \emph{e.g.} \cite[Theorem 1]{dyatlov2015power}, the Pollicott--Ruelle resonances of $SM$ with real part $>-1$ are exactly given by
	\[\mu_j=-\frac 12\pm \mathrm{i}\sqrt{\lambda_j-\frac 14}\in (-1,0]\cup \Big(-\frac 12+\mathrm{i} \mathbb R\Big),\]
moreover, the multiplicity of $\mu_j$ as a resonance is the multiplicity of $-\mu_j(\mu_j+1)$ as an eigenvalue of the Laplacian. In particular, resonances of real part $>-\frac 12$ correspond to eigenvalues of the Laplacian $\lambda_j\in [0,\frac 14)$. The result of \cite[Theorem 1.2]{magee2025strongconvergenceuniformlyrandom} shows that these small eigenvalues are stable under random Riemannian coverings:
	\begin{thm}[Magee--Puder--van Handel]\label{thm: magee puder van handel} Let $M$ be a closed hyperbolic surface. For any $\varepsilon>0$, asymptotically almost surely, a uniformly random covering $M_n\to M$ of degree $n$ has no new eigenvalue smaller than $\frac 14-\varepsilon$, i.e.
		\[\operatorname{Spec}(\Delta_{M_n})\cap \Big[0,\frac 14-\varepsilon\Big]=\operatorname{Spec}(\Delta_M)\cap \Big[0,\frac 14-\varepsilon\Big],\]
		where the multiplicities coincide on both sides.
	\end{thm}
	An immediate consequence of Theorem~\ref{thm: magee puder van handel}, due to the correspondence described above, is the following
	\begin{cor} \label{cor: thm for hyperbolic surface}  Let $M$ be a closed hyperbolic surface. For any $\varepsilon>0$, with probability tending to $1$ as $n\to +\infty$, a uniformly random covering of degree $n$ of $M$ has no new Pollicott--Ruelle resonance in the half space $\{\operatorname{Re}(z)\ge -\frac 12+\varepsilon\}$. \end{cor}
	Note that the value $\frac 14$ in Theorem~\ref{thm: magee puder van handel} is optimal. Indeed, by adapting the proof of \cite[Theorem 9.2]{LS1}, one can show that for all $\varepsilon>0$, there is a constant $c_\varepsilon>0$ such that a.a.s. we have
	\[\#\operatorname{Spec}(\Delta_{M_n}) \cap \Big[\frac 14,\frac 14+\varepsilon\Big]\ge c_\varepsilon n.\]
	Similarly, the $-\frac 12$ result in Corollary~\ref{cor: thm for hyperbolic surface} is optimal.
	
	\medskip
	
	While Theorem~\ref{thm: magee puder van handel} extends to variable curvature surfaces---one has to replace $\frac 14$ by the bottom of the spectrum of the universal cover of $M$, see \cite{hide2025spectralgap}---, the correspondence between Laplacian eigenvalues and Pollicott--Ruelle resonances breaks down, thus one has to tackle the resonance problem directly.
	
	\subsection{Related works} \label{subsec: related work}
	
	\subsubsection{Eigenvalues of random lifts of graphs} Consider a finite graph $G=(V,\vec E)$ with vertices $V$ and directed edges $\vec E$. We always assume that $(y,x)\in \vec E$ for any $(x,y)\in \vec E$. Denote by $A_G$ and $B_G$ the adjacency and nonbacktracking operators, that act respectively on $L^2(G)$ and $L^2(\vec E)$ by the formulae:
	\[A_G \varphi(x)=\sum_{(x,y)\in \vec E} \varphi(y), \qquad B_G \psi(f)=\sum_{e\to f} \psi(e),\]
	where for $e=(x,y)$ and $f=(x',y')\in \vec E$ we write $e\to f$ if $y=x'$ and $y'\neq x$. The following result was obtained by Friedman and Kohler \cite{friedman2014relativizedsecondeigenvalueconjecture}, by refining the method of \cite{Friedman2008}.
	\begin{thm}[Friedman--Kohler] \label{thm: Friedman Kohler} For a uniformly random cover $G_n\to G$ of a $d$-regular graph $G$, one has a.a.s.
		\[\operatorname{Spec}(A_{G_n})\cap [\sqrt{d-1}+\varepsilon, +\infty)=\operatorname{Spec}(A_G)\cap [\sqrt{d-1}+\varepsilon, +\infty),\]
		where the multiplicities coincide on both sides.\end{thm}Here, $\sqrt{d-1}$ is precisely the spectral radius of the adjacency operator on the universal cover of $G$, which is the $d$-regular tree. Bordenave \cite{bordenave2015} then obtained the following result, that holds for nonregular graphs.
	\begin{thm}[{\cite[Theorem 23]{bordenave2015}}] \label{thm: bordenave} Let $G=(V,\vec E)$ be a finite graph, with universal cover $\widetilde G$. For a uniformly random cover $G_n=(V_n,\vec E_n)\to G$, letting $H_0\subset L^2(\vec E_n)$ denote the space of functions with $0$-average over the fibers of the projection $\vec E_n\to \vec E$, one has a.a.s.
		\[\rho(B_{G_n}|_{H_0})\le \rho(B_{\widetilde G})+\varepsilon,\]
		where $\rho(B)$ denotes the spectral radius of $B$. \end{thm}
	If $G$ is $d$-regular, Theorem~\ref{thm: bordenave} implies Friedman and Kohler's result thanks to the explicit link between the adjacency and nonbacktracking spectra provided by the Ihara--Bass formula in this case. However, this left open the analogue of Theorem~\ref{thm: Friedman Kohler} for nonregular graphs. This was finally settled in a breakthrough of Bordenave--Collins \cite{bordenave2019eigenvalues}. The key technical tool	was to establish \emph{strong convergence} for uniformly random permutation representations of a free group $\mathbb F_r$. A short and elementary new proof of this result was recently obtained by Chen--Garza-Vargas--Tropp--van Handel \cite{Ch.Ga.Tr.va2024}.
	
	\subsubsection{Eigenvalues of the Laplacian on random negatively curved covers} A first spectral gap result was obtained by Magee and Naud for resonances of the Laplacian on Schottky surfaces \cite{magee2020explicit}. For closed hyperbolic surfaces, Magee--Naud--Puder \cite{MNP} obtained a spectral gap $\frac{3}{16}-\varepsilon$ using the Selberg trace formula. 
	
	It was then observed by Hide--Magee \cite{HM} that the strong convergence property could be used to obtain nearly optimal spectral gaps for the Laplacian on random covers of noncompact hyperbolic surfaces. Calder\'on, Magee and Naud  \cite{MCN} went beyond $L^2$-eigenvalues and obtained a low frequency spectral gap for Laplacian resonances on random covers of Schottky surfaces, which is analogous to our Theorem \ref{thm: spectral gap}. Strong convergence techniques made it possible to bypass the use of trace formulae to study directly operator norms. Although the analysis is more intricate on noncompact surfaces, the freeness of their fundamental groups allowed to obtain better results than in the closed case. This limitation was recently overcome by Magee, Puder and van Handel, who established an analogue of the Bordenave--Collins result for surface groups \cite{magee2025strongconvergenceuniformlyrandom} (see  \S\ref{subsec: strong conv} for details), which implies a spectral gap $(\frac 14-\varepsilon)$ for the Laplacian in the closed case. Building on this result, Hide--Macera--Thomas \cite{hide2025spectralgappolynomialrate} obtained a quantitative polynomial improvement $\frac 14-\mathcal O(n^{-b})$ for some positive constant $b$, where $n$ is the degree of the covering.
	
	 It would be interesting to obtain a similar improvement of Theorem \ref{thm: spectral gap} by allowing the compact $\mathcal K$ to approach the line $\operatorname{Re}(z)=\delta_0$ at distance $\varepsilon=\varepsilon(n)$ when $n\to \infty$. There are additional difficulties compared to the case of the Laplacian on hyperbolic surfaces, essentially because we lack explicit formulas and have to consider long time dynamics. By adapting the methods in this paper and using ideas from \cite{hide2025spectralgappolynomialrate}, we believe that one could take $\varepsilon(n)=\frac{c}{\log n}$ for some small constant $c$. Getting a polynomial rate of decay would certainly require new ideas.\medskip
	
	Models of random covers of hyperbolic surfaces extend without modifications to the variable negative curvature setting. Nearly optimal spectral gaps for the Laplacian were obtained in \cite{hide2025spectralgap} in the closed case and in \cite{moy2025spectralgaprandomcovers,ballmann2025spectralstabilityfinitecoverings} in the noncompact case.
	
	Just as there is a simple correspondence between the adjacency and nonbacktracking spectra for regular graphs, a simple correspondence between the Laplace and Pollicott--Ruelle spectra holds for hyperbolic surfaces. These correspondences break down for nonregular graphs and surfaces of variable curvature. To this regard, our results should rather be thought of as analogues of Bordenave's result \cite{bordenave2015}. More precisely, Proposition~\ref{prop: Mh etX Mh}, which is the main probabilistic input in the proof of Theorem~\ref{thm: spectral gap}, may be viewed as a manifold analogue of Theorem~\ref{thm: bordenave}.
	
	\subsubsection{Other models of random surfaces} For the Weil--Petersson model of random hyperbolic surfaces of genus $g$, a first probabilistic spectral gap $\lambda_1(X_g)>0.002$ was obtained by Mirzakhani \cite{Mi2013}. A gap $\lambda_1(X_g)\ge \frac{3}{16}-\varepsilon$ was obtained independently by Wu--Xue \cite{WX} and Lipnowski--Wright \cite{LW}. The optimal gap $\frac 14-\varepsilon$ was finally obtained by Anantharaman and Monk in a series of work \cite{AnMo1,AnMo2,AnMo3}. Very recently, Hide, Macera and Thomas \cite{hideWeilPol} obtained the quantitative lower bound $\lambda_1(X_g)\ge \frac 14-\mathcal O(g^{-b})$ for some $b>0$, by adapting the methods of \cite{magee2025strongconvergenceuniformlyrandom} to the Weil--Petersson model.
	
	For the Brooks--Makover model of random surfaces obtained by gluing $2n$ ideal hyperbolic triangles, a first nonexplicit, probabilistic spectral gap $\lambda_1(X_n)\ge c>0$ was shown in the original article \cite{BrooksMakover}. No further improvement was known until the recent work of Shen and Wu \cite{shen2025nearlyoptimalspectralgaps}, who proved the optimal lower bound $\lambda_1(X_n)\ge \frac 14-\varepsilon$ using methods inspired by \cite{magee2025strongconvergenceuniformlyrandom}.
		
	\subsection{Overview of the proof} 
	
		To prove Theorem~\ref{thm: extension resolvent}, we use the \emph{anisotropic Sobolev spaces} of distributions introduced by Faure--Sj\"ostrand \cite{Faure2011}. For any flat bundle $E\to \mathcal M$, there is a scale of Hilbert spaces $\mathcal H_h^s(\mathcal M,\mathcal E)$ depending on a semiclassical parameter $h$ and a parameter $s\in \mathbb R$, such that
	\[C^\infty(\mathcal M,\mathcal E)\subset \mathcal H_h^s(\mathcal M,\mathcal E)\subset \mathcal D'(\mathcal M,\mathcal E),\]
	and the inclusions are dense and continuous. These spaces are defined in  \S\ref{sec: aniso} using semiclassical quantization on flat bundles, as presented in the Appendix, and the escape function from \cite{Faure2011}. The advantage of these spaces is that the propagator $\mathrm{e}^{-t\mathbf X_{\mathcal E}}$ enjoys better properties when acting on $\mathcal H_h^s(\mathcal M,E)$ than on $L^2(\mathcal M,\mathcal E)$.
	
	\begin{rem} Even though we deal with low-frequency resonances, it is practical to work with semiclassical quantization, because our goal is to produce an exact inverse of the operator $\mathbf X_{\mathcal E}+z$ ---and not only an inverse up to a compact remainder, by using a parametrix construction. Remainder terms in semiclassical calculus are usually small in norm (say $\mathcal O(h)$), as opposed to being only of lower order, in the standard pseudodifferential calculus. \end{rem}
	
	Let $\chi$ be a smooth function, compactly supported on $[-1,2]$, such that $\chi\equiv 1$ on $[0,1]$. Motivated by the expression \eqref{eq: def resolvent}, we introduce an approximate resolvent
\[\mathbf Q(z):=\int_{0}^{+\infty}\chi(t/T)\mathrm{e}^{-zt}\mathrm{e}^{-t\mathbf X_{\mathcal E}}\mathrm{d}t.\]
Then, one has
	\[\mathbf Q(z)(\mathbf X_{\mathcal E}+z)=(\mathbf X_{\mathcal E}+z)\mathbf Q(z)=\operatorname{Id}+\mathbf R(z),\]
	where the remainder is given by
\[\mathbf R(z)=\frac 1T\int_0^{+\infty} \chi'(t/T) \mathrm{e}^{-zt}\mathrm{e}^{-t\mathbf X_{\mathcal E}}\mathrm{d}t.\]
The interest is that $\mathbf R(z)$ only involves propagation for large times $t\in [T,2T]$. If one can show that $\|\mathbf R(z)\|_{\mathcal H_h^s}<\frac 12$, then a Neumann series argument leads to invertibility of $(\mathbf X_{\mathcal E}+z)$.

	Showing $\|\mathbf R(z)\|_{\mathcal H_h^s}\le \frac 12$ requires understanding the action of the propagator $\mathrm{e}^{-t\mathbf X_{\mathcal E}}$ on anisotropic Sobolev spaces $\mathcal H_h^s(\mathcal M,\mathcal E)$. To this end, we introduce a smooth partition $1=a_1+a_2+a_3$ of the phase space $T^*\mathcal M$, which gives rise to pseudodifferential operators $\operatorname{Op}_h^{\mathcal E}(a_i)$ for $i\in \{1,2,3\}$ acting on sections of $\mathcal E$.

	\subsubsection*{Step 1. Probabilistic low frequency estimates} The function $a_1$ is supported on $\{(x,\xi)\in T^*\mathcal M~:~ |\xi|\le 2\}$. We show that if a sequence of random representations $(\rho_n,V_n)$ strongly converges a.a.s, then for any $\varepsilon>0$ we have a.a.s.
	\begin{equation}\label{eq: step 2 low freq}
		\big\|\operatorname{Op}_h^{\mathcal E_{\rho_n}}(a_1) \mathrm{e}^{-t\mathbf X_{\mathcal E_{\rho_n}}}\operatorname{Op}_h^{\mathcal E_{\rho_n}}(a_1)\big\|_{L^2(\mathcal M,\mathcal E_{\rho_n})}\le C_\varepsilon h^{-(d-1)}\mathrm{e}^{(\delta_0+\varepsilon)t},\end{equation}
	where $\delta_0=\operatorname{Pr}(-2\psi^u)/2$ and $d$ is the dimension of $M$. This estimate is only relevant when $t\ge K|\log h|$ for some sufficiently large constant $K$. To show \eqref{eq: step 2 low freq}, we use strong convergence techniques to relate the operator norm in \eqref{eq: step 2 low freq} to that of a limiting operator living on the unit tangent bundle of the universal cover of $M$, known as the \emph{spherical mean operator}.

	\subsubsection*{Step 2. Estimates in the direction of the flow} The function $a_2$ is supported in a small conical neighborhood of the trapped set $E_0^*$ (see \S \ref{sec: aniso} for the definition) which avoids low energies, that is
	\[\operatorname{Supp} a_2\subset \Big\{(x,\xi)\in T^*\mathcal M~:~ \xi(X)\ge \frac{|\xi|}{2}\ge \frac 12\Big\}.\] 
	The contributions from $\operatorname{Supp}(a_2)$ are damped by integration along the flow direction. It essentially relies on the ellipticity of the semiclassical operator $-\mathrm{i}h\mathbf X_{\mathcal E}$ on $\operatorname{Supp}(a_2)$. This is done in  \S\ref{sec: averaging along the flows}.

	\subsubsection*{Step 3. Estimates away from the trapped set} 
	Let $a_3$ be a smooth function supported in a conic subset disjoint of $E_0^*$, on energies $|\xi|\ge 1$. We show that there is an absolute constant $C>0$ such that for any $t_0>0$ fixed, there is a constant $C_0>0$ depending on $t_0$ such that for any flat bundle $\mathcal E\to \mathcal M$, we have	
	\begin{equation} \label{eq: away trapped set intro} \big\|\mathrm{e}^{-t_0\mathbf X_{\mathcal E}}\operatorname{Op}_h^{\mathcal E}(a_3)\big\|_{\mathcal H_h^s}\le C\mathrm{e}^{-C_0 s}, \quad \big\|\operatorname{Op}_h^{\mathcal E}(a_3)\mathrm{e}^{-t_0\mathbf X_{\mathcal E}}\big\|_{\mathcal H_h^s}\le C\mathrm{e}^{-C_0 s}.\end{equation}
	These estimates follow from an Egorov theorem with uniform estimates in $\mathcal E$, and rely on the properties of the escape function from \cite{Faure2011}. 
	
	\subsubsection*{Step 4. Gathering all contributions} We combine the estimates \eqref{eq: step 2 low freq}, 
	\eqref{eq: away trapped set intro}. We first take $\varepsilon>0$ small and $s$ large enough, then $T=K|\log h|$ with $K=K(s,\varepsilon,\mathcal K)$ large enough, then take $h=h(\varepsilon,K,s,\mathcal K)$ small enough, to show that $\|\mathbf R(z)\|_{\mathcal H_h^s}<\frac 12$ for all $z\in \mathcal K$ a.a.s.
	
	\begin{rem}[Relation with high-frequency spectral gap] We work in a regime where $z$ remains on a compact set, and obtain estimates on $\|\mathbf R(z)\|_{\mathcal H_h^s}$ as follows. The low frequency contributions are controlled by probabilistic methods, this is step 1. The high frequency contributions away from the trapped set are controlled in step 3 by the use of the anisotropic spaces, while the contribution from the trapped set $E_0^*$ is tamed by the integration along the flow direction (this is step 2, here it is crucial that $h|z|\ll 1$).
		
		In contrast, to obtain an asymptotic spectral gap, \cite{nonnenmacher2015decay} considers the regime $h|z|\sim 1$. In that setting, the low frequency contributions are damped (deterministically) by the integration along the flow. However, the analysis near the trapped set becomes much more delicate and represents the main difficulty in \cite{nonnenmacher2015decay}. Contributions away from the trapped set, although they become more technical, are controlled as in the present work, with the use of more sophisticated escape functions.
	\end{rem}
	
	\subsection*{Notations}  We write $f\lesssim g$ or $f=\mathcal O(g)$ if there is a uniform constant $C$ such that $f\le Cg$. If $C$ depends on some relevant set of parameters $\bullet$ we may write $f\lesssim_\bullet g$ or $f=\mathcal O_{\bullet}(g)$. We write $f\asymp g$ when $f\lesssim g$ and $g\lesssim f$, i.e. when there is a constant $C>0$ such that $C^{-1}f\le g\le Cf$.
	
	 The norm of a continuous operator $A:H\to H'$ is denoted indifferently $\|A\|_{H\to H'}$, or simply $\|A\|_H$ when $H=H'$. The notation $\mathcal O_{\bullet}(g)_{H\to H'}$ is used to denote a continuous operator $R:H\to H'$ with norm $\|R\|_{H\to H'}=\mathcal O_{\bullet}(g)$.

	If $\chi,\chi'$ are two compactly supported functions, we write $\chi\prec \chi'$ if $\operatorname{Supp} \chi\subset \{\chi'\equiv 1\}$.
	
 We let $d(\cdot,\cdot)$ denote the Riemannian distance. The ball and sphere of radius $r$ about $x$ are denoted by $B(x,r)$ and $S(x,r)$, respectively. Also, we let $A(x,r,c)$ denote the annulus $B(x,r+c)\backslash B(x,r-c)$.
 
 	Flat bundles over $M$ are systematically denoted with the letter $E$ while flat bundles over $SM$ are denoted with the letter $\mathcal E$.

	\subsection*{Acknowledgments} I thank Michael Magee for suggesting this problem. I am grateful to Yann Chaubet for his explanations on Jacobi fields, and to Fr\'ed\'eric Paulin for bringing my attention to the paper \cite{RuggieroGH}. Thanks also to Fr\'ed\'eric Faure and Fr\'ed\'eric Naud for stimulating discussions. Lastly, I thank St\'ephane Nonnenmacher for his guidance during the writing of this paper and his careful feedback on this manuscript.

	\section{Geometric preliminaries} \label{sec: geometry} Let $(M,g)$ be a closed Riemannian manifold. We recall some basic facts about the geometry of $SM$ and some consequences of the Anosov property. An important fact is that the universal cover $\widetilde M\to M$ of a closed Anosov manifold is a \emph{Gromov-hyperbolic} space. This property is used extensively to control the operator norm of the spherical mean operator introduced in the next section.
	
	The material for the next two subsections can be found \cite[\S 13.1]{Lefeuvre}.
		\subsection{Geometry of $SM$} Let $(M,g)$ be a closed manifold of dimension $d$, equipped with the Riemannian volume form. We denote by $\pi:\mathcal M\to M$ the unit tangent bundle of $M$. The geodesic vector field is denoted by $X$ and we write $\mathbb V:=\operatorname{ker}\mathrm{d}\pi$ for the vertical subbundle. The horizontal subbundle $\mathbb H$ is defined as follows. For $(x,v)\in \mathcal M$, let
	\[v^\bot=\{w\in T_xM~:~ g_x(v,w)=0\}.\]
	Given $w\in v^\bot$, let $\gamma_{(x,w)}(t)=\pi(\varphi_t(x,w))$ be the unit-speed geodesic starting at $x$ with speed $w$, and let $\tau_{x,w}(t)$ denote the parallel transport map $T_xM\to T_{\gamma_{x,w}(t)}M$. We set
	\[\mathbb H(x,v)=\Big\{\partial_t \big(\gamma_{x,w}(t),\tau_{x,w}(t)v\big)_{|t=0} ~:~ w\in v^\bot \Big\}.\]
	Then, one has a smooth splitting
	\[T\mathcal M=\mathbb RX\oplus \mathbb H\oplus \mathbb V.\]
	
	\begin{lem}For any $(x,v)\in SM$,
		\[\mathrm{d}\pi \cdot X(x,v)=v, \qquad \mathrm{d}\pi\cdot \mathbb H(x,v)=v^\bot, \qquad \mathrm{d}\pi\cdot \mathbb V(x,v)=\{0\},\]
	\end{lem}

	\subsubsection{Sasaki metric and Liouville measure} 
	There is a unique metric known as the Sasaki metric, denoted by $G$, such that $\mathbb RX\oplus \mathbb H\oplus \mathbb V$ is an orthogonal decomposition of $T\mathcal M$ at any point, and such that
	\begin{itemize}
 \item $\mathrm{d}\pi:\mathbb RX(x,v)\oplus \mathbb H(x,v)\to T_xM$ is an isometry.
 \item The natural identification $\mathbb V(x,v)\simeq v^\top\subset T_xM$ is an isometry.
\end{itemize}
	For each $x\in M$, the metric $g_x$ induces a volume form on the vector space $T_xM$, which can be contracted by the radial vector field to give a $(d-1)$ form, whose restriction to $S_xM$ yields a measure $\mu_x$. Note that $\mu_x(S_xM)=\operatorname{Vol}(S^{d-1})$. Then, one has (see \cite[Lemma 1.30]{GuillarmouBook}):
	\begin{lem}[Disintegration of the Sasaki measure]
		For any $F(x,v)\in C^\infty_{\rm comp}(SM)$,
		\[\int_{\mathcal M}F(x,v)\mathrm{d}\mu_G(x,v)=\int_M\Big(\int_{S_xM} F(x,v)\mathrm{d}\mu_x(v)\Big)\mathrm{d}\mu_g(x).\]
	\end{lem}
	The normalized Sasaki measure is called the \emph{Liouville measure}, and denoted $\mu_{\rm L}$. The Liouville measure is invariant under the geodesic flow, i.e. $(\varphi_t)_*\mu_{\rm L}=\mu_{\rm L}$ for any $t\in \mathbb{R}$.
	
	\subsection{The Anosov property} 
	\begin{defi}The geodesic flow $\varphi_t:\mathcal M\to \mathcal M$ is \emph{Anosov}, or \emph{uniformly hyperbolic}, if the tangent bundle $T\mathcal M$ admits a continuous flow-invariant decomposition
		\begin{equation} \label{eq: Anosov splitting} T\mathcal M=\mathbb{R}X\oplus E_s\oplus E_u,\end{equation}
	and constants $C,\lambda>0$ such that for all $t\ge 0$, one has
	\[\forall v\in E_s, \ |\mathrm{d}\varphi_t\cdot v|\le C\mathrm{e}^{-\lambda t}|v|,\]
	and
	\[ \forall v\in E_u, \ |\mathrm{d}\varphi_{-t}\cdot v|\le C\mathrm{e}^{-\lambda t}|v|,\]
	where the norms are taken with respect to any continuous metric on $\mathcal M$. We also write $E_0:=\mathbb RX$. The bundles $E_s,E_u$ are called \emph{stable} and \emph{unstable}, respectively. If the geodesic flow is Anosov, we say that $M$ is an \emph{Anosov manifold}. \end{defi} 
	The bundles $E_s,E_u$ are both $(d-1)$-dimensional and $E_s\cap E_u=\{0\}$. Moreover, 
	\[E_s\oplus E_u=\mathbb H\oplus \mathbb V=(\mathbb RX)^\bot,\]
	where $(\mathbb RX)^\bot$ is the orthogonal of $\mathbb RX$ with respect to the Sasaki metric. We can thus define projection maps 
	\[p_{E_u},p_{E_s}:(\mathbb RX)^\bot\to (\mathbb RX)^\bot, \qquad p_{\mathbb H},p_{\mathbb V}:(\mathbb RX)^\bot\to (\mathbb RX)^\bot.\]
	Since $E_s\cap \mathbb V=E_u\cap \mathbb V=\{0\}$, we know that the maps
	\begin{equation} \label{eq: maps are invertible}p_{E_u}:\mathbb V\to E_u, \quad p_{\mathbb H}:E_u\to \mathbb H, \quad p_{E_s}:\mathbb V\to E_s, \quad p_{\mathbb H}:E_s\to \mathbb H\end{equation}
	are invertible. 
	\begin{defi}
		The \emph{minimal rate of expansion of the flow} is defined by
		\begin{equation}\label{eq: def gamma0} \gamma_{\rm min}:=\lim_{t\to +\infty}\inf_{z\in \mathcal M} -\frac 1t\log \big\|\mathrm{d}\varphi_t(z)_{|E_s}\big\|=\lim_{t\to +\infty}\inf_{z\in \mathcal M} -\frac 1t\log \big\|\mathrm{d}\varphi_{-t}(z)_{|E_u}\big\|. \end{equation}
		We also define the \emph{exponential growth rate of the unstable Jacobian}
		\begin{equation} \label{eq: defi gamma 0 higher dim}\gamma_0:=\underset{t\to +\infty}{\lim} \frac1t \inf_{x\in \mathcal M} \log \det \big({\mathrm{d}\varphi_t}|_{E_u(x)\to E_u(\varphi_t(x))}\big),\end{equation}
		where the determinants are computed with respect to any continuous metric.
		\end{defi}
	The Anosov property implies that $\gamma_{0}>0$, but we can give some quantitative estimates on $\gamma_0$ when $M$ has pinched negative curvature. More precisely, if $M$ enjoys two-sided curvature bounds $-\kappa_{\rm max}^2\le -K\le -\kappa_{\rm min}^2$, and $d$ denotes the dimension of $M$, then
	\[(d-1)\kappa_{\rm min}\le \gamma_{0}\le (d-1)\kappa_{\rm max}.\]


	\subsection{Geometry of $\widetilde M$} Assume that $(M,g)$ is an Anosov manifold. Then, by a famous result of Klingenberg \cite{Klingenberg1974}, $M$ has no conjugate points, thus its universal cover $(\widetilde M,\tilde g)$ is simply connected and the exponential map $\exp_x:T_x\widetilde M\to \widetilde M$ is a global diffeomorphism for any $x\in \widetilde M$ (see e.g. \cite[Theorem 13.3.10]{Lefeuvre} and \cite[Lemma 13.2.5]{Lefeuvre}).
	
	When $M$ has curvature $\le -\kappa^2<0$ everywhere, we know by the Cartan--Hadamard theorem that $\widetilde M$ is a $\operatorname{CAT}(-\kappa^2)$-space, in particular it satisfies the standard geometric comparison theorems with the model space of constant curvature $-\kappa^2$. However, Anosov manifolds may contain some regions of positive curvature, so that $\widetilde M$ may even fail to be $\operatorname{CAT}(0)$. Nevertheless, $\widetilde M$ behaves like a negatively curved space on large scales. Indeed, it was first shown by Ruggiero \cite{RuggieroGH} that the universal cover of a closed Anosov manifold is Gromov-hyperbolic. Knieper gave another proof of this result \cite[Theorem 4.8]{Knieper2012}, which relied on the characterization from \cite[Chapter III.H.1, Proposition 1.26]{MNB}. For convenience, we include a quick proof below. We first recall the definition of Gromov-hyperbolicity.\medskip

\noindent
\begin{minipage}[c]{0.65\textwidth}
	\raggedright
	Let $\triangle = xyz$ be a geodesic triangle in $\widetilde M$. There is a unique triplet of nonnegative numbers $(a,b,c)$ such that
	\[ d(x,y)=a+b, \quad d(y,z)=b+c, \quad d(z,x)=c+a.\]
	One can then define an inscribed triangle $\triangle' = x'y'z'$ with $x'$ lying on $[yz]$, etc., such that
	\[ a = d(x,z') = d(x,y'), \quad b = d(y,z') = d(y,x'), \quad \]
	and $c = d(z,y') = d(z,x')$.
\end{minipage}%
\hfill
\begin{minipage}[c]{0.35\textwidth}
	\centering
	\includegraphics[width=\textwidth]{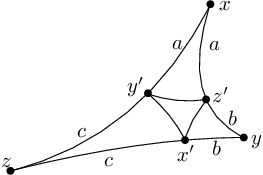}
\end{minipage}

		\begin{defi}\label{defi: GH} The manifold $(\widetilde M,\tilde g)$ is \emph{Gromov-hyperbolic} if there exists a constant $\delta>0$ such that for any geodesic triangle $\triangle=xyz$ in $\widetilde M$, the inscribed triangle $\triangle'=x'y'z'$ defined as above has diameter $\le \delta$. \end{defi} 
	Since the constant $\delta$ in Definition~\ref{defi: GH} may be very large, Gromov-hyperbolicity does not give information about the local geometry of the manifold, thus does not contradict the possible existence of regions with positive curvature.

	\begin{defi}\label{defi: divergence function} A function $e:\mathbb N\to \mathbb{R}$ is a \emph{divergence function} for $\widetilde M$ if for any geodesic lines $c_1,c_2:[0,+\infty)\to \widetilde M$ with $c_1(0)=c_2(0)=x$, for any $R,r\in \mathbb N$ with $d(c_1(R),c_2(R))>e(0)$, any continuous path $\gamma$ connecting $c_1(R+r)$ to $c_2(R+r)$ outside the ball $B(x,R+r)$ satisfies
		\[\operatorname{length}(\gamma)\ge e(r).\]\end{defi}
	Proposition 1.26 in \cite[Chapter III.H.1]{MNB} states that if $\widetilde M$ admits a divergence function $e$ such that $e(n)/n\to +\infty$ then $\widetilde M$ is Gromov-hyperbolic. We are left with showing
	\begin{prop}\label{prop: M has exp div function} Assume that $\widetilde M$ is the universal cover of an Anosov manifold. Then $\widetilde M$ admits an exponential divergence function. In particular, $\widetilde M$ is Gromov-hyperbolic.
	\end{prop}
	\begin{proof}
		Let $e(0)>0$ be fixed below. Consider two geodesic lines as in Definition~\ref{defi: divergence function}, and let $R\in \mathbb N$ be such that $d(c_1(R),c_2(R))>e(0)$. 
		Let $\gamma:[0,1]\to \widetilde M$ be a path connecting $c_1(R+r)$ to $c_2(R+r)$ outside the ball $B(x,R+r)$. In the proof, $g$ is the Riemannian metric on $\widetilde M$ and $G$ is the Sasaki metric on $S\widetilde M$. The following idea is due to Yann Chaubet. \medskip	
	
		\noindent\begin{minipage}[c]{0.6\textwidth}
			\raggedright
		 By approximation, we can always assume $\gamma$ to be smooth, and since $\exp_x:T_x\widetilde M\to \widetilde M$ is a diffeomorphism, we can parameterize
		\begin{equation} \label{eq: gamma(t) param}\gamma(t)=\pi\big(\varphi_{\ell(t)}(x,v(t))\big), \quad t\in [0,1], \end{equation}
		for some functions $\ell(t)\ge R+r$, and $v(t)\in S_x\widetilde M$. Let us also consider 
		\[c(t):=\pi\big(\varphi_R(x,v(t))\big).\] 
	
		\end{minipage}%
		\hspace{1em}
		\begin{minipage}[c]{0.35\textwidth}
			\centering
			\includegraphics[width=\textwidth]{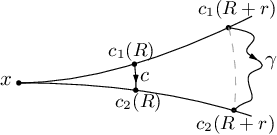}
		\end{minipage} \vspace{0.3em}\\ \medskip 
		Then $c$ is a smooth path from $c_1(R)$ to $c_2(R)$, in particular,
		\[\operatorname{length}(c)\ge d(c_1(R),c_2(R))>e(0).\]
		Note that if $r=0$ then $\gamma$ joins $c_1(R)$ to $c_2(R)$ thus $\operatorname{length}(\gamma)\ge d(c_1(R),c_2(R))>e(0)$, as we wish. Thus from now on we can assume $r\ge 1$. The length of $\gamma$ is given by
		\[\operatorname{length}(\gamma)=\int_0^1 |\dot \gamma(t)|_g\mathrm{d}t.\]
		Differentiating \eqref{eq: gamma(t) param}, we obtain by the chain rule 
		\[\dot \gamma(t)=\mathrm{d}\pi\cdot \big(\dot \ell(t) X(\varphi_{\ell(t)}(x,v(t))+\mathrm{d}\varphi_{\ell(t)}(x,v(t))\cdot \dot v(t)\big).\]
		Since $\mathrm{d}\pi\cdot \mathbb RX$ and $\mathrm{d}\pi\cdot (\mathbb H\oplus \mathbb V)$ are orthogonal subspaces of $TM$, we have
		\[|\dot \gamma(t)|_g\ge |\mathrm{d}\pi\cdot\mathrm{d}\varphi_{\ell(t)}(x,v(t))\cdot \dot v(t)|_g= |p_\mathbb H\cdot\mathrm{d}\varphi_{\ell(t)}(x,v(t))\cdot \dot v(t)|_G.\]
	Similarly,
		\[|\dot c(t)|_g=|p_{\mathbb H}\cdot \mathrm{d}\varphi_R(x,v(t))\dot v(t)|_G.\]
		We write
		\[\mathrm{d}\varphi_{\ell(t)}(x,v(t))\cdot \dot v(t)=\mathrm{d}\varphi_{\ell(t)-R}\cdot (p_{E_u}+p_{E_s}) \mathrm{d}\varphi_R\dot v(t),\]
		and find by the triangle inequality 
		\begin{align*}|\dot \gamma(t)|_g & \ge |p_{\mathbb H}\cdot \mathrm{d}\varphi_{\ell(t)-R}\cdot p_{E_u}\mathrm{d}\varphi_{R}\dot v(t)|-|p_{\mathbb H}\cdot \mathrm{d}\varphi_{\ell(t)-R}\cdot p_{E_s}\mathrm{d}\varphi_{R}\dot v(t)|
		\\ & \ge \|(p_{\mathbb H}\cdot \mathrm{d}\varphi_{\ell(t)-R}\cdot p_{E_u})^{-1}\|^{-1}\cdot \|p_{E_u} \mathrm{d}\varphi_R\dot v(t)\|-\|p_{\mathbb H}\cdot \mathrm{d}\varphi_{\ell(t)-R}\cdot p_{E_s}\|\cdot |\mathrm{d}\varphi_{R}\dot v(t)|_G.\end{align*}
		Now, by the Anosov property of the flow and the fact that the maps in \eqref{eq: maps are invertible} are invertible, there is a constant $C>0$ such that
		\[|p_{E_u}\mathrm{d}\varphi_R\dot v(t)|_G\ge \frac 1C |p_{\mathbb H} \mathrm{d}\varphi_R\dot v(t)|_G=\frac 1C|\dot c(t)|_g,\]
		and
		\[|p_{E_s}\mathrm{d}\varphi_R\dot v(t)|_G\le C|p_{\mathbb H} \mathrm{d}\varphi_R\dot v(t)|_G=C|\dot c(t)|_g.\]
		%
		Similarly, because $\ell(t)-R\ge r$ we have
		\[\|(p_{\mathbb H}\cdot \mathrm{d}\varphi_{\ell(t)-R}\cdot p_{E_u})^{-1}\|^{-1}\ge C_\varepsilon \mathrm{e}^{(\gamma_0-\varepsilon)r}, \quad \|p_{\mathbb H}\cdot \mathrm{d}\varphi_{\ell(t)-R}\cdot p_{E_s}\|\le C_\varepsilon\mathrm{e}^{-(\gamma_0-\varepsilon)r}.\]
		We deduce
		\[|\dot \gamma(t)|_g\ge \big(c_\varepsilon \mathrm{e}^{(\gamma_0-\varepsilon)r}-C_\varepsilon\big)|\dot c(t)|_g.\]
		Integrating over $[0,1]$ leads to
		\[\operatorname{Length}(\gamma)\ge \big(c_\varepsilon \mathrm{e}^{(\gamma_0-\varepsilon)r}-C_\varepsilon\big)\operatorname{Length}(c)\ge e(0)\big(c_\varepsilon \mathrm{e}^{(\gamma_0-\varepsilon)r}-C_\varepsilon\big),\]
		where the last inequality holds for $r\ge r_\varepsilon$ large enough. Thus, we can set $e(0)=1$, $e(r)=0$ for $1\le r\le r_\varepsilon$ and for $r> r_\varepsilon$ set 
		\[e(r)=\big(c_\varepsilon \mathrm{e}^{(\gamma_0-\varepsilon)r}-C_\varepsilon\big).\]
	\end{proof}
	A byproduct of the proof above is the following estimate on the divergence of geodesics, which essentially states that if $x,y,z\in \widetilde M$ and $y$ is close to $z$, then the geodesics joining $x$ to $y$ and $x$ to $z$ remain close to each other (in the unit tangent bundle) for the whole trajectory.
		\begin{lem}	\label{lem: divergence geod} Assume $(M,g)$ is a closed Anosov manifold. For any $\eta>0$, there is a constant $C_\eta>0$ such that if $x\in \widetilde M$ and $v,w\in S_x\widetilde M$ are such that $d_{\widetilde M}(\exp_x(rv),\exp_x(rw))\le \eta$, then for all $0\le t\le r$, one has
		\[d_{S\widetilde M}\big(\varphi_t(x,v),\varphi_t(x,w)\big)\le C_\eta \mathrm{e}^{-c(r-t)}.\]
	\end{lem}

	\subsection{Fourier analysis on the fibers}\label{subsec: Fourier} Each fiber of the unit tangent bundle $SM\to M$ is isometric to the $(d-1)$-sphere $S^{d-1}$. By diagonalizing the Laplacian on the $(d-1)$-sphere, this allows decomposing any function $f\in L^2(SM)$ in so-called spherical harmonics. This decomposition allows us to compute seamlessly some operator norms in the next section.
		
	For the material presented here, see \cite[Chapter 14]{Lefeuvre}. Recall that we have an orthogonal decomposition
	\[T(SM)=\mathbb RX\oplus \mathbb H\oplus \mathbb V.\]
	We can then define the \emph{vertical gradient} $\nabla_{\mathbb V}f$ of a function $f\in C^\infty(SM)$, and a dual operator $\nabla_{\mathbb V}^*:C^\infty(T(SM))\to C^\infty(SM)$. The \emph{vertical Laplacian} is then defined as
	\[\Delta_{\mathbb V}:=\nabla_{\mathbb V}^*\nabla_{\mathbb V}:C^\infty(SM)\to C^\infty(SM).\]
	Given an isometric identification $S_xM\simeq S^{d-1}$, one has
	\[\Delta_{\mathbb V}f(x,v)=\Delta_{S^{d-1}}f(x,\cdot)\big|_v.\]
	
	We turn to the case of flat bundles. Let $E\to M$ be a flat bundle over $M$; it can be lifted to a flat bundle $\mathcal E:=\pi^*E$ over $SM$. If $f(x,v)$ is a smooth section of $\mathcal E\to SM$, then for each $x\in M$, one can view $f(x,\cdot)$ as an element of $C^\infty(S_xM,E_x)$, where $E_x$ is the fiber at $x$ of the projection $E\to M$. As above, by fixing an isometric identification $S_xM\simeq S^{d-1}$, one has:
	\[\Delta_{\mathbb V}^E f(x,v)=\Delta_{S^{d-1}} f(x,\cdot)\big|_v.\]
	The vertical Laplacian $\Delta_{\mathbb V}^E$ is essentially self-adjoint on $L^2(SM,\mathcal E)$, hence one can define its functional calculus.

\subsubsection{Fourier analysis on $S\widetilde M$} \label{subsubsec: Fourier analysis SM} We now discuss the case of the universal cover $\widetilde M$, with the trivial bundle $\mathbb C\to \widetilde M$. Since $\widetilde M$ is simply connected, we have an isomorphism $S\widetilde M\simeq \widetilde M\times S^{d-1}$ such that for each $x$, the restriction $S_x\widetilde M\to S^{d-1}$ is an isometry. We denote by $(\lambda_n)_{n\ge 0}$ the ordered sequence of eigenvalues of $\Delta_{S^{d-1}}$, and consider an associated orthonormal basis $(\omega_n)_{n\ge 0}$ of eigenfunctions of the spherical Laplacian, that is $\|\omega_n\|_{L^2(S^{d-1})}=1$ and $\Delta_{S^{d-1}}\omega_n=\lambda_n\omega_n$. These functions pullback to functions on $S\widetilde M$ that will still be denoted $\omega_n$, and such that
	\[\Delta_{\widetilde{\mathbb V}} \omega_{n}=\lambda_n \omega_{n}.\]
	Then, any $u\in L^2(S\widetilde M)$ can be decomposed as
	\[u(x,v)=\sum_{n\ge 0}u_n(x)\omega_n(x,v),\qquad \|u\|_{L^2(S\widetilde M)}^2=\sum_{n\ge 0} \|u_{n}\|_{L^2(\widetilde M)}^2.\]

	\subsubsection{Fourier projectors} We now introduce Fourier projectors on the fibers and study how they act as semiclassical pseudodifferential operators. Let $g$ be a smooth compactly supported function, equal to $1$ on $[-2,2]$. We need to elucidate how $g(h^2\Delta_{\mathbb V}^E)$ behaves when composed with pseudodifferential operators. We refer to Appendix~\ref{appendix} for the notions of semiclassical calculus used here. We first consider the case of the trivial bundle $\mathbb C\to SM$. We show that $g(h^2\Delta_{\mathbb V})$ acts microlocally like the identity on $\{|\xi|\le 2\}$.	\begin{prop}\label{prop: inserting g(h2lap) high dim}
		
		For any $a\in C^\infty(T^*\mathcal M)$ supported on $|\xi|\le 1$, one has
		\[\operatorname{Op}_h^{\mathcal M}(a)g(h^2\Delta_{\mathbb V})=\operatorname{Op}_h^{\mathcal M}(a)+\mathcal O(h^\infty), \qquad g(h^2\Delta_{\mathbb V})\operatorname{Op}_h^{\mathcal M}(a)=\operatorname{Op}_h^{\mathcal M}(a)+\mathcal O(h^\infty).\]\end{prop}	
	\begin{proof} 
		
		\textbf{1.} By the Helffer--Sj\"ostrand formula (see \cite[Theorem 14.18]{zworski2012semiclassical}), one can express
		\[g(h^2\Delta_{\mathbb V})=\frac{1}{2\mathrm{i}\pi}\int_{\mathbb{C}} \bar \partial \tilde g(z)(z-h^2\Delta_{\mathbb V})^{-1}\mathrm{d}z\wedge \mathrm{d}\bar z.\]
		where $\tilde g$ is an almost analytic extension of $g$, \emph{i.e.} a compactly supported function on $\mathbb C$ such that $\tilde g_{|\mathbb R}=g$, and 
		\[|\bar \partial \tilde g(z)|=\mathcal O(|\operatorname{Im}(z)|^\infty).\]
		The convergence of the integral in operator norm is ensured by the spectral-theoretical bound
		\[\|(z-h^2\Delta_{\mathbb V})^{-1}\|\le \frac{1}{|\operatorname{Im}(z)|},\]
		which follows from the self-adjointness of $\Delta_{\mathbb V}$.
		
		\textbf{2.} Since $\Delta_{\mathbb V}$ is a differential operator, we know that $h^2\Delta_{\mathbb V}=\operatorname{Op}_h^{\mathcal M}(p)+\mathcal O_{L^2\to L^2}(h^\infty)$, for some smooth, $h$-dependent symbol $p$, satisfying $p=|\xi|_{\mathbb V}^2+hr$ with some remainder $r\in S^1(T^*\mathcal M)$. Note that $p$ depends on the choice of quantization $\operatorname{Op}_h^{\mathcal M}(\cdot)$. Using a parametrix construction, we can build a symbol $b(z)$ supported on $\{(x,\xi)\in T^*(SM)~:~|\xi|\le 2\}$ with controlled estimates such that 
		\begin{equation} \label{eq: local parametrix b(z)} (z-\operatorname{Op}_h^{\mathcal M}(p))\operatorname{Op}_h^{\mathcal M}(b(z))\operatorname{Op}_h^{\mathcal M}(a)=\operatorname{Op}_h^{\mathcal M}(a)+\mathcal O\Big(\frac{h^N}{|\operatorname{Im}(z)|^N}\Big)_{L^2\to L^2},\end{equation}
		with $a$ as in Proposition~\ref{prop: inserting g(h2lap) high dim} independent of $z$. 	More precisely, letting $\chi\equiv 1$ on $\{|\xi|\le 2\}$, one can take $b$ of the form
		\begin{equation} \label{eq: expression b(z)}b(z)=\frac{\chi}{z-p}+\chi\sum_{n=2}^{N}h^n\sum_{2\le k\le 2n+1}\frac{Q_{n,k}}{(z-p)^k}
			.\end{equation}
		Here $ Q_{n,k}$ 
		is a polynomial in the derivatives of $p$ whose precise expression depends on the choice of quantization, but is irrelevant to us. The important fact is that $Q_{n,k}$ is independent of $z$. Similar computations where expression of the form \eqref{eq: expression b(z)} appear can be found e.g. in \cite{NonnenmacherLectures}---beware that our expression differs slightly because we do not want $z$ to appear in the numerator. On manifolds, one can use the calculus developed in the Appendix of \cite{DJN21} to obtain \eqref{eq: local parametrix b(z)}. Although \cite{DJN21} works with the right quantization, their presentation can be adapted to deal with the Weyl quantization.

		Let us go back to the proof. From \eqref{eq: local parametrix b(z)} and the $L^2$-continuity theorem we obtain
		\begin{equation} \label{eq: replace Op(z-b) by lap} (z-h^2\Delta_{\mathbb V})\operatorname{Op}_h^{\mathcal M}(b(z))\operatorname{Op}_h^{\mathcal M}(a)=\operatorname{Op}_h^{\mathcal M}(a)+\mathcal O_{L^2\to L^2}(\frac{h^N}{|\operatorname{Im}(z)|^N}).\end{equation}
		Recall that $\|(z-h^2\Delta_{\mathbb V})^{-1}\|\le |\operatorname{Im}(z)|^{-1}$. By \eqref{eq: replace Op(z-b) by lap}, one obtains
		\begin{equation} \operatorname{Op}_h^{\mathcal M}(b(z))\operatorname{Op}_h^{\mathcal M}(a)=(z-h^2\Delta_{\mathbb V})^{-1}\operatorname{Op}_h^{\mathcal M}(a)+\mathcal O(\frac{h^N}{|\operatorname{Im}(z)|^{N+1}})_{L^2\to L^2}.\end{equation}
		
		\textbf{3.} We now integrate over $\mathbb C$. We use the fact that
		\[\int_{\mathbb C} \frac{\bar \partial \tilde g(z)}{(z-p)^k}\mathrm{d}z\wedge \mathrm{d}\bar z=\frac{2\pi \mathrm{i}}{(k-1)!}g^{(k-1)}(p).\]
		By \eqref{eq: expression b(z)}, we have
		\begin{equation} \label{eq: integrated d bar g b}\int_{\mathbb C} \bar \partial \tilde g(z)b(z)\mathrm{d}z\wedge \mathrm{d}\bar z=\chi\Big(g(p)+\sum_{n=2}^N h^n\sum_{2\le k\le 2n+1} g^{(k-1)}(p)Q_{2n,k}\Big) +\mathcal O(h^N).\end{equation}
		We used the almost analyticity of $\tilde g$ to control the remainder. Since $g\equiv 1$ on $[-2,2]$ we have $g^{(k)}\equiv 0$ on this interval for all $k\ge 1$. Because $\chi$ is supported on $|\xi|\le 1$, we have $|p|\le 1+\mathcal O(h)$ there, thus for $h$ small enough we have $\chi g^{(k)}(p)\equiv 0$. Hence \eqref{eq: integrated d bar g b} simplifies to
		\[\int \bar \partial \tilde g(z)b(z)\mathrm{d}z\wedge \mathrm{d}\bar z= \chi g(p)+\mathcal O(h^N).\]
		Altogether, we obtain
		\[g(h^2\Delta_{\mathbb V}) \operatorname{Op}_h^{\mathcal M}(a)=\operatorname{Op}_h^{\mathcal M}(\chi g(p))\operatorname{Op}_h^{\mathcal M}(a)+\mathcal O(h^\infty).\]
		Since $\chi g(p)\equiv 1$ on the support of $a$, we conclude that
		\[g(h^2\Delta_{\mathbb V}) \operatorname{Op}_h^{\mathcal M}(a)=\operatorname{Op}_h^{\mathcal M}(a)+\mathcal O(h^\infty),\]
		as we wished.\end{proof}
	We now deal with flat bundles $\mathcal E\to \mathcal SM$ obtained by lifting $E\to M$, and provide a generalization of Proposition~\ref{prop: inserting g(h2lap) high dim} to this setting. We refer to Appendix \ref{appendix} for the notations used here.
	\begin{prop}\label{prop: inserting g(h2lap) highdim bundle}Let $\mathcal E\to \mathcal M$ be a unitary flat bundle obtained by lifting $E\to M$. For any $a\in C^\infty(T^*\mathcal M)$ supported on $|\xi|\le 1$, one has
		\[\operatorname{Op}_h^{\mathcal E}(a)g(h^2\Delta_{\mathbb V}^E)=\operatorname{Op}_h^{\mathcal E}(a)+\mathcal O(h^\infty), \qquad g(h^2\Delta_{\mathbb V}^E)\operatorname{Op}_h^{\mathcal E}(a)=\operatorname{Op}_h^{\mathcal E}(a)+\mathcal O(h^\infty).\]
		The errors are uniform in $\mathcal E$.\end{prop}
	\begin{proof}
		We only deal with the first identity, the second one is proved similarly. By definition of $\operatorname{Op}_h^{\mathcal E}(\cdot)$, we have
		\[\operatorname{Op}_h^{\mathcal E}(a)g(h^2\Delta_{\mathbb V}^E)=\sum_j R_j^*\chi_j \operatorname{Op}_h^{\mathcal M}(\eta_j a)R_j \chi_j g(h^2\Delta_{\mathbb V}^E).\]
		With our choice of local trivializations, we have
		\[R_jg(h^2\Delta_{\mathbb V}^E)=g(h^2\Delta_{\mathbb V})R_j,\]
		 This leads to:
		\[\operatorname{Op}_h^{\mathcal E}(a)g(h^2\Delta_{\mathbb V}^E)=\sum_j R_j^*\chi_j \operatorname{Op}_h^{\mathcal M}(\eta_j a)\chi_j g(h^2\Delta_{\mathbb V})R_j\eta_j.\]
		By Proposition~\ref{prop: inserting g(h2lap) high dim} we have for each $j$:
		\[\operatorname{Op}_h^{\mathcal M}(\eta_j a)\chi_jg(h^2\Delta_{\mathbb V})=\operatorname{Op}_h^{\mathcal M}(\eta_j a)\chi_j +\mathcal O(h^\infty).\]
		Finally, using $\chi_j\prec \eta_j$:
		\[\operatorname{Op}_h^{\mathcal E}(a)g(h^2\Delta_{\mathbb V}^E)=\sum_j R_j^*\chi_j \operatorname{Op}_h^{\mathcal M}(\eta_j a)R_j\chi_j+\mathcal O(h^\infty)=\operatorname{Op}_h^{\mathcal E}(a)+\mathcal O(h^\infty),\]
		which concludes the proof.

	\end{proof}

	\section{Low frequency propagation on $\widetilde M$} \label{sec: spherical mean} 	In this section, we study the evolution under the geodesic flow of functions $f\in L^2(S\widetilde M)$ that oscillate on scales $\lesssim h^{-1}$ in the vertical direction. Proposition~\ref{prop: operator norm spherical mean} below gives a quantitative formulation of the following heuristic: under the propagation by the geodesic flow, the $L^2$-mass tends to escape towards high Fourier modes in the fibers. 
	More precisely, we show that if $f$ is concentrated on Fourier modes of order $\lesssim h^{-1}$, then the $L^2$-mass of the projection of $f\circ \varphi_t$ to modes of order $\lesssim h^{-1}$ decays like $h^{-(d-1)}\mathrm{e}^{\delta_0 t}$, where $\delta_0$ is half the pressure of the potential $-2\psi^u$. This estimate is the main analytical input in the proof of Theorem \ref{thm: extension resolvent Hs}, and may be of independent interest.

	\subsection{Spherical mean operator} We first study propagation of functions that do not oscillate in the vertical direction. We have a pullback operator
	\[\pi^*:L^2(\widetilde M)\to L^2(S\widetilde M), \ f\mapsto f\circ \pi.\]
	We also define $\pi_*:L^2(S\widetilde M)\to L^2(\widetilde{M})$ by
	\[\pi_*F(x)=\frac{1}{\operatorname{vol}(S^{d-1})}\int_{S_x{\widetilde M}} F(x,v)\mathrm{d}\mu_x(v).\] 
	\begin{defi}The \emph{spherical mean operator} $\mathcal L_t:L^2(\widetilde M)\to L^2(\widetilde M)$ is defined by
		\begin{equation} \label{def: spherical mean operator} \mathcal L_t f:=\pi_*\mathrm{e}^{-t\widetilde X}\pi^*.\end{equation}
	More explicitly, one has
	\[\mathcal L_t f(x)=\frac{1}{\operatorname{vol}(S^{d-1})}\int_{S_x\widetilde M} f(\pi(\varphi_t(x,v)))\mathrm{d}\mu_x(v).\]
	\end{defi}
	We can give an alternative expression of $\mathcal L_t$. If $y=\exp_x(tv)$ for some $v\in S_x\widetilde M$, we define the volume density at $y$ seen in normal coordinates centered at $x$:
	\begin{equation} \label{eq: def J(x,y)}J(x,y):=\det \mathrm{d}_v\exp_x(t~\! \cdot),\end{equation}
	where the determinant is computed with respect to any orthonormal bases of $T(S_x\widetilde M)$ and $T_yS(x,t)$. Note that $J(x,y)\sim d(x,y)^{d-1}$ when $d(x,y)\to 0$. Also, let $\mathrm{d}\sigma_{S(x,t)}$ be the Riemannian surface measure on $S(x,t)$. Then, one has, for any $f\in C^\infty(\widetilde M)$:
	\begin{equation} \label{eq: alternative L_tf}\mathcal L_t f(x)=\int_{S(x,t)} J(x,y)^{-1}f(y)\mathrm{d}\sigma_{S(x,t)}(y).\end{equation}
	
	\begin{rem} \label{rem: confusion operators} The spherical mean operator as defined in \eqref{def: spherical mean operator} should not be confused with the averaging operator over the Riemannian sphere of radius $t$ with respect to the surface measure, denoted by $\mathcal A_t$, that was studied by Knieper \cite{KnieperSphericalMean}. Indeed, one has to replace the weight $J(x,y)^{-1}$ by $\tfrac{1}{\operatorname{vol}S(x,t)}$ in \eqref{eq: alternative L_tf}. Note however that both operators coincide in constant curvature.
	\end{rem}

	\subsection{Topological Pressure} \label{subsec: pressure} We review the various definitions of the topological pressure of a continuous potential $F\in C^0(SM)$. Any such potential lifts to a $\Gamma$-invariant potential $\widetilde F\in C^0(S\widetilde M)$. Some particularly interesting potentials are given by scalar multiples of the \emph{unstable Jacobian}, which is defined as follows.
	\begin{defi}[Unstable Jacobian] The unstable Jacobian $\psi_u$ is defined by
	\begin{equation} \label{eq: def unstable jac}\psi_u(x,v)=\frac{\mathrm{d}}{\mathrm{d}t}\Big|_{t=0}\log \det (\mathrm{d}\varphi_t)_{|E_u(x,v)\to E_u(\varphi_t(x,v))},\end{equation}
	where the determinant is computed with respect to the Sasaki metric. This is a H\"older continuous function on $SM$.\end{defi}

	 For $x,y\in \widetilde M$ such that $y=\exp_x(tv)$ with $v\in S_x\widetilde M$, we set
	\[\int_x^y \widetilde F:=\int_0^t \widetilde F(\varphi_s(x,v))\mathrm{d}s.\]
	Note that if $F(x,v)=F(x,-v)$ then
	\[\int_x^y \widetilde F=\int_y^x \widetilde F.\]

	\begin{defi}
		A subset $S\subset SM$ is \emph{$(\varepsilon,t)$-separated} if for any $x,y\in S$, there is some $s\in [0,t]$ such that $d(\varphi_s(x),\varphi_s(y))>\varepsilon$. A subset $S\subset SM$ is \emph{$(\varepsilon,t)$-spanning} if for each $y\in SM$ there is some $x\in S$ such that $d(\varphi_s(x),\varphi_s(y))<\varepsilon$ for all $s\in [0,t]$.
	\end{defi}
	Let $F:SM\to \mathbb R$ be a continuous potential. Following Ruelle \cite{RuellePressure}, we define
	\[Z_t^+(F,\varepsilon):=\sup \Big\{\sum_{x\in S}\exp\int_0^t F(\varphi_s(x))\mathrm{d}t~:~ \text{$S$ is $(\varepsilon,t)$-separated}\Big\},\]
	and
	\[Z_t^-(F,\varepsilon):=\inf \Big\{\sum_{x\in S}\exp\int_0^t F(\varphi_s(x))\mathrm{d}t~:~ \text{$S$ is $(\varepsilon,t)$-spanning}\Big\}.\]
	One defines
	\[P^{\pm}(F,\varepsilon)=\underset{t\to +\infty}{\limsup} \frac{1}{t} \log Z^{\pm}_t(F,\varepsilon).\]
	
	\begin{defi} The \emph{topological pressure} $\operatorname{Pr}(F)$ is defined as
		\[\operatorname{Pr}(F)=\lim_{\varepsilon\downarrow 0} P^{\pm}(F,\varepsilon),\]
		where both limits coincide.
	\end{defi}
	For a proof that $\lim P^+(F,\varepsilon)=\lim P^-(F,\varepsilon)$, see Walters \cite{Walters1975}. When $F=0$, one recovers the \emph{topological entropy}  $h_{\rm top}$ of the flow. It was shown by Manning \cite{Manning} that when $(M,g)$ is negatively curved the topological entropy coincides with the volume entropy: for any $x\in \widetilde M$ we have
	\[h_{\rm top}=\underset{r\to +\infty}\lim \frac 1r \log \operatorname{vol} B(x,r).\]
	This result was extended in two directions. First, Ruelle \cite{RuellePressure} extended Manning's result to arbitrary continuous potentials and showed that for any $c>0$, denoting by $A(x,r,c)$ the annulus $\{y~:~ |d(x,y)-r|\le c\}$, one has
	\begin{equation} \label{eq: pressure as lim int over annulus} \operatorname{Pr}(F)=\underset{r\to +\infty}\lim \frac 1r \log \int_{A(x,r,c)} \exp(\int_x^y \widetilde F)\mathrm{dvol(y)}.\end{equation}
	Secondly, Freire and Ma{\~n}{\'e} \cite{FreireMane} showed that Manning's result still holds under the assumption that $(M,g)$ has no conjugate points, in particular when $(M,g)$ is Anosov. Although the proof in \cite{FreireMane} is conceptually different that the one in \cite{Manning}, it was pointed out by the authors that Manning's proof extends seamlessly to the case of Anosov manifolds; in fact this is also the case of Ruelle's proof. Indeed, the nonnegative curvature assumption in \cite{Manning,RuellePressure} is only used to obtain a mild version of Lemma \ref{lem: divergence geod}, which we have proved for Anosov manifolds. The moral of the story is that \eqref{eq: pressure as lim int over annulus} holds under the sole assumption that $(M,g)$ is Anosov. The pressure $\operatorname{Pr}(F)$ is also the \emph{critical exponent} of the Poincar{\'e} series associated with the potential $F$. For any $x,y\in \widetilde M$ we have
	\[\operatorname{Pr}(F)=\inf\Big\{s\in \mathbb R~:~\sum_{\gamma\in \Gamma}\mathrm{e}^{-s d(x,\gamma y)+\int_x^{\gamma y}\widetilde F}<+\infty\Big\}.\]
	
		\subsection{Operator norm of $\mathcal L_t$}  
	The norm of $\mathcal L_t:L^2(\widetilde M)\to L^2(\widetilde M)$ is estimated as follows. Recall that $\psi_u$ denotes the unstable Jacobian.
	
	\begin{prop}\label{prop: operator norm spherical mean} Let $\delta_0=\operatorname{Pr}(-2\psi^u)/2$. Then $\delta_0<0$ and  
		\begin{equation} \label{eq: limsup 1tlog Lt}\underset{t\to +\infty}\lim \frac 1t\log \|\mathcal L_t\|_{L^2(\widetilde M)}= \delta_0.\end{equation}
	\end{prop}
	Our method allows computing the norm of more general operators. Let $F\in C^\alpha(SM)$ be a H\"older continuous potential satisfying $F(x,v)=F(x,-v)$, and denote by $\widetilde F$ its lift to $S\widetilde M$. Then,
	\[\underset{t\to +\infty}{\lim} \frac 1t\log \big\|\pi_*\mathrm{e}^{-t(\widetilde X+\widetilde F)}\pi^*\big\|_{L^2(\widetilde M)}=\frac 12\operatorname{Pr}(2(F-\psi_u)),\]
	The spherical mean operator $\mathcal L_t$ corresponds to $F\equiv 0$. By taking $F=\psi_u$, one can compute the norm of the averaging operator $\mathcal A_t$ over the Riemannian sphere of radius $t$ studied in \cite{KnieperSphericalMean}. Since $\mathcal A_t$ is well approximated by $\mathrm{e}^{-\operatorname{Pr}(0) t} \pi_* \mathrm{e}^{-t(\widetilde X+\psi^u)}\pi^*$, one finds 
	\[\underset{t\to +\infty}{\lim} \frac 1t\log \|\mathcal A_t\|_{L^2(\widetilde M)}=\frac 12 \operatorname{Pr}(0)-\operatorname{Pr}(0)=-\frac 12 h_{\rm top}.\]
	
	\subsubsection{Related results} Nica studied the spherical averaging operator on nonelementary hyperbolic groups \cite{nica2017operator,nica2024normssphericalaveraging} (here one fixes a word-length metric with respect to a finite subset of generators). Given a hyperbolic group $\Gamma$, with regular representation $(\lambda_{\Gamma},\ell^2(\Gamma))$, he showed that 
	\begin{equation} \label{eq: nica result} \|\lambda_{\Gamma}(\mathbf 1_{S(n)})\|_{\ell^2(\Gamma)\to \ell^2(\Gamma)}\asymp n |S(n)|^{-\frac 12}.\end{equation}
	This mirrors the estimate for the averaging operator over Riemannian spheres of radius $t$. As we mentioned in Remark \ref{rem: confusion operators}, we do not consider the averaging operator over the Riemannian sphere of radius $t$, because different points on the sphere are assigned different weights---specifically, we use weights $J(x,y)^{-1}$ instead of $1/ \operatorname{vol} S(t)$. Nevertheless, because of the particular properties of the function $J(x,y)$ (see \S \ref{subsec: estimates J}), the methods of Nica can be adapted to estimate the norm of some analogous operators on hyperbolic groups, by making use of the \emph{cocycle bound} from \cite[\S 3]{nica2024normssphericalaveraging}. Ultimately, despite the continuous setting demands some additional work, our overall strategy follows the same principles as \cite{nica2024normssphericalaveraging}, which is not surprising as both rely heavily on the Gromov-hyperbolicity property.

	\subsubsection{Strengthening Proposition \ref{prop: operator norm spherical mean}}
	When $M$ is negatively curved, one can show by modifying the proof below that
	\begin{equation}\label{eq: better bound Lt} \|\mathcal L_t\|_{L^2(\widetilde M)}\le C t\mathrm{e}^{\delta_0t}, \end{equation}
	which is more precise than \eqref{eq: limsup 1tlog Lt}. Indeed, in this case we know from \cite{PaulinPollicottSchapira} that for any H\"older potential $F$, one has
	\begin{equation} \label{eq: better bound integrals}\int_{A(x,t,c)} \exp\Big(\int_x^{y}\widetilde F\Big)\mathrm{dvol}(y)\le C \mathrm{e}^{\operatorname{Pr}(F)t},\end{equation}
	which is in turn stronger than \eqref{eq: pressure as lim int over annulus}. Although \eqref{eq: better bound integrals} should be expected to hold when $M$ is merely Anosov, this result does not appear in the literature. The improved estimate \eqref{eq: better bound Lt} would be important in order to obtain an improvement of Theorem \ref{thm: extension resolvent} allowing the compact $K$ to approach the line $\operatorname{Re}(z)=\delta_0$ when $n\to +\infty$. 
	

	\subsection{Preliminary estimates on the function $J$} \label{subsec: estimates J} First, note that we have the symmetry property 
	\[J(x,y)=J(y,x).\]
	The Jacobian $J(x,y)$ might be expressed in terms of the potential $\psi_u$ defined in \eqref{eq: def unstable jac}.
	\begin{lem}\label{lem: comparison of J(x,y) with unstable jac}
	If $\varphi_t(x,v)=(y,w)$ and $t\ge 1$, we have
		\[J(x,y)\asymp \exp\Big(\int_0^t \psi_u(\varphi_s(x,v))\mathrm{d}s\Big)=\det (\mathrm{d}\varphi_t)_{E_u(x,v)\to E_u(y,w)}.\]
	If $d(x,y)\le 1$ then $J(x,y)\asymp d(x,y)^{d-1}$.
	\end{lem}
	\begin{proof}\textbf{1.} If $(y,w)=\varphi_t(x,v)$ we can express 
		\[J(x,y)=\det \mathrm{d}\pi\circ \mathrm{d}\varphi_t:\mathbb V(x,v)\to w^\bot\subset T_y\widetilde M.\]
		Recall that $\mathbb V\subset E_u\oplus E_s$ and $\mathbb V\cap E_u=\mathbb V\cap E_s=\{0\}$. Restricted to $\mathbb V$, we can write
		\[\mathrm{d}{\varphi_t}=\mathrm{d}\varphi_t\circ p_{E_u}+\mathrm{d}\varphi_t\circ p_{E_s},\]
		where the maps $p_{E_u}:\mathbb V\to E_u$ and $p_{E_s}:\mathbb V\to E_s$ are both invertible. 
		By the Anosov property, $\mathrm{d}{\varphi_t}_{|E_s}$ is exponentially contracting and $\mathrm{d}{\varphi_t}_{|E_u}$ is exponentially expanding as $r\to +\infty$. Moreover $E_u$ and $E_s$ are preserved by the flow, and both project diffeomorphically onto $(\mathbb Rw)^\bot$ through the map $\mathrm{d}\pi(y,w):T_{(y,w)}(SM)\to T_yM$ (see \cite[Lemma 13.3.8]{Lefeuvre}). This leads to
		\[\det (\mathrm{d}\pi\circ \mathrm{d}\varphi_t)_{\mathbb V\to w^\bot}\asymp \det (\mathrm{d}\varphi_t)_{E_u\to E_u}.\]
		
		\textbf{2.} If $t=d(x,y)\le 1$ then using \eqref{eq: def J(x,y)} we can write
		\[ J(x,y)=t^{d-1}\det \mathrm{d}_{tv} \exp_x(\cdot).\]
		Since $\mathrm{d}_0\exp_x(\cdot)$ is the identity, we deduce $J(x,y)\sim d(x,y)^{d-1}$ as $d(x,y)\to 0$.
	\end{proof}
	A direct consequence of Lemma \ref{lem: comparison of J(x,y) with unstable jac} is the following almost multiplicativity property
	\begin{lem}\label{lem: J multiplicative} Assume $x,y,z$ are placed in this order on a geodesic and $d(x,y),d(y,z)\ge 1$. Then
		\[J(x,z)\asymp J(x,y)J(y,z),\]
		where the implied constants are uniform.
	\end{lem}
	
	We have the following comparison property:
	\begin{lem}\label{lem: temperness J} There is a constant $C>0$ such that for any $x,y,z$ with $d(x,y)\ge 1$ and $d(x,z)\ge 1$:
		\[C^{-1}\mathrm{e}^{-Cd(y,z)}J(x,z)\le  J(x,y)\le C\mathrm{e}^{Cd(y,z)}J(x,z).\]\end{lem}
	\begin{proof}It is enough to prove the result for $d(y,z)\le 1$. Further, we can assume $d(x,y),d(x,z)\ge 1$ because otherwise the result follows from the fact that $J(x,y)\sim d(x,y)^{d-1}$ as the distance goes to $0$. The proof relies on Lemma \ref{lem: divergence geod} about the divergence of geodesics.
		
		\textbf{1.} We first assume that $y$ and $z$ both lie on the sphere $S(x,r)$, we write $y=\exp_x(rv)$ and $z=\exp_x(rw)$, with $v,w\in S_x\widetilde M$. Then,
		\[J(x,y)\asymp \exp\Big(\int_0^r \psi^u(\varphi_t(x,v))\mathrm{d}t\Big), \quad J(x,z)\asymp \exp\Big(\int_0^r \psi^u(\varphi_t(x,w))\mathrm{d}t\Big).\]By Lemma \ref{lem: divergence geod}, we have $d(\varphi_t(x,v),\varphi_t(x,w))\le C\mathrm{e}^{-c(r-t)}$, and because $\psi^u$ is H\"older continuous, we have
			\begin{align*}
			\Big|\int_0^r \psi_u(\varphi_t(x,v))\mathrm{d}t-\int_0^t\psi_u(\varphi_t(x,w))\mathrm{d}t\Big|& \lesssim \int_0^r d_{\rm Sas}(c(0,t);c(1,t))^{\alpha}\mathrm{d}t\\ & \lesssim \int_0^r \mathrm{e}^{-c\alpha(r-t)}\mathrm{d}t\lesssim 1.
		\end{align*}
		By taking exponentials we obtain
		\[\frac{J(x,z)}{J(x,y)}\asymp 1.\]
		\textbf{2.} If $y,z$ lie at distances $r,r'$ from $x$ (with $|r-r'|\le 1$), we replace $y=:\exp_x(rv)$ by $y':=\exp_x(r'v)$. Then, we have easily $J(x,y')\asymp J(x,y)$ because $y$ and $y'$ lie on the same geodesic issuing from $x$ and $d(y,y')\le 1$. Now, by the previous case we have $J(x,y')\asymp J(x,z)$ thus $J(x,y)\asymp J(x,z)$. This concludes the proof.
	\end{proof}
	
	It is convenient to introduce a slight modification of $J$ that enjoys better properties.
	\begin{defi}We set
		\[\widetilde J(x,y):=\left\{\begin{array}{ll} 1& \text{if $d(x,y)\le 1$} \\ J(x,y) & \text{otherwise}\end{array}\right..\]
	\end{defi}
	Then the following holds
	\begin{lem}\label{lem: properties tild J}For any $x,y,z$ aligned on a geodesic, we have
		\[\widetilde J(x,z)\asymp \widetilde J(x,y)\widetilde J(y,z).\]
		For any $x,y,z\in \widetilde M$ we have
		\[C^{-1}\mathrm{e}^{-Cd(y,z)}\widetilde J(x,z)\le  \widetilde J(x,y)\le C\mathrm{e}^{Cd(y,z)}\widetilde J(x,z).\]
	\end{lem}
	\begin{defi}\label{defi: x-j} Fix $o\in \widetilde M$. Let $x=\exp_o(r_x v),y=\exp_o(r_y w)\in \widetilde M$ satisfy $d(x,y)=t$. Let $j\in [-t,t]$ be such that $d(o,y)-d(o,x)\in [j,j+1]$. We define $r_j:=\frac{t-j}{2}$ and
		\begin{equation} \label{eq: defi x-j}x_{-j}:=\exp_o\big((r_x-r_j)v\big), \qquad y_j=\exp_o\big((r_y-r_{-j})w\big).\end{equation}
	\end{defi}
	A consequence of the Gromov-hyperbolicity property is the following Lemma.
	\begin{lem}\label{lem: xj y-j} With the notations of Definition \ref{eq: defi x-j}, we have
		\[d(x_{-j}, y_{+j})=\mathcal O(1).\]
	Moreover, $x_{-j}$ and $y_{j}$ lie at bounded distance from the geodesic segment $[xy]$.
	\end{lem}

\noindent	\begin{minipage}[c]{0.7\textwidth}
		\raggedright
	When $x$ and $y$ are as in Definition \ref{defi: x-j}, by combining Lemma \ref{lem: properties tild J} with Lemma \ref{lem: xj y-j}, we find 
	\[\widetilde J(x,x_{-j})\asymp \widetilde J(x,y_j), \quad \widetilde J(x_{-j},y)\asymp \widetilde J(y_{j},y).\]
	Moreover, 
	\[\widetilde J(x,y)\asymp \widetilde J(x,x_{-j})\widetilde J(x_{-j},y)\]
	and also
	\[\widetilde J(x,y)\asymp   \widetilde J(x,y_{j})\widetilde J(y_j,y).\]
	
	\end{minipage}%
	\hfill
	\begin{minipage}[c]{0.05\textwidth}\end{minipage}
	\hfill
	\begin{minipage}[c]{0.25\textwidth}
		\centering
		\includegraphics[width=\textwidth]{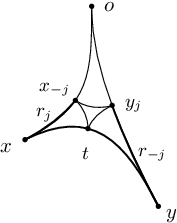}
	\end{minipage}
We end this subsection with the following Lemma.
	\begin{lem}\label{lem: swap integral}
		For any integrable function $h(x,y)$ on $\widetilde M\times \widetilde M$, one has
		\[\int_{\widetilde M}\int_{S(x,t)} h(x,y)\mathrm{d}\sigma_{S(x,t)}(y)\mathrm{dvol}(x)=\int_{\widetilde M}\int_{S(y,t)} h(x,y)\mathrm{d}\sigma_{S(y,t)}(x)\mathrm{dvol}(x).\]
	\end{lem}
	\begin{proof}Work in normal coordinates and use the fact that the Liouville measure is preserved by the geodesic flow. Further, exploit the symmetry property $J(x,y)=J(y,x)$. Alternatively, this quantity can be seen as the integral of $h$ over the submanifold $N=\{(x,y)\in \widetilde M\times \widetilde M~:~ d(x,y)=t\}$ with respect to the induced measure, computed in two different ways. This was pointed out to us by Louis-Brahim Beaufort. \end{proof}

	\subsection{Proof of Proposition \ref{prop: operator norm spherical mean}} In the following, we let $A(x,r,d)$ denote the annulus $\{y~:~ |d(x,y)-r|\le d\}$.

	\begin{proof} \textbf{1.} We first show that $\limsup \frac{1}{t}\log \|\mathcal L_t\|\le \delta_0$. Assume $t\ge 1$, and fix $\varepsilon>0$. By definition,
\[\|\mathcal L_t\|=\sup_{\|f\|=\|g\|=1} \langle \mathcal L_tf,g\rangle.\]
Take $f,g\in L^2(\widetilde M)$. Since $t\ge 1$, we integrate over pairs $(x,y)$ satisfying $d(x,y)\ge 1$, in particular $J(x,y)=\widetilde J(x,y)$. Thus, recalling \eqref{eq: alternative L_tf}, one has
		\[\langle \mathcal L_tf,g\rangle=\int_{\widetilde M}\int_{S(x,t)} \widetilde J(x,y)^{-1}g(x)f(y)\mathrm{d}\sigma_{S(x,t)}(y).\]
		Fix a point $o\in \widetilde M$, as in Lemma \ref{lem: xj y-j}. The idea is to separate the pairs $(x,y)$ according to the relative distances of $x$ and $y$ to the point $o$. 
		We define $\chi_j(x,y)=\mathbf 1_{d(y,o)-d(x,o)\in [j,j+1]}$. Then $1=\sum_{|j|\le t} \chi_j(x,y)$ and we can bound the above quantity by
		\[\langle \mathcal L_t f,g\rangle\le \sum_{|j|\le t}\int_{\widetilde M}\int_{S(x,t)} \widetilde J(x,y)^{-1}\chi_j(x,y)g(x)f(y)\mathrm{d}\sigma_{S(x,t)}(y).\]
		For $|j|\le t$, let $a_j(x,y)$ be a smooth positive function on $\widetilde M\times \widetilde M$, that will be fixed below. Using the elementary inequality $uv\le a_j u^2+\frac{1}{a_j}v^2$ and Lemma \ref{lem: swap integral}, we have
		\begin{multline} \label{eq: bound Lt sum over j}\langle \mathcal L_tf,g\rangle \le \sum_{|j|\le t}\int_{\widetilde M}\int_{S(x,t)} \frac{1}{a_j(x,y)}\chi_j(x,y)\widetilde J(x,y)^{-1}g(x)^2\mathrm{d}\sigma_{S(x,t)}(y)\mathrm{dvol}(x) \\ +\int_{\widetilde M}\int_{S(y,t)}a_j(x,y)\chi_j(x,y)\widetilde J(x,y)^{-1}f(y)^2\mathrm{d}\sigma_{S(y,t)}(x)\mathrm{dvol}(y).\end{multline}
		We set 
		\[U(x,j):= \int_{S(x,t)} \frac{1}{a_j(x,y)}\chi_j(x,y)\widetilde J(x,y)^{-1}\mathrm{d}\sigma_{S(x,t)}(y),\]
		and
		\[V(y,j):= \int_{S(y,t)}a_j(x,y)\chi_j(x,y)\widetilde J(x,y)^{-1}\mathrm{d}\sigma_{S(y,t)}(x).\]
		Taking the supremum over $\|f\|=\|g\|=1$ in \eqref{eq: bound Lt sum over j} leads to
		\begin{equation}\label{eq: Lt bound U V}
			\|\mathcal L_t\|\le  \sum_{|j|\le t}\Big( \sup_{x\in \widetilde M} U(x,j) +\sup_{y\in \widetilde M}V(y,j)\Big).
		\end{equation}
		We now choose 
		\[a_j(x,y)=\alpha_j\frac{\widetilde J\big(x_{-j},y\big)  }{\widetilde J\big(x,x_{-j}\big)},\]
		with $\alpha_j=\exp(-j\delta_0)$. We now estimate $U(x,j)$ and $V(x,j)$.
	
	\emph{Estimating $U(x,j)$.} By Lemma \ref{lem: temperness J}, we can replace the integral over a sphere by an integral over an annulus. Letting $\tilde \chi_j(x,y)=\mathbf 1_{d(o,y)-d(o,x)\in [j-1,j+2]}$ we have
	\[U(x,j)\lesssim   \int_{A(x,t,1)}\frac{1}{a_j(x,y)}\tilde \chi_j(x,y)\widetilde J(x,y)^{-1} \mathrm{dvol}(y).\]
	Moreover, by Lemma \ref{lem: xj y-j}, there is a uniform constant $C>0$ such that
		\[A(x,t,1)\cap \{y~:~ d(x,o)-d(x,y)\in [j-1,j+2]\}\subset A(x_{-j},r_j,C),\]
		where we recall $r_j$ and $x_{-j}$ have been defined in \eqref{defi: x-j}. Thus $U(x,j)$ is further bounded by
		\[U(x,j)\lesssim   \int_{A(x_{-j},r_j,C)} \frac{1}{a_j(x,y)} \tilde \chi_j(x,y) \widetilde J(x,y)^{-1}\mathrm{dvol}(y).\]
		By the definition of $a_j(x,y)$, we have
		\[\frac{1}{a_j(x,y)}\widetilde J(x,y)^{-1}=\alpha_j^{-1} \frac{\widetilde J(x,x_{-j})}{\widetilde J(x_{-j},y)\widetilde J(x,y)}.\]
	Since $x_{-j}$ lies at bounded distance from the geodesic segment $[xy]$ (for $y$ in the support of the integrand), Lemma \ref{lem: properties tild J} implies
	\[\widetilde J(x,y)\asymp \widetilde J(x,x_{-j})\widetilde J(x_{-j},y).\]
	We obtain
	\[\frac{1}{a_j(x,y)}\widetilde J(x,y)^{-1}\lesssim \alpha_j^{-1} \frac{1}{\widetilde J(x_{-j},y)^2}.\]
	In turn, by \eqref{eq: pressure as lim int over annulus},
	\[U(x,j)\lesssim \alpha_j^{-1} \int_{A(x_{-j},r_j,C)} \widetilde J(x_{-j},y)^{-2}\mathrm{dvol}(y)\le C_\varepsilon \alpha_j^{-1}\mathrm{e}^{(2\delta_0+\varepsilon)r_j},\]
	where we recall that $2\delta_0=\operatorname{Pr}(-2\psi^u)$. Since $\alpha_j=\mathrm{e}^{-j\delta_0}$ and $r_j=\frac{t-j}{2}$, we have eventually
	\begin{equation} \label{eq: U(x,j) bound} U(x,j)\le C_\varepsilon \mathrm{e}^{(\delta_0+\varepsilon)t}.\end{equation}
		
		\emph{Estimating $V(y,j)$.} This follows the same lines as above, but we use the alternative expression
		\[a_j(x,y)\asymp \alpha_j\frac{\widetilde J(y_{+j},y)}{\widetilde J(x,y_{+j})},\]
		which follows from Lemma \ref{lem: temperness J} and the fact that $d(x_{-j},y_{+j})=\mathcal O(1)$. One finds
		\begin{equation}\label{eq: V(y,j) bound}
			\begin{aligned}
			V(y,j)& \lesssim \alpha_j\int_{S(y_{+j},r_{-j})} \widetilde J(y_{+j},x)^{-2}\mathrm{d}\sigma_{S(y_{+j},r_{-j})}(x) \\ 
				& \lesssim \mathrm{e}^{\delta_0 j} \mathrm{e}^{(2\delta_0+\varepsilon)r_{-j}}\lesssim \mathrm{e}^{(\delta_0+\varepsilon)t}.
\end{aligned}
		\end{equation}
		Injecting \eqref{eq: U(x,j) bound},\eqref{eq: V(y,j) bound} in \eqref{eq: Lt bound U V} leads to
		\[\|\mathcal L_t\|\le C_\varepsilon t\mathrm{e}^{(\delta_0+\varepsilon)t},\]
		for $t\ge 1$. This holds for any $\varepsilon>0$, so
		\begin{equation}\label{eq: limsup Lt} \underset{t\to +\infty}{\limsup} \frac 1t\log \|\mathcal L_t\|\le \delta_0.\end{equation}
		
		\textbf{2.} We turn to the lower bound on $\|\mathcal L_t\|$. It is enough to exhibit a function $f\in L^2(\widetilde M)$ with unit norm such that $\|\mathcal L_t f\|$ is large enough. Fix $x_0\in \widetilde M$, and set
		\[f(x)=\mathbf 1_{d(x_0,x)\in [t-1,t+1]} J(x_0,x)^{-1}.\]
		Then, by \eqref{eq: pressure as lim int over annulus},
	\begin{equation} \label{eq: bound |f|} \|f\|_{L^2}^2=\int \mathbf 1_{d(x_0,x)\in [t-1,t+1]} J(x_0,x)^{-2}\mathrm{dvol}(x)\le C_\varepsilon \mathrm{e}^{(2\delta_0+\varepsilon)t}. \end{equation}
		Moreover, for $x\in B(x_0,1)$, one has
		\[\mathcal L_t f(x)=\int_{S(x,t)}J(x_0,y)^{-1}J(x,y)^{-1}\mathrm{d}\sigma_{S(x,t)}(y).\]
		By Lemma \ref{lem: temperness J}, since $d(x,x_0)\le 1$ we have for all $y\in S(x,t)$,
		\[J(x_0,y)^{-1}\gtrsim J(x,y)^{-1},\]
		so that
		\[\mathcal L_t f(x)\gtrsim \int_{S(x,t)}J(x,y)^{-2}\mathrm{d}\sigma_{S(x,t)}(y)\gtrsim \mathrm{e}^{(2\delta_0-\varepsilon)t}.\]
		This holds for all $x\in B(x_0,1)$, thus for large $t$ we have
		\begin{equation} \label{eq: lw bound Ltf}\|\mathcal L_tf\|_{L^2}\gtrsim \mathrm{e}^{(2\delta_0-\varepsilon)t}.\end{equation}
		Combining \eqref{eq: bound |f|} with \eqref{eq: lw bound Ltf} leads to
		\[\|\mathcal L_t\|\ge \frac{\|\mathcal L_tf \|_{L^2}}{\|f\|_{L^2}}\ge c_\varepsilon\mathrm{e}^{(\delta_0-\varepsilon)t}.\]
		This holds for any $\varepsilon>0$ hence
		\begin{equation} \label{eq: liminf Lt}\underset{t\to +\infty}{\liminf} \frac 1t\log \|\mathcal L_t\|\ge \delta_0.\end{equation}
		Together, \eqref{eq: limsup Lt} and \eqref{eq: liminf Lt} conclude the proof.
		 \end{proof}
		 
		 \subsubsection{The case of the hyperbolic space}
		When $M$ has constant curvature $-1$, its universal cover is the hyperbolic $d$-space, on which there is an explicit link between $\mathcal L_t$ and the half-wave operator. By \cite[Chapter II, Exercises F1-F2]{Helgason1984} (see the solutions of the exercises at the end of the book, in particular eq. (4) p. 540), on the hyperbolic space $\mathbb H^d$, there is a dimensional constant $c_d$ such that
		\begin{equation} \label{eq: L_t hyperbolic} \mathcal L_t=c_d\frac{\sinh(t)}{\sinh^{d-1}(t)}\int_0^t \cos\big(s(\Delta_{\mathbb H^d}-(\tfrac{d-1}{2})^2)^{\frac 12}\big) (\cosh(t)-\cosh(s))^{\frac{d-3}{2}}\mathrm{d}s.\end{equation}
	Since the spectrum of $\Delta_{\mathbb H^d}$ is the half-line $[(\tfrac{d-1}{2})^2,+\infty)$, we get by the spectral theorem:
	\[\|\mathcal L_t\|_{L^2(\mathbb H^d)}=c_d \frac{\sinh(t)}{\sinh^{d-1}(t)}\int_0^t (\cosh(t)-\cosh(s))^{\frac{d-3}{2}}\mathrm{d}s.\]
	Then, a short computation shows that 
	\begin{equation}\|\mathcal L_t\|_{L^2(\mathbb H^d)} \underset{t\to +\infty}\sim c_d t\exp(-(\tfrac{d-1}{2})t), \end{equation}
	for some other dimensional constant $c_d$. This is a continuous analogue of Nica's estimate \eqref{eq: nica result}. It would be interesting to obtain an analogous asymptotic for the operator norm of $\pi_*\mathrm{e}^{-t(\widetilde X+\widetilde F)}\pi^*:L^2(\widetilde M)\to L^2(\widetilde M)$ when $M$ is only assumed to be Anosov and $\widetilde F$ is the lift of some H\"older continuous potential $F\in C^\alpha(SM)$.

	\subsection{Propagation on modes of order $\lesssim h^{-1}$}
 	The operator $\mathcal L_t$ involves averaging functions over the whole fiber $S_x\widetilde M$, in other words projecting function on the space of $0$-modes on the fibers. We now project functions on a space of functions oscillating on scales $\gtrsim h$ on the fibers.
	
	We consider the operator $g(h^2\Delta_{\widetilde{\mathbb V}})\mathrm{e}^{-t\widetilde X}g(h^2\Delta_{\widetilde{\mathbb V}})$ acting on $L^2(S\widetilde M)$, for a compactly supported function $g$. This operator acts by filtering out high Fourier modes (of order $\gtrsim 1/h$) in the fibers, then propagating by the flow, and finally by applying the filter again. This operator is similar in spirit to $\operatorname{Op}_h(\mathbf 1_{|\xi|_{\widetilde{\mathbb V}}\le 1})\mathrm{e}^{-t\widetilde X}\mathrm{Op}_h(\mathbf 1_{|\xi|_{\widetilde{\mathbb V}}\le 1})$, but has the advantage of a simple geometric definition that doesn't resort to semiclassical quantization. Moreover, the spherical mean decomposition of functions in $L^2(S\widetilde M)$ allows to seamlessly estimate the operator norm of $g(h^2\Delta_{\widetilde{\mathbb V}})\mathrm{e}^{-t\widetilde X}g(h^2\Delta_{\widetilde{\mathbb V}})$ in terms of that of the spherical mean operator. 

\subsubsection{Operator norm of $g(h^2\Delta_{\widetilde{\mathbb V}}) \mathrm{e}^{-t\widetilde X} g(h^2\Delta_{\widetilde{\mathbb V}})$} The following proposition makes precise the intuition that Fourier modes escape to infinity during the evolution.
\begin{prop}\label{prop: fourier flow fourier} Let $d=\operatorname{dim}(M)$. Then, for any Schwartz function $g$, we have 
	\[\big\|g(h^2\Delta_{\widetilde{\mathbb V}}) \mathrm{e}^{-t\widetilde X}g(h^2\Delta_{\widetilde{\mathbb V}})\big\|_{L^2(S\widetilde M)}\lesssim (1+h^{-(d-1)})\|\mathcal L_t\|_{L^2(\widetilde M)}.\]
Moreover, by Proposition \ref{prop: operator norm spherical mean}, the right-hand side is smaller than $C_\varepsilon (1+h^{-(d-1)})\mathrm{e}^{(\delta_0+\varepsilon)t}$.
\end{prop}	
	\begin{rem}It is not surprising that we have to reach times $t\gtrsim |\log h|$ for the estimate to be nontrivial. Indeed, if $u$ has little oscillations then $\mathrm{e}^{-tX}u$ concentrates on vertical Fourier modes of order $\lesssim \mathrm{e}^{\gamma_{\rm max}t}$, where $\gamma_{\rm max}$ is the maximal rate of expansion of the flow. Thus we should not observe decay of the norm while $\mathrm{e}^{\gamma_{\rm max} t}\ll h^{-1}$ .
	\end{rem}
	\begin{proof}
The proof relies on the decomposition of $L^2$-functions in spherical harmonics presented in  \S\ref{subsubsec: Fourier analysis SM}, together with pointwise estimates on eigenfunctions. We tacitly view functions on $\widetilde M$ as functions on $S\widetilde M$ that are constant over the fiber of the projection $S\widetilde M\to \widetilde M$, and vice versa. By definition, since $g(h^2\Delta_{\widetilde{\mathbb V}})$ is self-adjoint, we have
\[\big\|g(h^2\Delta_{\widetilde{\mathbb V}}) \mathrm{e}^{-t\widetilde X}g(h^2\Delta_{\widetilde{\mathbb V}})\big\|_{L^2(S\widetilde M)}=\sup_{\|u\|=\|v\|=1} \langle  \mathrm{e}^{-t\widetilde X}g(h^2\Delta_{\widetilde{\mathbb V}})u, g(h^2\Delta_{\widetilde{\mathbb V}})v\rangle.\]
Let $u,w\in L^2(S\widetilde M)$ satisfy $\|u\|_{L^2}=\|w\|_{L^2}=1$. As in \S \ref{subsec: Fourier}, we expand
\[u(x,v)=\sum_{m\ge 0}u_m(x)\omega_m(x,v), \qquad w(x,v)=\sum_{n\ge 0}w_n(x)\omega_n(x,v).\]
Then,
\[g(h^2\Delta_{\widetilde{\mathbb V}}) u=\sum_{m\ge 0} g(h^2\lambda_m)u_{m}\omega_{m}, \qquad g(h^2\Delta_{\widetilde{\mathbb V}})w=\sum_{n\ge 0}g(h^2\lambda_n)w_{n}\omega_{n}.\]
We set
\[U(x):=\Big(\sum_{m\ge 0} |u_m(x)|^2\Big)^{\frac 12}, \qquad W(x):=\Big(\sum_{n\ge 0}|w_n(x)|^2\Big)^{\frac 12}.\]
Note that $\|U\|_{L^2}=\|W\|_{L^2}=1$. By Cauchy--Schwarz inequality, we have
\[\big|g(h^2\Delta_{\widetilde{\mathbb V}}) u(x,v)\big|\le \Big\|\sum_{m\ge 0} g(h^2\lambda_m)^2|\omega_{m}|^2\Big\|_{L^{\infty}(S^{d-1})}^{\frac 12}\cdot U(x),\]
and similarly for $w$. Then,
\[\langle  \mathrm{e}^{-t\widetilde X}g(h^2\Delta_{\widetilde{\mathbb V}})u, g(h^2\Delta_{\widetilde{\mathbb V}})w\rangle\le\Big\|\sum_{m\ge 0} g(h^2\lambda_m)^2|\omega_{m}|^2\Big\|_{L^{\infty}(S^{d-1})} \cdot\langle \mathrm{e}^{-t\widetilde X}U,W\rangle .\]
Because the functions $U,W$ are constant over the fibers of $S\widetilde M\to \widetilde M$ and $\|U\|_{L^2}=\|W\|_{L^2}=1$, we have
\[\langle  \mathrm{e}^{-t\widetilde X} U,W\rangle=\langle \mathcal L_t U,W\rangle\le \|\mathcal L_t\|_{L^2(\widetilde M)}.\]
Eventually,
\[\big\|g(h^2\Delta_{\widetilde{\mathbb V}}) \mathrm{e}^{-t\widetilde X}g(h^2\Delta_{\widetilde{\mathbb V}})\big\|_{L^2(S\widetilde M)}\le \|\mathcal L_t\|\cdot \Big\|\sum_{m\ge 0} g(h^2\lambda_m)^2|\omega_{m}|^2\Big\|_{L^{\infty}(S^{d-1})}. \]
To control the right-hand side, we recall the pointwise Weyl's law of Avakumovi{\'c} \cite{avakumovic}: 
\[\sum_{\lambda_m\le \lambda} |\omega_m(x,v)|^2\lesssim 1+\lambda^{\frac{d-1}2}.\]
A summation by parts then gives
\[\sum_{m\ge 0} g(h^2\lambda_m)^2|\omega_{m}(x,v)|^2=-\int_0^{+\infty}\Big(\sum_{\lambda_m\le \lambda} |\omega_m(x,v)|^2\Big) \partial_\lambda\big[g(h^2\lambda)^2\big] \mathrm{d}\lambda,\]
which can in turn be bounded by a constant times
\[\int_0^{+\infty} (1+\lambda^{\frac{d-1}{2}})  h^2|(g^2)'|(h^2\lambda)\mathrm{d}\lambda\lesssim 1+h^{-(d-1)}.\]
This concludes the proof.
\end{proof}

	\section{Probabilistic low frequency estimates}\label{sec: proba estimates}
	
	\subsection{Strong convergence of representations}\label{subsec: strong conv} We refer to the recent surveys \cite{Magee_survey,vanhandel2025strongconvergencephenomenon} for a smooth introduction to strong convergence, and its application to spectral gaps.
	
	\begin{defi} Let $\Gamma$ be a discrete group. The \emph{regular representation} of $\Gamma$ is the unitary representation $\lambda_{\Gamma}:\Gamma\to U(\ell^2(\Gamma))$ defined by
		\[\lambda_{\Gamma}(\gamma)(x_g)_{g\in \Gamma}=(x_{\gamma^{-1} g})_{g\in \Gamma}.\]
		In other words a function $f:g\mapsto f(g)$ is mapped to $\lambda_{\Gamma}(\gamma)f:g\mapsto f(\gamma^{-1}g)$.
	\end{defi}
	\begin{defi}\label{defi: strong conv} A sequence of unitary representations $(\rho_n,V_n)$ of a discrete group $\Gamma$ \emph{strongly converges} to the regular representation $(\lambda_{\Gamma},\ell^2(\Gamma))$ if for all $w\in \mathbb{C}[\Gamma]$, 
		\begin{equation} \label{eq: def strong conv}\|\rho_n(w)\|_{V_n}\underset{n\to +\infty}{\longrightarrow} \|\lambda_{\Gamma}(w)\|_{\ell^2(\Gamma)},\end{equation}
		where the norms are operator norms. If the sequence $(\rho_n)$ is random, we say that $(\rho_n)$ strongly converges  \emph{asymptotically almost surely} (or \emph{a.a.s.}) if \eqref{eq: def strong conv} holds in probability.
		
	\end{defi}
Equation \eqref{eq: def strong conv} simply means that for any finitely supported map $\gamma\in \Gamma\mapsto a_\gamma\in \mathbb C$, one has
		\begin{equation} \label{eq: alternative def strong conv} \Big\|\sum_{\gamma\in \Gamma} a_\gamma\rho_n(\gamma)\Big\|_{V_n}\underset{n\to +\infty}{\longrightarrow} \Big\|\sum_{\gamma\in \Gamma} a_\gamma\lambda_{\Gamma}(\gamma)\Big\|_{\ell^2(\Gamma)}.\end{equation}
	The following theorem was recently proved by Magee--Puder--van Handel.
	\begin{thm}[{\cite[Theorem 1.1]{magee2025strongconvergenceuniformlyrandom}}]\label{thm: strong con mpvh} Let $\Gamma$ be the fundamental group of a closed surface of genus $\ge 2$. If $\phi_n\in \operatorname{Hom}(\Gamma,S_n)$ is picked uniformly at random, then the sequence of representations $\rho_n:=\operatorname{std}_{n-1}\circ \phi_n$ strongly converges a.a.s. to the regular representation. Here $\operatorname{std}_{n-1}$ is the standard irreducible $(n-1)$-dimensional representation of $S_n$. \end{thm}
	By matrix amplification (see e.g. \cite[\S 3]{Magee_survey}), one can replace the scalar coefficients $a_\gamma$ in \eqref{eq: alternative def strong conv} by matrices, and by approximating compact operators by finite rank operators, one can improve \eqref{eq: alternative def strong conv} to
	\begin{prop}[{\cite[Proposition 3.3]{Magee_survey}}] \label{prop: amplified} Assume that $(\rho_n,V_n)$ strongly converges to the regular representation (a.a.s.). Let $\mathcal H$ be a separable Hilbert space, and denote by $\mathcal K(\mathcal H)$ the space of compact operators on $\mathcal H$. Let $\varepsilon>0$. Then, for any finitely supported map $\gamma\in \Gamma\mapsto a_\gamma\in \mathcal K(\mathcal H)$, one has (a.a.s.)
		\begin{equation} \label{eq: strong conv amplified} \Big\|\sum_{\gamma\in \Gamma} a_\gamma\otimes \rho_n(\gamma)\Big\|_{\mathcal H\otimes V_n}\le \Big\|\sum_{\gamma\in \Gamma} a_\gamma\otimes \lambda_{\Gamma}(\gamma)\Big\|_{\mathcal H\otimes V_n}+\varepsilon.\end{equation}
	\end{prop}
	The interest of Proposition~\ref{prop: amplified} is that in many situations, the limit operator appearing on the right-hand side of \eqref{eq: strong conv amplified} is easier to understand. We will use Proposition~\ref{prop: amplified} to show the following. Recall that
	\[\widehat f(-\mathrm{i}\mathbf X_{\mathcal E})=\int f(t)\mathrm{e}^{-t\mathbf X_{\mathcal E}}\mathrm{d}t.\]
	\begin{prop} \label{prop: Mh etX Mh}Assume that the sequence $(\rho_n,V_n)$ strongly converges a.a.s. Then, for any smooth real compactly supported function $f$ vanishing near $0$ and any $\varepsilon>0$, we have a.a.s., for $\mathcal E=\mathcal E_{\rho_n}$:
		\begin{align*}\label{eq: lim of norm} \big\|g(h^2\Delta_{\mathbb V}^E)\widehat f(-\mathrm{i}\mathbf X_{\mathcal E})g(h^2\Delta_{\mathbb V}^E)\big\|_{L^2(SM,\mathcal E)}\le \big\|g(h^2\Delta_{\widetilde{\mathbb V}}) \widehat f(-\mathrm{i}\widetilde X)g(h^2\Delta_{\widetilde{\mathbb V}})\big\|_{L^2(S\widetilde M)}+\varepsilon.\end{align*} 
	\end{prop}
	Proposition \ref{prop: Mh etX Mh} will be proved in \S \ref{subsec: proof prop Mh etx Mh} below after some preliminary work. Using Proposition \ref{prop: fourier flow fourier}, Proposition~\ref{prop: operator norm spherical mean} and Proposition~\ref{prop: fourier flow fourier}, we obtain the following Corollary.
	
	\begin{cor}Assume that the sequence $(\rho_n,V_n)$ strongly converges a.a.s. to the regular representation. Let $f_T$ be compactly supported on $[T,T+1]$, and let $h\le 1$. Then, for any $\varepsilon>0$, we have a.a.s., for $\mathcal E=\mathcal E_{\rho_n}$:
	\end{cor}
	\[\big\|g(h^2\Delta_{\mathbb V}^E)\widehat f_T(-\mathrm{i}\mathbf X_{\mathcal E})g(h^2\Delta_{\mathbb V}^E)\big\|_{L^2(SM,\mathcal E)} \le C_\varepsilon h^{-(d-1)}\mathrm{e}^{(\delta_0+\varepsilon)T} \|f_T\|_{\infty}.\]

	\subsection{Riemannian coverings, permutation representations of $\Gamma$ and unitary flat bundles} \label{subsec: link cover bundle} Let $M_n\to M$ be a finite degree covering. Recall that there is a smooth splitting
	\[C^\infty(M_n)=C^\infty(M)\oplus C^\infty_{\rm new}(M_n).\]
	As explained in the introduction, the covering $M_n$ is associated with a homomorphism $\phi_n\in \operatorname{Hom}(\Gamma,S_n)$. By composing $\phi_n$ with the standard irreducible $(n-1)$-dimensional representation of $S_n$, we obtain a representation $\rho_n:\Gamma\to U(V_n^0)$, where $V_n^0=\mathbf 1^{\perp}\subset \mathbb{C}^n$ is the subspace of vectors with $0$ mean, that is
	\[V_n^0=\Big\{(x_1,\ldots,x_n)\in \mathbf C^n~:~ \sum x_i=0\Big\}.\]
	The point is that the representation $\rho_n$ gives rise to an isomorphism 
	\begin{equation} \label{eq: iso Cinf new and C inf Erho} C^\infty_{\rm new}(M_n)\simeq C^\infty(M,E_{\rho_n}),\end{equation}
	where $E_{\rho_n}$ is the unitary flat bundle over $M$ associated with $\rho_n$, which is defined as the quotient
	\begin{equation} \label{eq: defi bundle associated}E_{\rho_n}:=\Gamma\backslash (\widetilde M\times V_n^0), \qquad \gamma.(x,z)=(\gamma.x, \rho_n(\gamma)z).\end{equation}
	The bundle $E_{\rho_n}\to M$ lifts to a bundle $\mathcal E_{\rho_n}\to SM$, whose smooth sections correspond to functions in $C^\infty(SM_n)$ with $0$-average over the fibers of the projection $SM_n\to SM$. Since $\Gamma$ acts on $S\widetilde M$ by $\gamma.(x,v)=(\gamma x,\mathrm{d}\gamma(x)v)$, we have a natural identification
	\[\mathcal E_{\rho_n}=\Gamma\backslash (S\widetilde M\times V_n^0).\]
	
	\subsection{Sections of flat bundles as equivariant functions on $S\widetilde M$}
	From now on, we work in the following abstract framework: $(\rho_n,V_n)$ is a sequence of random representations of $\Gamma$ that strongly converges to the  regular representation a.a.s., in the sense of Definition~\ref{defi: strong conv}. The main example of interest if provided by Theorem~\ref{thm: strong con mpvh}. For an arbitrary representation $\rho_n$ of $\Gamma$, we denote by $E_{\rho_n}$ the unitary flat bundle over $M$ defined as in \eqref{eq: defi bundle associated}, and by $\mathcal E_{\rho_n}$ the lift of this bundle to $SM$.
	\subsubsection{Function spaces} Here we denote by $\mathbf x=(x,v)$ the points in $SM$. Any smooth section of $\mathcal E_{\rho_n}$ corresponds to a smooth function $f:S\widetilde M\to V_n$ satisfying the equivariance property
	\[f(\gamma \mathbf x)=\rho_n(\gamma)f(\mathbf x).\]
	The space of such smooth functions is denoted by $C^\infty_{\rho_n}(S\widetilde M,V_n)$. Let $F$ be a fundamental domain for the action of $\Gamma$ on $S\widetilde M$ (that may be obtained by lifting a fundamental domain for the action of $\Gamma$ on $\widetilde M$). We define a norm
	\[\|f\|_{L^2_{\rho_n}(S\widetilde M,V_n)}:=\int_F \|f(\mathbf x)\|_{V_n}^2\mathrm{d}\mathbf x,\]
	and let $L^2_{\rho_n}(S\widetilde M,V_n)$ be the completion of $C^\infty_{\rho_n}(S\widetilde M,V_n)$ with respect to this norm. The isomorphism $C^\infty(SM,\mathcal E_{\rho_n})\simeq C^\infty_{\rho_n}(S\widetilde M,V_n)$ extends to an isomorphism of Hilbert spaces
	\begin{equation} \label{eq: iso hilbert spaces} L^2(SM,\mathcal E_{\rho_n})\simeq L^2_{\rho_n}(S\widetilde M,V_n).\end{equation}

	\subsection{Proof of Proposition~\ref{prop: Mh etX Mh}}\label{subsec: proof prop Mh etx Mh} From now on, to ease notations, we drop the $\rho_n$ subscripts at some places.	Under the isomorphism \eqref{eq: iso hilbert spaces} the operator $\mathrm{e}^{-t\mathbf X_{\mathcal E}}$ identifies with $\mathrm{e}^{-t\widetilde X}$. Similarly, $\Delta_{\mathbb V}^E$ identifies with $\Delta_{\widetilde{\mathbb V}}$. It follows that under this isomorphism, the operator 
	\[g(h^2\Delta_{\mathbb V}^E)\widehat f(-\mathrm{i}\mathbf X_{\mathcal E})g(h^2\Delta_{\mathbb V}^E):L^2(SM,E_{\rho_n})\to L^2(SM,E_{\rho_n})\]
	identifies with the operator
	\begin{equation} \label{eq: recall form of the op}g(h^2\Delta_{\widetilde{\mathbb V}}) \widehat f(-\mathrm{i}\widetilde X)g(h^2\Delta_{\widetilde{\mathbb V}}):L^2_{\rho_n}(\widetilde M,V_n)\to L^2_{\rho_n}(\widetilde M,V_n). 
	\end{equation}
	Beware that we are not done because we want an operator acting on $L^2(\widetilde M,\mathbb{C})$. To apply the strong convergence techniques, it is important to deal with compact operators, which will be ensured by the following Lemma.
	\begin{lem}\label{lem: K(w,z) continuous}Assume that $f$ is compactly supported and vanishes near $t=0$. Then the operator \eqref{eq: recall form of the op} has smooth kernel. \end{lem} \begin{proof}The proof relies on wave-front set arguments; we refer to \cite[Chapter 4]{Lefeuvre} for an introduction. Note that $h$ is fixed and we work with the classical wave-front set.
		
		In restriction to $S_xM$, the operator $g(h^2\Delta_{\widetilde{\mathbb V}})$ acts as $g(h^2\Delta_{S^{d-1}})$. Since $\Delta_{S^{d-1}}$ is an elliptic operator on $S^{d-1}$ and $g$ is compactly supported, we know that the operator $g(h^2\Delta_{S^{d-1}})$ is smoothing, viewed as an operator acting on $\mathcal D'(S^{d-1})$. This follows for example from a parametrix construction. Thus, schematically, the operator $g(h^2\Delta_{\widetilde{\mathbb V}})$ has kernel
		\[K((x,v);(y,w))=\delta(x-y) \varphi(v,w),\]
		where $\varphi$ is the kernel of $g(h^2\Delta_{S^{d-1}})$, which is a smooth function. By this we mean that for any $u_1,u_2\in C^\infty_{\rm comp}(S\widetilde M)$, we have
		\[\langle g(h^2\Delta_{\widetilde{\mathbb V}})u_1,u_2\rangle=\int_{\widetilde M} \Big(\int_{S_x\widetilde M\times S_x\widetilde M} \varphi(v_1,v_2)u_1(x,v_1)u_2(x,v_2) \mathrm{d}v_1\mathrm{d}v_2\Big)\mathrm{dvol}(x),\]
		where we identify $S_x\widetilde M\simeq S^{d-1}$ to evaluate $\varphi(v_1,v_2)$.  
		It follows that for any $u\in \mathcal E'(S\widetilde M)$ (the space of compactly supported distributions), we have
		\begin{equation} \label{eq: WF A10} \begin{aligned} \operatorname{WF}\big(g(h^2\Delta_{\widetilde{\mathbb V}}) u\big)\subset\big\{(\mathbf x,\xi)\in \operatorname{WF}(u)~:~ \xi(\mathbb V)=\{0\}\big\}\end{aligned}.\end{equation}
		Now, by \cite[Lemma 4.3.4]{Lefeuvre} (see also \cite[Exercise 6.2.10]{Lefeuvre}), we have for any $u\in \mathcal E'(SM)$:
		\begin{equation} \label{eq: WF phi chap X} \operatorname{WF} \big(\widehat f(-\mathrm{i}\widetilde X)u\big)\subset \left\{(\mathbf x,\xi)~:~ \begin{array}{l}\xi(X(\mathbf x))=0, \exists t\in \operatorname{Supp}(f), \\  (\varphi_{-t}(\mathbf x),\mathrm{d}\varphi_{-t}^{-\top}\xi)\in \operatorname{WF}(u) \end{array}\right\}.\end{equation}
		By combining \eqref{eq: WF A10} with \eqref{eq: WF phi chap X}, we obtain that for any $u\in \mathcal E'(S\widetilde M)$, one has
		\[\operatorname{WF} \big(g(h^2\Delta_{\widetilde{\mathbb V}})\widehat f(-\mathrm{i}\widetilde X)g(h^2\Delta_{\widetilde{\mathbb V}}) u\big)\subset \left\{(\mathbf x,\xi)~:~ \begin{array}{l}
			\exists\, t\neq 0,
			\xi(\mathbb V)=\{0\},\\
			\xi(X(\mathbf x))=0,
			\mathrm{d}\varphi_{-t}^{-\top}\xi(\mathbb V)=\{0\}
		\end{array}\right\}.\]
		Now, any $\xi$ that satisfies $\xi(\mathbb V)=\xi(\mathbb RX)=\xi(\mathrm{d}\varphi_{-t}\mathbb V)=\{0\}$ for some $t\neq 0$ must be zero, because 
		\[\mathbb RX\oplus \mathbb V\oplus \mathrm{d}\varphi_{-t}\mathbb V=T(SM).\]
		By dimension considerations, it is enough to show that for $t\neq 0$, $v\in (\mathbb RX\oplus \mathbb V)\cap \mathrm{d}\varphi_{-t}\mathbb V$ implies $v=0$. Consider such a $v$. Since the flow preserves the sum $\mathbb H\oplus \mathbb V$, we have $v\in (\mathbb RX\oplus \mathbb V)\cap (\mathbb H\oplus \mathbb V)$, and then $v\in \mathbb V$ because $\mathbb RX\oplus \mathbb H\oplus \mathbb V=T(SM)$. This shows $v=0$ because $\mathbb V\cap (\mathrm{d}\varphi_{-t}\mathbb V)=\{0\}$ for $t\neq 0$, by Klingenberg's result \cite{Klingenberg1974}.
		It follows that we have a continuous operator
		\[g(h^2\Delta_{\widetilde{\mathbb V}}) \widehat f(-\mathrm{i}\widetilde X)g(h^2\Delta_{\widetilde{\mathbb V}}):\mathcal E'(S\widetilde M)\to C^\infty(S\widetilde M).\]
		By \cite[Lemma 3.2.3]{Lefeuvre}, the kernel of this operator is smooth.\end{proof}

		\begin{proof}[End of the proof of Proposition~\ref{prop: Mh etX Mh}] We denote by $K(\mathbf x,\mathbf x')$ the kernel of the operator~$g(h^2\Delta_{\widetilde{\mathbb V}}) \widehat f(-\mathrm{i}\widetilde X)g(h^2\Delta_{\widetilde{\mathbb V}})$, which is smooth by Lemma~\ref{lem: K(w,z) continuous}. Denote by $a_\gamma$ the operator $L^2(F)\to L^2(F)$ with kernel $a_\gamma(\mathbf x,\mathbf x'):=K(\mathbf x,\gamma \mathbf x')$. Note that there is another isomorphism of Hilbert spaces 
		\begin{equation}\label{eq: deuxieme iso} L^2_{\rho_n}(S\widetilde M,V_n)\simeq L^2(F,V_n)\simeq L^2(F)\otimes V_n,\end{equation}
		where the first isomorphism is the restriction $f\mapsto f_{|F}$ which has inverse 
		\[s\mapsto \sum_{\gamma\in \Gamma} \rho_n(\gamma)s(\gamma^{-1}\bullet)\mathbf 1_{\gamma F}.\]
		Under the isomorphism \eqref{eq: deuxieme iso}, the operator \eqref{eq: recall form of the op} is unitarily conjugated to
		\[\sum_{\gamma\in \Gamma } a_\gamma\otimes \rho_n(\gamma):L^2(F)\otimes V_n\to L^2(F)\otimes V_n.\]
		Since $K(\mathbf x,\mathbf x')$ is continuous and $F$ is bounded, the operator $a_\gamma$ is Hilbert--Schmidt hence compact on $L^2(F)$. Since the geodesic flow has finite speed of propagation, there are finitely many $\gamma\in \Gamma$ such that $a_\gamma\neq 0$. We can thus invoke Proposition~\ref{prop: amplified} to deduce that for any $\varepsilon>0$, one has a.a.s.
		\begin{equation} \label{eq: strong convergence applied}\Big\|\sum_{\gamma\in \Gamma } a_\gamma\otimes \rho_n(\gamma)\Big\|_{L^2(F)\otimes V_n}\le \Big\|\sum_{\gamma\in \Gamma } a_\gamma\otimes \lambda_{\Gamma}(\gamma)\Big\|_{L^2(F)\otimes\ell^2(\Gamma)}+\varepsilon.\end{equation}
		Here $(\lambda_{\Gamma},\ell^2(\Gamma))$ is the regular representation of $\Gamma$. Finally, there is an isomorphism $L^2(F)\otimes \ell^2(\Gamma)\simeq L^2(S\widetilde M)$ that maps $u\otimes \delta_\gamma$ to the function $u(\gamma^{-1}\bullet )\cdot\mathbf 1_{\gamma F}$. Under this isomorphism, the operator on the right-hand side of \eqref{eq: strong convergence applied} identifies with 
		\[g(h^2\Delta_{\widetilde{\mathbb V}}) \widehat f(-\mathrm{i}\widetilde X)g(h^2\Delta_{\widetilde{\mathbb V}}):L^2(S\widetilde M)\to L^2(S\widetilde M).\]
		Thus we can recast \eqref{eq: strong convergence applied} into
		\[\|g(h^2\Delta_{\mathbb V}^E)\widehat f(-\mathrm{i}\mathbf X_{\mathcal E})g(h^2\Delta_{\mathbb V}^E)\|_{L^2(SM,\mathcal E_{\rho_n})}\le \big\|g(h^2\Delta_{\widetilde{\mathbb V}}) \widehat f(-\mathrm{i}\widetilde X)g(h^2\Delta_{\widetilde{\mathbb V}})\big\|_{L^2(S\widetilde M)}+\varepsilon,\]
		asymptotically almost surely. This concludes the proof.
	\end{proof}
	
	\section{Proof of the main Theorem} From now on we only work on $\mathcal M=SM$, where points are denoted with the letter $x$.
	
	\subsection{Anisotropic Sobolev spaces}\label{sec: aniso}
	 We recall the construction of the anisotropic Sobolev spaces introduced by Faure and Sj\"ostrand in \cite{Faure2011}. The starting point is the observation that for $f\in C^\infty(\mathcal M)$, the backward evolution $\mathrm{e}^{-tX}f=f\circ \varphi_{-t}$ produces fast oscillations in the stable direction, and slow oscillations in the unstable direction. This motivates considering an \emph{anisotropic} norm that allows irregularity in the stable direction but penalizes irregularity in the unstable direction.

	Recall the Anosov splitting $T\mathcal M=E_0\oplus E_u\oplus E_s$ from \eqref{eq: Anosov splitting}. Consider the dual distribution
	\[T^*\mathcal M=E_0^*\oplus E_s^*\oplus E_u^*,\]
	defined by $E_0^*(E_s\oplus E_u)=0$, $E_s^*(E_s\oplus E_0)=0$ and $E_u^*(E_u\oplus E_0)=0$. Then, for all $t\ge 0$,
	\[\forall \xi\in E_s^*, \ |(\mathrm{d}\varphi_t)^{-\top} \cdot \xi|\le C \mathrm{e}^{-ct}|\xi|, \qquad \forall \xi\in E_u^*, \ |(\mathrm{d}\varphi_{-t})^{-\top} \cdot \xi|\le C \mathrm{e}^{-ct}|\xi|.\]
	The subbundle $E_0^*$ is known as the \emph{trapped set}. We let $H$ denote the Hamiltonian vector field associated with the Hamiltonian $p(x,\xi)=\xi(X(x))$. Then, $H$ generates the flow
	\[\Phi_t(x,\xi)=\big(\varphi_t(x),(\mathrm{d}\varphi_t(x))^{-\top}\xi\big),\]
	which is the symplectic lift of $\varphi_t$ to $T^*\mathcal M$.
	
	 For the next proposition, we quote \cite{Lefeuvre}. The H\"ormander class of symbols $S_{1^-}^s(T^*\mathcal M)$ is introduced in Appendix~\ref{appendix}.
	
	\begin{prop}[Construction of an escape function, {\cite[Lemma 9.1.9]{Lefeuvre}}]\label{prop: construction escape} Let $N_0$ be a conical neighborhood of the \emph{trapped set} $E_0^*\subset T^*\mathcal M$. There is a smooth metric $g$ on $T^*\mathcal M$, a constant $C>0$, and an order function $m\in C^\infty(T^*\mathcal M,[-1,1])$, such that away from the zero section, $m\equiv 1$ in a conic neighborhood of $E_s^*$, $m\equiv -1$ in a conic neighborhood of $E_u^*$; and such that the function
		\[G_m:=m(x,\xi)\log|\xi|_g,\]
		satisfies $\mathrm{e}^{sG_m}\in S^s_{1^-}(T^*\mathcal M)$ for all $s\in \mathbb R$, and $G_m$ decreases along the flow lines of $\Phi_t$:
		\begin{itemize}
			\item $HG_m(x,\xi)\le 0$ for $|\xi|_g\ge 1$.
			\item $HG_m(x,\xi)\le -C$ for $|\xi|_g\ge 1$ outside of the conical neighborhood $N_0$ of $E_0^*$.
		\end{itemize}
		The function $G_m$ is called an \emph{escape function}.\end{prop}

	Let $\mathcal E\to \mathcal M$ be a unitary flat bundle. In Appendix~\ref{appendix}, we introduce a semiclassical quantization procedure $\operatorname{Op}_h^{\mathcal E}$ mapping symbols $a\in C^\infty(T^*\mathcal M)$ to operators acting on sections of $\mathcal E$. We refer to the Appendix for the properties of the $\operatorname{Op}_h^{\mathcal E}$-calculus.
	\begin{prop}[Anisotropic Sobolev spaces] For any $s\in \mathbb R$, there is a constant $h_0=h_0(s)>0$, such that for any $0<h<h_0$, and for any unitary flat bundle $\mathcal E\to \mathcal M$, the operator $\operatorname{Op}_h^{\mathcal E}(\mathrm{e}^{sG_m}):C^\infty(\mathcal M,\mathcal E)\to C^\infty(\mathcal M,\mathcal E)$ is invertible. We set
		\[\|u\|_{\mathcal H_h^s}:=\big\|\operatorname{Op}_h^{\mathcal E}(\mathrm{e}^{sG_m}) u\big\|_{L^2(\mathcal M,\mathcal E)}.\]
		The space $\mathcal H_h^s(\mathcal M,\mathcal E)$ is then defined as the closure
		\[\mathcal H_h^s(\mathcal M,\mathcal E):=\overline{C^\infty(\mathcal M,\mathcal E)}^{\|\bullet\|_{\mathcal H_h^s}}.\]
		Observe that $C^\infty(\mathcal M,\mathcal E)\subset \mathcal H_h^s(\mathcal M,\mathcal E)\subset \mathcal D'(\mathcal M,\mathcal E)$, with dense and continuous inclusions.\end{prop}
	\begin{proof}This follows from Corollary~\ref{cor: inverse op (e^sg)} in Appendix~\ref{appendix}.
	\end{proof}
	The unbounded operator $\mathbf X_{\mathcal E}:\mathcal H_h^s(\mathcal M,\mathcal E)\to \mathcal H_h^s(\mathcal M,\mathcal E)$ has domain 
	\[\mathcal D(\mathbf X_{\mathcal E})=\big\{u\in \mathcal H_h^s(\mathcal M,\mathcal E)~:~ \mathbf X_{\mathcal E}u\in \mathcal H_h^s(\mathcal M,\mathcal E)\big\},\] where $\mathbf X_{\mathcal E}u$ is understood in the sense of distributions.

	\begin{prop}[Strong continuity of $\mathrm{e}^{-t\mathbf X_{\mathcal E}}$] \label{prop: strong continuity group} The propagator $\mathrm{e}^{-t\mathbf X_{\mathcal E}}$ extends to a strongly continuous semigroup on $\mathcal H_h^s(\mathcal M,\mathcal E)$, for $t\ge 0$. There are constants $C_s,\nu_s,h_0(s)>0$ depending only on $s$ such that for $0<h<h_0(s)$, we have
		\[\big\|\mathrm{e}^{-t\mathbf X_{\mathcal E}}\big\|_{\mathcal H_h^s(\mathcal M,\mathcal E)}\le C_s \mathrm{e}^{\nu_s t}, \qquad \forall t\ge 0\]
	\end{prop}

	\begin{proof}By the semi-group property, it is enough to prove the result for $0\le t\le t_0$ for some $t_0>0$. Since $\mathrm{e}^{t\mathbf X_{\mathcal E}}$ is unitary on $L^2(\mathcal M,\mathcal E)$, we equivalently need to show that there is a constant $C$ such that for all $0\le t\le t_0$,
		\[\|\mathrm{e}^{t\mathbf X_{\mathcal E}}\operatorname{Op}_h^{\mathcal E}(\mathrm{e}^{sG_m})\mathrm{e}^{-t\mathbf X_{\mathcal E}}\operatorname{Op}_h^{\mathcal E}(\mathrm{e}^{sG_m})^{-1}\|_{L^2(\mathcal M,\mathcal E)}\le C.\]
		The operator norm above is bounded by
		\[ \|\mathrm{e}^{t\mathbf X_{\mathcal E}}\operatorname{Op}_h^{\mathcal E}(\mathrm{e}^{sG_m})\mathrm{e}^{-t\mathbf X_{\mathcal E}}\operatorname{Op}_h^{\mathcal E}(\mathrm{e}^{-sG_m})\|\cdot \|\operatorname{Op}_h^{\mathcal E}(\mathrm{e}^{-sG_m})^{-1} \operatorname{Op}_h^{\mathcal E}(\mathrm{e}^{sG_m})^{-1}\|.\]
		The proposition then follows from Corollary~\ref{cor: inverse op (e^sg)} and Proposition~\ref{prop: Egorov theorem}.

	\end{proof}
	
	The anisotropic Sobolev spaces $\mathcal H_h^s(\mathcal M,\mathcal E)$ can be used as in \cite{Faure2011} to obtain the following
	\begin{prop}[Meromorphic continuation of the resolvent, {\cite[Theorem 1.5]{Faure2011}}] Let $\mathcal E\to \mathcal M$ be a unitary flat bundle over $\mathcal M$. Then, the family of operators $(\mathbf X_{\mathcal E}+z)^{-1}:C^\infty(\mathcal M,\mathcal E)\to \mathcal D'(\mathcal M,\mathcal E)$, initially defined for $\operatorname{Re}(z)>0$ by formula \eqref{eq: def resolvent}, continues meromorphically to $\mathbb{C}$. The poles are of finite rank and are called Pollicott--Ruelle resonances.
	\end{prop}
	
	The set of Pollicott--Ruelle resonances of $\mathbf X_{\mathcal E}$ is denoted by $\operatorname{Res}(\mathbf X_{\mathcal E})$. Given a smooth splitting $\mathcal E=\mathcal E_1\oplus \mathcal E_2$, one has
	\[\operatorname{Res}(\mathbf X_{\mathcal E})=\operatorname{Res}(\mathbf X_{\mathcal E_1})\cup \operatorname{Res}(\mathbf X_{\mathcal E_2}),\]
	where the union is taken with multiplicity.

	\begin{rem}Some authors prefer working with the meromorphic continuation of the family of operators $(-\mathbf X_{\mathcal E}+z)^{-1}$, defined for $\operatorname{Re}(z)>0$ by
		\[(-\mathbf X_{\mathcal E}+z)^{-1}=\int_0^{+\infty}\mathrm{e}^{-zt}\mathrm{e}^{t\mathbf X_{\mathcal E}}\mathrm{d}t.\]
		In our case, since $-\mathrm{i}\mathbf X_{\mathcal E}$ is self-adjoint on $L^2(\mathcal M,\mathcal E)$, we have for any $f,g\in C^\infty(\mathcal M,\mathcal E)$:
		\[\big\langle (-\mathbf X_{\mathcal E}+z)^{-1} f,g\big\rangle_{\mathcal D'(\mathcal M,\mathcal E),C^\infty(\mathcal M,\mathcal E)}=\big\langle (\mathbf X_{\mathcal E}+\bar z)^{-1}g,f\big\rangle_{\mathcal D'(\mathcal M,\mathcal E),C^\infty(\mathcal M,\mathcal E)}^*,\]
		where the $*$ denotes complex conjugation. Thus the sets of resonances associated with these two operators are related to each other via complex conjugation. In the case where $\mathcal E\to \mathcal M$ comes from a real representation (\emph{e.g.} when $\mathcal E$ is associated with a covering $M'\to M$), the resonance set is actually invariant under complex conjugation, so this distinction is immaterial.
		
	\end{rem}
	
	Theorem \ref{thm: extension resolvent} follows from the more precise result:
	\begin{thm}\label{thm: extension resolvent Hs} Let $(M,g)$ be a closed Anosov manifold, with fundamental group $\Gamma$. Let $\{\rho_n\}_{n=1}^{+\infty}$ be a sequence of random unitary representations of $\Gamma$, which strongly converges to $(\lambda_\Gamma,\ell^2(\Gamma))$ a.a.s. Let $\mathcal E_n:=\mathcal E_{\rho_n}$ denote the associated unitary flat bundles over $\mathcal M$. Then, for any compact subset $\mathcal K\subset \{\operatorname{Re}(z)>\delta_0\}$, there are constants $h=h(\mathcal K)>0$ and $C(\mathcal K)>0$ such that a.a.s., for all $z\in \mathcal K$, the operator $(\mathbf X_{\mathcal E_n}+z):\mathcal D(\mathbf X_{\mathcal E_n})\subset \mathcal H_h^s(\mathcal M,\mathcal E_n)\to \mathcal H_h^s(\mathcal M,\mathcal E_n)$ is invertible, with bounded inverse:
		\[\forall z\in \mathcal K, \quad \big\|(\mathbf X_{\mathcal E_n}+z)^{-1}\big\|_{ \mathcal H_h^s(\mathcal M,\mathcal E_n)\to \mathcal H_h^s(\mathcal M,\mathcal E_n)}\le C(\mathcal K).\]\end{thm}
	
		\begin{rem}Faure and Tsujii \cite{faure2022fractal,faure2024micro} introduced some different anisotropic Sobolev spaces $\mathcal H_W(\mathcal M,\mathcal E)$ to study the Ruelle spectrum. These depend on some escape function $W\in C^\infty(T^*\mathcal M)$ and are defined using wave-packet transform. These spaces allow for finer results; for example when the bundle $\mathcal E\to \mathcal M$ is fixed, one has some uniform resolvent bounds in a vertical strip left to the imaginary axis---away from the resonances---, as opposed to the polynomial bounds in $|\operatorname{Im}(z)|$ obtained in \cite{nonnenmacher2015decay} using the spaces $\mathcal H_h^s(\mathcal M,E)$. Since Theorem \ref{thm: extension resolvent Hs} is only concerned with the resolvent $(\mathbf X_{\mathcal E}+z)^{-1}$ for $z$ in a compact set this distinction might appear irrelevant, however it might be useful for future applications to obtain the analogue of Theorem \ref{thm: extension resolvent Hs} when working in the spaces $\mathcal H_W(\mathcal M,\mathcal E)$ instead. The result should follow by running again the proof in the present paper with the spaces $\mathcal H_W$, relying on semiclassical analysis with wave-packets instead of Weyl quantization. \end{rem}
	
	\subsection{Approximate resolvent} Let $\chi$ be a smooth function, compactly supported on $[-1,2]$, such that $\chi\equiv 1$ on $[0,1]$. Motivated by the expression \eqref{eq: def resolvent}, we introduce an approximate resolvent
	\[\mathbf Q(z):=\int_{0}^{+\infty}\chi(t/T)\mathrm{e}^{-zt}\mathrm{e}^{-t\mathbf X_{\mathcal E}}\mathrm{d}t,\]
	Then, $\mathbf Q(z)$ is continuous $\mathcal H_h^s(\mathcal M,\mathcal E)\to \mathcal H_h^s(\mathcal M,\mathcal E)$. Indeed, by Proposition~\ref{prop: strong continuity group},
	\begin{equation} \label{eq: estimate Q(z)}\|\mathbf Q(z)\|_{\mathcal H_h^s(\mathcal M,\mathcal E)}\le \int_{0}^{2T} \mathrm{e}^{-\operatorname{Re}(z)t}C_s\mathrm{e}^{\nu_st}\mathrm{d}t\le C_s \mathrm{e}^{2(\nu_s-\operatorname{Re}(z))T}. \end{equation}
	Note that the estimate on $\|\mathbf Q(z)\|$ is uniform with respect to $\mathcal E$. One has
	\[\mathbf Q(z)(\mathbf X_{\mathcal E}+z)=\operatorname{Id}+\mathbf R(z)\quad \text{on $\mathcal D(\mathbf X_{\mathcal E})$},\]
	and
	\[(\mathbf X_{\mathcal E}+z)\mathbf Q(z)=\operatorname{Id}+\mathbf R(z)\quad \text{on $\mathcal H_h^s(\mathcal M,\mathcal E)$},\]
	where $\mathbf R(z):\mathcal H_h^s(\mathcal M,\mathcal E)\to \mathcal H_h^s(\mathcal M,\mathcal E)$ is a holomorphic family of operators given by
	\[\mathbf R(z)=\frac 1T\int_0^{+\infty} \chi'(t/T) \mathrm{e}^{-zt}\mathrm{e}^{-t\mathbf X_{\mathcal E}}\mathrm{d}t.\]
	The gain is that $\chi'(\bullet/T)$ is supported on $[T,2T]$ with $T$ to be taken large, which is crucial to apply the results of  \S\ref{sec: spherical mean} and  \S\ref{sec: proba estimates} on the large time propagation by the flow. If one can prove that $\|\mathbf R(z)\|_{\mathcal H_h^s}\le \frac 12$ for all $z\in \mathcal K$, then the operator $\operatorname{Id}+\mathbf R(z)$ can be inverted using a Neumann series, leading to
	\[(\operatorname{Id}+\mathbf R(z))^{-1}\mathbf Q(z)(\mathbf X_{\mathcal E}+z)=(\mathbf X_{\mathcal E}+z)\mathbf Q(z)(\operatorname{Id}+\mathbf R(z))^{-1}=\operatorname{Id}_{\mathcal H_h^s(\mathcal M,\mathcal E)}.\]
	
	Thus, our main theorem will follow from
	\begin{prop} Given a sequence of random representations $(\rho_n,V_n)$ of $\Gamma$ which strongly converges a.a.s. to the regular representation, we have asymptotically almost surely,
		\[\forall z\in \mathcal K, \quad \|\mathbf R(z)\|_{\mathcal H_h^s(\mathcal M,\mathcal E_{\rho_n})}=\Big\|\frac 1T\int_0^{+\infty} \chi'(t/T) \mathrm{e}^{-zt}\mathrm{e}^{-t\mathbf X_{\mathcal E_{\rho_n}}}\mathrm{d}t\Big\|_{\mathcal H_h^s(\mathcal M,\mathcal E_{\rho_n})}\le \frac 12.\]\end{prop}
	
	\subsection{Reduction to a finite number of $z$.} The strong convergence techniques only allow dealing with a \emph{finite} number of operators. We show that it is enough to prove that $\|\mathbf R(z)\|_{\mathcal H_h^s}\le \frac 14$ for a finite number of $z\in \mathcal K$ to obtain $\|\mathbf R(z)\|<\frac 12$ for all $z\in \mathcal K$. This procedure is analogous to those performed in \cite{HM,MCN}.
	\begin{lem}There is a constant $C_s$ such that for any $z,z'$ with $\operatorname{Re}(z),\operatorname{Re}(z')\ge \delta_0$,
		\[\|\mathbf R(z)-\mathbf R(z')\|_{\mathcal H_h^s}\le |z-z'|\mathrm{e}^{C_sT}.\]
	\end{lem}
	\begin{proof}
		Since $|\mathrm{e}^{-zt}-\mathrm{e}^{-z't}|\le |z-z'|\mathrm{e}^{-\delta_0t}$, it follows from the triangle inequality that
		\[\|\mathbf R(z)-\mathbf R(z')\|_{\mathcal H_h^s}\le |z-z'| \frac 1T\int |\chi'(t/T)| \mathrm{e}^{-\delta_0t} \|\mathrm{e}^{-t\mathbf X_{\mathcal E}}\|\mathrm{d}t.\]
		Since $\|\mathrm{e}^{-t\mathbf X_{\mathcal E}}\|_{\mathcal H_h^s}\le C_s\mathrm{e}^{\nu_s t}$ (recall Proposition~\ref{prop: strong continuity group}) and $\chi$ is supported in $[1,2]$ we infer
		\[\|\mathbf R(z)-\mathbf R(z')\|_{\mathcal H_h^s}\le C_s|z-z'| \mathrm{e}^{2(\nu_s-\delta_0)T}.\]
	\end{proof}
	Consequently, if $\|\mathbf R(z)\|_{ \mathcal H_h^s(\mathcal M,\mathcal E)}\le \frac 14$ then $\|\mathbf R(z')\|_{ \mathcal H_h^s(\mathcal M,\mathcal E)}\le \frac 12$ for $z'$ in a small neighborhood of $z$, which is independent of the bundle $\mathcal E$, but depends on $T$ and $s$. Since $\mathcal K$ is compact, it follows that we only have to show that $\|\mathbf R(z)\|\le \frac 14$ for $z$ in a \emph{finite} subset $\mathcal I\subset \mathcal K$.
	
	So, we are left with showing:
	\begin{prop}\label{prop: R(z) for finite number of z} There is a choice of $(T,s,h)$ such that for any finite subset $\mathcal I\subset \mathcal K$, one has a.a.s.
		\[\forall z\in \mathcal I, \ \|\mathbf R(z)\|_{\mathcal H_h^s}\le \frac 14.\] \end{prop}

	\subsection{Averaging along the flow} \label{sec: averaging along the flows} Before turning to the proof of Proposition \ref{prop: R(z) for finite number of z}, we show that integrating along the flow lines damps functions that oscillate on scales $\lesssim h$ in the flow direction; this relies on the microlocal ellipticity of the operator $-\mathrm{i}h\mathbf X_{\mathcal E}$ in a conical neighborhood of the flow direction. This property will be used to control the contributions to the resolvent from points in a neighborhood of the trapped set. 
\begin{prop}\label{prop: averaging along $X$} Let $f$ be a smooth function compactly supported on a sufficiently small neighborhood of $0$. Let $a,b\in S^{m_1}(T^*\mathcal M)$ and $b\in S^{m_2}(T^*\mathcal M)$ with $a$ or $b$ supported on the cone $\{\xi~:~\xi(X)\ge \frac{|\xi|}{2}\ge \frac 12\}$. Then for any unitary flat bundle $\mathcal E\to \mathcal M$, one has
	\[\operatorname{Op}_h^{\mathcal E}(a) \widehat f(-\mathrm{i}\mathbf X_{\mathcal E})\operatorname{Op}_h^{\mathcal E}(b)=\mathcal O_{a,b}(h^\infty)_{L^2\to L^2}.\]
	The error is uniform with respect to $\mathcal E$.
\end{prop}

\begin{proof}We perform a reduction similar to that in the proof of Proposition~\ref{prop: inserting g(h2lap) highdim bundle} to reduce the proof to the case of the $\operatorname{Op}_h^{\mathcal M}$ quantization. In turn, we can work in local charts where the geodesic vector field is given by $\partial_{x_1}$, and all we have to show is that for any $a\in S^{m_1}(T^*\mathbb{R}^d)$ and $b\in S^{m_2}(T^*\mathbb{R}^d)$ with either $a$ or $b$ supported in the cone $\{\xi~:~ \xi_1 \ge \frac{|\xi|}{2}\ge \frac 12\}$, one has
	\[\operatorname{Op}_h^{\rm w}(a) \widehat f(-\mathrm{i}\partial_{x_1})\operatorname{Op}_h^{\rm w}(b)=\mathcal O(h^\infty)_{L^2\to L^2}.\]
	Assume \emph{e.g.} that $b$ satisfies the support condition above (the case of $a$ is treated similarly). Then,
	\[ \widehat f(-\mathrm{i}\partial_{x_1})\operatorname{Op}_h^{\rm w}(b)u(x)=\frac{1}{(2\pi h)^n}\int f(t) \mathrm{e}^{\frac{\mathrm{i}}{h}\langle x-te_1-y,\xi\rangle} b\big(\frac{x-te_1+y}{2},\xi\big) u(y)\mathrm{d}y\mathrm{d}\xi\mathrm{d}t.\]
	We can perform some integrations by parts in the $t$ variable, since $(-\mathrm{i}\partial_t) \mathrm{e}^{-\frac{\mathrm{i}}{h} t\xi_1}=-\frac{\xi_1}{h}$ we infer that for any $N$,
	\[\int f(t) \mathrm{e}^{-\frac{\mathrm{i}}{h} t\xi_1} b\big(\frac{x-te_1+y}{2},\xi\big)\mathrm{d}t=\int \mathrm{e}^{-\frac{\mathrm{i}}{h}t\xi_1} (-\frac{h}{\xi_1})^N (-\mathrm{i}\partial_t)^N\Big[f(t)b\big(\frac{x-te_1+y}{2},\xi\big)\Big]\mathrm{d}t.\]
	This rewrites by Leibniz formula:
	\[(\frac{ \mathrm{i}h}{\xi_1})^N\int \mathrm{e}^{-\frac{\mathrm{i}}{h}t\xi_1} \sum_{k=0}^N\binom{N}{k}\frac{(-1)^k}{2^k} f^{(N-k)}(t) (\partial_{x_1}^kb)\big(\frac{x-te_1+y}{2},\xi\big)\mathrm{d}t.\]
	Let us introduce
	\[b_{N,k}:=\xi_1^{-N}\partial_{x_1}^kb, \qquad b_{N,k}^t:=b_{N,k}(x-te_1,\xi).\]
	Observe that $b_{N,k}\in S^{m_2-N}(T^*\mathbb{R}^d)$ due to the support condition on $b$. Since 
	\[\mathrm{e}^{-t\partial_{x_1}}\operatorname{Op}_h^{\rm w}(b_{N,k})=\operatorname{Op}_h^{\rm w}(b_{N,k}^t)\mathrm{e}^{-t\partial_{x_1}},\] we deduce from the composition formula for Weyl quantization that
	\[\operatorname{Op}_h^{\rm w}(a)\operatorname{Op}_h^{\rm w}(b_{N,k}^t)=\operatorname{Op}_h^{\rm w}(a\# b_{N,k}^t),\]
	with $a\# b_{N,k}^t\in S^{m_1+m_2-N}(T^*\mathbb{R}^d)$, in particular for $N$ large enough, by the $L^2$-continuity theorem we have
	\[\big\|\operatorname{Op}_h^{\rm w}(a\# b_{N,k}^t)\big\|_{L^2\to L^2}=\mathcal O_{a,b,N,k}(1).\]
	Eventually
	\[\operatorname{Op}_h^{\rm w}(a) \widehat f(-\mathrm{i}\partial_{x_1})\operatorname{Op}_h^{\rm w}(b)=\sum_{k=0}^N c_{N,k}\int f^{(N-k)}(t)\operatorname{Op}_h^{\rm w}(a\# b_{N,k}^t)\mathrm{e}^{-t\partial_{x_1}}\mathrm{d}t,\]
	where $c_{N,k}$ are numerical coefficients. Since $\mathrm{e}^{-t\partial_{x_1}}$ is unitary on $L^2$ we obtain that for $N>m_1+m_2$ the operator norm is
	\[\Big\|\operatorname{Op}_h^{\rm w}(a) \widehat f(-\mathrm{i}\partial_{x_1})\operatorname{Op}_h^{\rm w}(b)\Big\| \le C_{N,a,b} \|f\|_{C^N} h^N .\]
	The constant $C_{N,a,b}$ depends on a finite number (depending on $N$) of seminorms of $a$ and $b$.
\end{proof}

We deduce the following result
\begin{prop}\label{prop: average along X long time} Fix $K\ge 1$, $C_0>0$, $h\le 1$ and let $T\le K|\log h|$. Let $f$ be a smooth function supported on $[-T,T]$, with $\|f\|_{C^N}=\mathcal O_N(\mathrm{e}^{C_0T})$. Then, if $a,b\in S^0(T^*\mathcal M)$, with either $a$ or $b$ supported on the cone $\{\xi~:~\xi(X)\ge \frac{|\xi|}{2}\ge \frac 12\}$, we have
	\[\operatorname{Op}_h^{\mathcal E}(a \mathrm{e}^{sG_m})\widehat f(-\mathrm{i}\mathbf X_{\mathcal E})\operatorname{Op}_h^{\mathcal E}(b\mathrm{e}^{-sG_m})=\mathcal O_{K,C_0,s}(h^\infty)_{L^2\to L^2}.\]
	The error is uniform in $f$ if the implied constants in the estimate $\|f\|_{C^N}=\mathcal O_N(\mathrm{e}^{C_0T})$ remain bounded. \end{prop}
\begin{proof}Assume for example that $b$ satisfies the support condition. Let $\omega\in C^\infty_{\rm comp}(\mathbb{R})$ satisfy $\sum \omega(t-n)\equiv 1$. Fix a small $t_0>0$. The operator under consideration can be expressed as
	\[\sum_n\operatorname{Op}_h^{\mathcal E}(a \mathrm{e}^{sG_m})\Big[\int f(nt_0+t)\omega(\frac t{t_0})\mathrm{e}^{-(nt_0+t)\mathbf X_{\mathcal E}}\mathrm{d}t\Big]\operatorname{Op}_h^{\mathcal E}(b \mathrm{e}^{-sG_m}).\]
	Each summand can be written as
	\[\big[\operatorname{Op}_h^{\mathcal E}(a \mathrm{e}^{sG_m})\mathrm{e}^{-nt_0\mathbf X_{\mathcal E}}\operatorname{Op}_h^{\mathcal E}(\mathrm{e}^{sG_m})^{-1}\big]\circ \operatorname{Op}_h^{\mathcal E}(\mathrm{e}^{sG_m})\int f(nt_0+t)\omega(\frac{t}{t_0})\mathrm{e}^{-t\mathbf X_{\mathcal E}}\mathrm{d}t\operatorname{Op}_h^{\mathcal E}(b \mathrm{e}^{-sG_m}).\]
	By Proposition~\ref{prop: strong continuity group} and Corollary~\ref{cor: inverse op (e^sg)}, the norm of the first factor is controlled by $\mathcal O(\mathrm{e}^{\nu_s nt_0})$, and by Proposition~\ref{prop: averaging along $X$} the norm of the second factor is $\mathcal O_{t_0,s,m}(\|f\|_{C^N}h^N)$ (here we use the fact that $\mathrm{e}^{\pm sG_m}\in S^{s+}(T^*\mathcal M)$). Summing over all $n$ with $|nt_0|\le T$, one finds
	\[\big\|\operatorname{Op}_h^{\mathcal E}(a \mathrm{e}^{sG_m})\widehat f(-\mathrm{i}\mathbf X_{\mathcal E})\operatorname{Op}_h^{\mathcal E}(b \mathrm{e}^{-sG_m})\big\|_{L^2\to L^2}\le C_{N,s,t_0} \|f\|_{C^N} T\mathrm{e}^{\nu_s T} h^{N}.\]
	By the assumptions on $f$ and $T$ we have $\|f\|_{C^N}T\mathrm{e}^{\nu_s T}=\mathcal O_{N,s}(h^{-C_{K,s}})$, which allows to conclude by taking $N$ large enough.
\end{proof}

	\subsection{Proof of Proposition~\ref{prop: R(z) for finite number of z}} We begin by dividing the time interval $[T,2T]$ into subintervals of small size. Fix $t_0>0$ small enough and let $1=\sum_{n\in \mathbf Z} \omega(t-n)$ be a smooth partition of unity with $\omega\in C^\infty_{\rm comp}([-2,2],[0,1])$. Define 
	\begin{equation} \label{eq: def phi n t}\varphi_{n,T}(t)=\omega(t/t_0)\chi'((t+nt_0)/T) \mathrm{e}^{zt}.\end{equation}
	Then we can write
	\begin{equation} \label{eq: R(z) as split sum} \mathbf R(z)=\sum_{n\in \mathbb N}\mathrm{e}^{-znt_0}\int_{\mathbb R} \varphi_{n,T}(t)\mathrm{e}^{-(t+nt_0)\mathbf X_{\mathcal E}}\mathrm{d}t.\end{equation}
	The functions $\varphi_{n,T}$ are supported in $[-2t_0,2t_0]$, and are uniformly bounded in $C^\ell$ norm for all $\ell$ (independently of $n,T\ge 1$), provided $z$ remains in a compact set. We can therefore forget the dependence in $n,T$ and we will show

	\begin{prop}\label{prop: long time aniso} Consider a subset $\mathcal F\subset C^\infty_{\rm comp}((-1,1),\mathbb R)$ of functions with uniformly bounded derivatives, i.e.
		\[\forall \ell\ge 0, \quad \sup_{\varphi\in \mathcal F} \|\varphi\|_{C^\ell}<+\infty.\]
		Let $\varepsilon>0$ be small and $s>0$ be large enough. One can successively chose $K=K(s,\mathcal F,\varepsilon)$ large enough and $h=h(s,\mathcal F,K,\varepsilon)$ small enough such that the following holds. Let $T=K|\log h|$, and assume $T\le nt_0\le 2T$. Let $\varphi\in \mathcal F$. Then, asymptotically almost surely as $k\to +\infty$, we have, for $\mathcal E=\mathcal E_{\rho_k}$:
		\[\big\|\widehat \varphi(-\mathrm{i}\mathbf X_{\mathcal E})\mathrm{e}^{-nt_0\mathbf X_{\mathcal E}}\big\|_{\mathcal H_h^s(\mathcal E_{\rho_k})}\le \mathrm{e}^{(\delta_0+\varepsilon)nt_0}.\]
	\end{prop}
	We introduce a smooth partition $1=a_1+a_2+a_3$ of the phase space $T^*\mathcal M$, where:
	\begin{enumerate}
		\item $a_1$ is supported on $\{|\xi|\le 2\}$ (low frequency region).
		\item $a_2$ is supported on a conical neighborhood of $\{|\xi|\ge 1\}\cap E_0^*$ (high frequency, near the trapped set), say
		\[\operatorname{Supp}a_2\subset \Big\{(x,\xi)\in T^*\mathcal M~:~ |\xi(X(x))|\ge \frac{|\xi|}{2}\ge \frac 12\Big\}.\] 
		\item $a_3$  is supported on a conical subset of $\{|\xi|\ge 1\}$ away from the trapped set (high frequency, away from the trapped set), say
		\[\operatorname{Supp}a_3\subset \Big\{(x,\xi)\in T^*\mathcal M~:~ |\xi(X(x))|\le \frac{|\xi|}{4}, ~ |\xi|\ge 1\Big\}.\]
	\end{enumerate}
	Moreover, we take $a_2,a_3$ to be $0$-homogeneous on $\{|\xi|\ge 2\}$. For each $i,j\in\{1,2,3\}$, we define operators
	\[\mathbf A_{i,j}(n):=\operatorname{Op}_h^{\mathcal E}(a_i\mathrm{e}^{sG_m})\widehat \varphi(-\mathrm{i}\mathbf X_{\mathcal E})\mathrm{e}^{-nt_0\mathbf X_{\mathcal E}}\operatorname{Op}_h^{\mathcal E}(a_j\mathrm{e}^{-sG_m}).\]
	Next, we define $\mathbf N(n):= \sum_{(i,j)}\|\mathbf A_{i,j}(n)\|_{L^2\to L^2}$, so that
	\[\big\|\widehat \varphi(-\mathrm{i}\mathbf X_{\mathcal E})\mathrm{e}^{-nt_0\mathbf X_{\mathcal E}}\big\|_{\mathcal H_h^s}\le 3\mathbf N(n).\]
	We will obtain estimates on $\mathbf N(n)$ from estimates on $\mathbf N(n-1)$. This relies on different Lemmata. The first one is a probabilistic estimate for the low frequency contributions.
	\begin{lem}There is a function $(s,\varepsilon,K)\mapsto h_0(s,\varepsilon,K)>0$ such that for any $\varepsilon>0$ and $T\le K|\log h|$, and any $n,h,s,z$  with $h<h_0(s,\varepsilon,K)$ we have a.a.s.
		\[\|\mathbf A_{1,1}(n)\|_{L^2\to L^2}\le h^{-d}\mathrm{e}^{(\delta_0+\varepsilon) nt_0},\]
		uniformly in $T\le nt_0\le 2T$.\end{lem}
	\begin{proof}By Proposition~\ref{prop: inserting g(h2lap) highdim bundle}, and because $\mathrm{e}^{-t\mathbf X_{\mathcal E}}$ is unitary on $L^2(\mathcal M,\mathcal E)$, we have
		\[\mathbf A_{1,1}(n)=\operatorname{Op}_h^{\mathcal E}(\mathrm{e}^{sG_m}a_1) g(h^2\Delta_{\mathbb V}^E)\widehat \varphi(-\mathrm{i}\mathbf X_{\mathcal E})\mathrm{e}^{-nt_0\mathbf X_{\mathcal E}} g(h^2\Delta_{\mathbb V}^E)\operatorname{Op}_h^{\mathcal E}(a_1\mathrm{e}^{-sG_m})+\mathcal O_{s}(h^\infty).\]
		Here we recall that $g\equiv 1$ near $0$, so that $g(h^2\Delta_{\mathbb V}^E)$ localizes on Fourier modes of order $\lesssim h^{-1}$ in the vertical fibers. In turn, since $\mathrm{e}^{\pm sG_m}a_1\in C^\infty_{\rm comp}(T^*\mathcal M)\subset S^0_1(T^*\mathcal M)$, we obtain from Proposition~\ref{prop: L2 continuity}:
		\[\|\mathbf A_{1,1}(n)\|\le C_{s} \big\|g(h^2\Delta_{\mathbb V}^E)\widehat \varphi(-\mathrm{i}\mathbf X_{\mathcal E})\mathrm{e}^{-nt_0\mathbf X_{\mathcal E}}g(h^2\Delta_{\mathbb V}^E)\big\|_{L^2\to L^2}+\mathcal O_{s}(h^\infty).\]
		We now invoke Proposition~\ref{prop: Mh etX Mh}. For any $\varepsilon>0$ we have a.a.s:
		\[\big\|g(h^2\Delta_{\mathbb V}^E)\widehat \varphi(-\mathrm{i}\mathbf X_{\mathcal E})\mathrm{e}^{-nt_0\mathbf X_{\mathcal E}}g(h^2\Delta_{\mathbb V}^E)\big\|_{L^2\to L^2}\le C_{s,\varepsilon} h^{-(d-1)}\mathrm{e}^{(\delta_0+\varepsilon)nt_0}.\]
		It remains to take $h$ small enough to ensure $C_{s,\varepsilon}\le h^{-1}$ and $\mathcal O_s(h^\infty)\le h^{-d}\mathrm{e}^{(\delta_0+\varepsilon)nt_0}$.
		
	\end{proof}
	
	The next two Lemma are deterministic and hold for any unitary flat bundle over $\mathcal M$.
	\begin{lem} There is a function $h\mapsto h_0(s,K,N,\mathcal F)$ such that if $T=K|\log h|$ and $0<h<h_0(s,K,N)$, one has, uniformly in $n\le 2T/t_0$:
		\[\|\mathbf A_{i,2}(n)\|\le h^N,\qquad \|\mathbf A_{2,i}(n)\|\le h^N.\]
	\end{lem}
	\begin{proof} This follows directly from Proposition~\ref{prop: average along X long time}. \end{proof}

	\begin{lem} There is a constant $C_1>0$ which is independent of $s$ such that for all $i\in \{1,2,3\}$ and $0<h<h_0(s)$:
		\[\max\big(\|\mathbf A_{i,3}(n)\|,\|\mathbf A_{3,i}(n)\|\big)\le C_1 \mathrm{e}^{-C_0s}\cdot \sum_{j=1}^3\|\mathbf A_{i,j}(n-1)\|.\]
	\end{lem}
	\begin{proof}Let us deal with $\mathbf A_{i,3}(n)$, the case of $\mathbf A_{3,i}(n)$ is analogous. Using the semigroup property, we have
		\[\|\mathbf A_{i,3}(n)\|\le \Big\|\operatorname{Op}_h^{\mathcal E}(\mathrm{e}^{sG_m}a_i)\widehat \varphi(-\mathrm{i}\mathbf X_{\mathcal E})\mathrm{e}^{-(n-1)t_0\mathbf X_{\mathcal E}} \operatorname{Op}_h^{\mathcal E}(\mathrm{e}^{sG_m})^{-1}\Big\|\]\[\times \big\|\operatorname{Op}_h^{\mathcal E}(\mathrm{e}^{sG_m})\mathrm{e}^{-t_0\mathbf X_{\mathcal E}}\operatorname{Op}_h^{\mathcal E}(a_3\mathrm{e}^{-sG_m})\big\|.\]
		By Corollary~\ref{cor: inverse op (e^sg)}, the first norm is controlled by
		\[\Big\|\operatorname{Op}_h^{\mathcal E}(\mathrm{e}^{sG_m}a_i)\widehat \varphi(-\mathrm{i}\mathbf X_{\mathcal E})\mathrm{e}^{-(n-1)t_0\mathbf X_{\mathcal E}} \operatorname{Op}_h^{\mathcal E}(\mathrm{e}^{sG_m})^{-1}\Big\|\le (1+\mathcal O_s(h))\sum_{j=1}^3 \|\mathbf A_{i,j}(n-1)\|\]
		The function $G_m$ is strictly decreasing along the flow lines for times $\le t_0$ when starting in the support of $a_3$, by Proposition \ref{prop: construction escape}. By Proposition~\ref{prop: Egorov theorem}, the second norm is controlled by
		\begin{align*}
			\big\|\operatorname{Op}_h^{\mathcal E}(\mathrm{e}^{sG_m})\mathrm{e}^{-t_0\mathbf X_{\mathcal E}}\operatorname{Op}_h^{\mathcal E}(a_3\mathrm{e}^{-sG_m})\big\| & \lesssim \|\mathrm{e}^{s(G_m\circ \Phi_t-G_m)}\|_{L^{\infty}(\operatorname{Supp}(a_2))}+\mathcal O_s(h)\\ & \le C_1(1+\mathcal O_s(h))\mathrm{e}^{-C_0s}.
		\end{align*}
		It remains to take $h$ small enough depending on $s$ to ensure that the $\mathcal O_s(h)$ term is $\le 1$.
	\end{proof}
	
	\begin{proof}[Proof of Proposition \ref{prop: long time aniso}]
		By combining the above Lemmata, we find that for all $(s,N,K,\varepsilon)$, provided $T\le K|\log h|$ and $0<h<h_0(s,N,K,\varepsilon)$, we have  
	\[\mathbf N(n)\le C_1\big[ h^{-d}\mathrm{e}^{(\delta_0+\varepsilon)nt_0}+h^N+\mathrm{e}^{-C_0 s} \mathbf N(n-1)\big].\]
	Fix $\varepsilon>0$, and define
	\[n_0:=\Big\lceil \frac{d}{\varepsilon t_0}|\log h|\Big\rceil,\]
	so that
	\[h^{-d}\mathrm{e}^{(\delta_0+\varepsilon)nt_0} \le \mathrm{e}^{(\delta_0+2\varepsilon)nt_0}, \qquad \forall n\ge n_0.\]
	We take $N\gg 1$ large enough depending on $K$ to ensure $h^N\le \mathrm{e}^{(\delta_0+\varepsilon)2T}$ (recall that $T=K|\log h|$). It follows that
	\begin{equation} \label{eq: induction N(n)}\mathbf N(n)\le C_1\big[\mathrm{e}^{(\delta_0+2\varepsilon)nt_0}+\mathrm{e}^{-C_0 s}\mathbf N(n-1)\big], \qquad \forall n\in \Big[n_0, \frac{2T}{t_0}\Big], \end{equation}
	where $C_1$ may have doubled. Iterating \eqref{eq: induction N(n)} down to $n-k_0=n_0$ yields
	\begin{equation} \label{eq: after iteratingdown}\mathbf N(n)\le \Big(\sum_{k=0}^{k_0} C_1 \mathrm{e}^{(\delta_0+2\varepsilon)(n-k)t_0} (C_1 \mathrm{e}^{-C_0 s})^{k}\Big) +(C_1\mathrm{e}^{-C_0 s})^{n-n_0} \mathbf N(n_0). \end{equation}
	Recall from Proposition~\ref{prop: strong continuity group} that $\mathbf N(n_0)\le C_3 \mathrm{e}^{{C_3n_0}}$, with $C_3=C_3(s,\mathcal F)$. By \eqref{eq: after iteratingdown} we obtain
	\[\mathbf N(n)\le C_1\mathrm{e}^{(\delta_0+2\varepsilon)n t_0}\Big(\sum_{k=0}^{k_1} (\mathrm{e}^{-(\delta_0+2\varepsilon)t_0}C_1 \mathrm{e}^{-C_0 s})^{k}\Big) +(C_1\mathrm{e}^{-C_0 s})^{n} \cdot (C_1 \mathrm{e}^{-C_0 s})^{-n_0} C_3\mathrm{e}^{C_3n_0}.\]
	We take $s$ large enough to ensure $(\mathrm{e}^{-(\delta_0+2\varepsilon)t_0}C_1 \mathrm{e}^{-C_0 s})<\frac 12$ which ensures that the series is convergent, and we get
	\[\mathbf N(n)\le C_1\mathrm{e}^{(\delta_0+2\varepsilon)n t_0}\big[1+C_4\mathrm{e}^{C_4 n_0}\big],\]
	where $C_4=\max(C_3+C_0s,C_1C_3)$ depends on $s,\mathcal F$ again. Since $n_0\lesssim_\varepsilon |\log h|$, by taking $T=K|\log h|$ large enough depending on $s,\varepsilon$, we may assume $C_4 \mathrm{e}^{C_4n_0}\le \mathrm{e}^{2\varepsilon T}$, in particular, for any $n\in [T/t_0,2T/t_0]$, we obtain
	\[\mathbf N(n)\le C_1\mathrm{e}^{(\delta_0+3\varepsilon)n t_0}.\]
	This concludes the proof of Proposition \ref{prop: long time aniso}.\end{proof}
	
\begin{proof}[End of the proof of Proposition \ref{prop: R(z) for finite number of z}] Choose $\varepsilon>0$ such that $\mathcal K\subset \{z\in \mathbb C~:~\operatorname{Re}(z)>\delta_0+2\varepsilon\}$. Consider the set of functions 
	\[\mathcal F=\{\varphi_{n',T'}~:~ n'\ge 1, T'\ge 1\},\]
	where $\varphi_{n',T'}$ has been defined in \eqref{eq: def phi n t}. Let $s,K,h$ be as in Proposition \ref{prop: long time aniso}. Then, for any $\varphi\in \mathcal F$, by Proposition \ref{prop: long time aniso}, we have a.a.s., with $T=K|\log h|$:
	\[\forall n\in \Big[\frac{T}{t_0},\frac{2T}{t_0}\Big],\quad \|\widehat \varphi(-\mathrm{i}\mathbf X_{\mathcal E})\mathrm{e}^{-nt_0\mathbf X_{\mathcal E}}\|_{\mathcal H_h^s}\le \mathrm{e}^{(\delta_0+\varepsilon)nt_0}.\]
	By taking $\varphi=\varphi_{n',T'}$ with $(n',T')=(n,T)$ we obtain that a.a.s.
	\[\forall n\in \Big[\frac{T}{t_0},\frac{2T}{t_0}\Big],\quad \|\widehat \varphi_{n,T}(-\mathrm{i}\mathbf X_{\mathcal E})\mathrm{e}^{-nt_0\mathbf X_{\mathcal E}}\|_{\mathcal H_h^s}\le \mathrm{e}^{(\delta_0+\varepsilon)nt_0}.\]
	Recalling \eqref{eq: R(z) as split sum}, we deduce
	\begin{equation} \label{eq: bound R(z) in terms of N} \|\mathbf R(z)\|_{\mathcal H_h^s\to \mathcal H_h^s}\lesssim \sum_{T/t_0\le n\le 2T/t_0} \mathrm{e}^{-\operatorname{Re}(z)nt_0}\mathrm{e}^{(\delta_0+\varepsilon)nt_0}\lesssim \sum_{T/t_0\le n\le 2T/t_0} \mathrm{e}^{-\varepsilon nt_0}\lesssim  \frac{T}{t_0}\mathrm{e}^{-\varepsilon T}.\end{equation}
	For $T$ large enough depending on $\varepsilon$, we obtain $\|\mathbf R(z)\|<\frac 14$, as we wished.
\end{proof}

	\begin{rem}[A note on the order of choice of constants]
		
	We choose a compact set $\mathcal K\subset \{z\in \mathbb C~:~\operatorname{Re}(z)\ge \delta_0+2\varepsilon\}$ for some $\varepsilon>0$. We take $s\gg 1$, next $K=K(s,\mathcal K,\varepsilon)$ and $N=N(K,s,\varepsilon)$ are successively taken large enough, and finally $h=h(\mathcal K,\varepsilon,s,K,N)$ is chosen small enough depending on all the previous parameters.\end{rem}
	\subsection{Finishing the proof of Theorem~\ref{thm: extension resolvent Hs}} It remains to show that a.a.s., for $\mathcal E=\mathcal E_{\rho_n}$, we have
	\[\big\|(\mathbf X_{\mathcal E}+z)^{-1}\big\|_{\mathcal H_h^s(\mathcal M,\mathcal E)\to \mathcal H_h^s(\mathcal M,\mathcal E)}\le C(\mathcal K).\]
	This follows from the estimate \eqref{eq: estimate Q(z)} on the approximate inverse $\|\mathbf Q (z)\|$, because a.a.s., for all $z\in \mathcal K$, we have
	\[\big\|(\mathbf X_{\mathcal E}+z)^{-1}\big\|=\|\mathbf Q(z)(\operatorname{Id}+\mathbf R(z))^{-1}\|\le 2\|\mathbf Q(z)\|\lesssim \mathrm{e}^{\nu_sT}=:C(\mathcal K).\]
	Note that the constant $C(\mathcal K)$ that we obtain depends on $T=K|\log h|$, and we need to work in a regime $h|z|\ll 1$ so we could certainly take $C(\mathcal K)=C_\varepsilon(1+\sup_{z\in \mathcal K} |z|^{N_0})$ for some large constant $N_0$.

	\subsection{About the optimality of the result} \label{subsec: optimality} It was conjectured by Faure--Weich \cite[Conjecture C.6]{faureweich} and Dolgopyat--Pollicott \cite{DolgopyatPollicottAddendum}) that the value $\frac 12\operatorname{Pr}(-2\psi^u)$ is the essential spectral gap for the Pollicott--Ruelle resonances of Anosov geodesic flows. More precisely, the conjecture states that for a fixed Anosov manifold $M$, we have
	\[\frac{1}{2}\operatorname{Pr}(-2\psi^u)=\inf\Big\{\kappa\in \mathbb R~:~ \begin{array}{l}\text{there are finitely many resonances}\\ \text{in the half-space $\operatorname{Re}(z)>\kappa$} \end{array}\Big\}.\]
	An analogous conjecture for resonances of the Laplacian on Schottky surfaces was proposed by Jakobson and Naud \cite{JakobsonNaud2012}. 
	
	The conjectural value for the essential spectral gap coincides with the probabilistic low-frequency spectral gap obtained in Theorem~\ref{thm: extension resolvent}. While we are not able to prove that the constant $\delta_0=\frac 12\operatorname{Pr}(-2\psi^u)$ is optimal in Theorem~\ref{thm: extension resolvent}, we can show that it is the best possible constant in the full version of Conjecture~\ref{conj: uniform gap}, where we assume uniform polynomial resolvent bounds. In fact, the gap $\delta_0$ is optimal under the mild assumption that the covers possess some regions where the injectivity radius is increasingly large. Indeed, mixing---or rather dispersion---cannot occur faster on $M_n$ than it does on the universal cover $\widetilde M$, for times shorter than the maximal injectivity radius.

	\begin{prop}\label{prop: optimality} Let $(M_n\to M)$ be a sequence of finite degree Riemannian coverings of $M$, such that for any $R>0$, we have
		\begin{equation} \label{eq: vol inj}\operatorname{vol}\big(\{x\in M_n~:~ \operatorname{inj}(x)\ge R\}\big)\underset{n\to+\infty}{\longrightarrow} +\infty, \end{equation}
		where $\operatorname{inj}(x)$ denotes the injectivity radius of $x$. Let $(\mathcal E_n):=(\mathcal E_{\rho_n})$ be the associated sequence of vector bundles over $SM$. Assume that there is some $s>0$ such that for all $n\ge n_0$, the resolvent $(\mathbf X_{\mathcal E_n}+z)^{-1}$ extends analytically to the half-plane $\{\operatorname{Re}(z)>\kappa\}$, and that we have uniform polynomial resolvent bounds
		\begin{equation} \label{eq: assum resolvent bounds}\big\|(\mathbf X_{\mathcal E_n}+z)^{-1}\big\|_{\mathcal H^s(\mathcal M,\mathcal E_n)}\le C_\kappa (1+|\operatorname{Im}(z)|)^{N_\kappa}, \end{equation}
	for some constants $C_\kappa,N_\kappa$ that are independent of $n$. Then $\kappa\ge \delta_0$.\end{prop} 
	\begin{rem}The assumption \eqref{eq: vol inj} is equivalent to the fact that there is a sequence of points $x_n\in M_n$ such that $\operatorname{inj}(x_n)\to +\infty$. Note that it is much less restrictive than Benjamini--Schramm convergence to $\widetilde M$, which requires 
		\[\forall R\ge 0, \qquad \frac{\operatorname{vol}\big(\{x\in M_n~:~ \operatorname{inj}(x)\ge R\}\big)}{\operatorname{vol}(M_n)}\underset{n\to+\infty}{\longrightarrow} 1.\]
	\end{rem}
	
	\begin{proof}\textbf{1.} Recall from the proof of Proposition \ref{prop: operator norm spherical mean} that for $x_0\in \widetilde M$, the function $f:x\mapsto \mathbf 1_{t-1\le d(x_0,x)\le t+1}J(x_0,x)^{-1}$ satisfies 
	\[\langle \mathcal L_t f,\mathbf 1_{B(x_0,1)}\rangle_{L^2(\widetilde M)} \ge c_\varepsilon\mathrm{e}^{(2\delta_0-\varepsilon) t}, \qquad \|f\|_{L^2(\widetilde M)}\le C_\varepsilon \mathrm{e}^{(\delta_0+\varepsilon)t}.\]
	We introduce a smoothed version of $f$ that enjoys some better regularity properties. Let $D$ denote the diameter of $M$, and let $\chi$ be a smooth cutoff function that is equal to $1$ on $B(x_0,D+1)\subset \widetilde M$. Define
	\begin{equation} \label{eq: tilde f better reg} \tilde f(x):=\sum_{t-D\le d(x_0,\gamma x_0)\le t+D} J(x_0,\gamma x_0)^{-1}\chi(\gamma^{-1}x).\end{equation}
	Then, using the properties of the function $J$ (namely Lemma \ref{lem: temperness J}), one can show
	\begin{equation} \label{eq: Lt tilde f,chi}\langle \mathcal L_t \tilde f,\chi\rangle_{L^2(\widetilde M)} \gtrsim \mathrm{e}^{2(\delta_0-\varepsilon)t},\end{equation}
	while 
	\[\|\tilde f\|_{H^k(\widetilde M)}\lesssim \mathrm{e}^{(\delta_0+\varepsilon) t}, \qquad \|\chi\|_{H^k(\widetilde M)}\lesssim 1, \qquad \forall k\ge 0.\]
	
	\textbf{2.} Now, let $(M_n\to M)$ be a sequence of coverings satisfying the assumptions of Proposition \ref{prop: optimality}. Arguing as in \cite[\S9]{nonnenmacher2015decay}, the resolvent bounds \eqref{eq: assum resolvent bounds} imply that there is some integer $k$ depending on $\kappa$ such that for all $f,g\in C^\infty_{\rm new}(SM_n)$, one has
	\begin{equation} \label{eq: implies mixing} \int_{SM_n} (f\circ \varphi_t)g=\mathcal O_\kappa(\|f\|_{H^k(SM_n)}\|g\|_{H^k(SM_n)}\mathrm{e}^{\kappa t}).\end{equation}We now make a convenient choice for $f,g$. Fix some $R>0$. Then, for $n$ large enough, we can find $x_0\in \widetilde M$ such that $M_n$ contains two disjoint embedded copies of $B(x_0,R)\subset \widetilde M$ that are distance $\ge R$ from each other. For $t\le R$, we can use the function $\tilde f$ from \eqref{eq: tilde f better reg} to build a function $f_n\in C^\infty_{\rm new}(M_n)$ such that $f_n\equiv \tilde f$ in the first copy of $B(x_0,R)$ and $f_n\equiv -\tilde f$ in the second copy. Similarly, we build $\chi_n\in C^\infty_{\rm new}(M_n)$ such that $\chi_n\equiv \chi$ in the first copy of $B(x_0,R)$ and $\chi_n\equiv -\chi$ in the second copy. Let $\pi:SM_n\to M_n$ denote the projection; the functions $f_n,\chi_n$ lift to functions $\pi^*f_n,\pi^*\chi_n\in C^\infty_{\rm new}(SM_n)$. Moreover, by \eqref{eq: Lt tilde f,chi}, for all $t\le R$ we have
	\begin{equation} \label{eq: Ltfn chin}\langle \mathrm{e}^{-tX}(\pi^*f_n),\pi^*\chi_n\rangle_{L^2(SM_n)}=\langle \mathcal L_t f_n,\chi_n\rangle_{L^2(M_n)}=2\langle \mathcal L_t\tilde f,\chi\rangle_{L^2(\widetilde M)} \gtrsim \mathrm{e}^{(2\delta_0-\varepsilon)t},\end{equation}
	and $\|\pi^*f_n\|_{H^k(SM_n)}\lesssim \mathrm{e}^{(\delta_0+\varepsilon)t}$, $\|\pi^*\chi_n\|_{H^k}\lesssim 1$. Combined with \eqref{eq: implies mixing} this implies
	\[\mathrm{e}^{2(\delta_0-\varepsilon)t}\lesssim \mathrm{e}^{(\delta_0+\varepsilon)t} \mathrm{e}^{\kappa t}.\]
	Letting $t=R$ go to $+\infty$, one obtains $\delta_0\le \kappa+3\varepsilon$. Since this holds for any $\varepsilon>0$ we infer $\kappa\ge \delta_0$.

	\end{proof}

	\appendix 
	\section{A short proof that $\delta_0<-\gamma_0$.} We show that the low-frequency gap $\delta_0/2=\operatorname{Pr}(-2\psi^u)/2$ obtained in Theorem \ref{thm: spectral gap} is a strict improvement over the asymptotic gap $-\gamma_0/2$ from Theorem \ref{thm: tsujii essential gap}, whenever the unstable Jacobian $\psi^u$ is not cohomologous to a constant. This applies in particular when $(M,g)$ is a closed surface with nonconstant negative curvature. We recall that a H\"older continuous potential $F\in C(SM)$ is cohomologous to a constant if there exists a function $u$ and a constant $c$ such that for any $x\in SM$ and $t\in \mathbb R$, one has
	\[\int_0^t F(\varphi_s(x))\mathrm{d}s=u(\varphi_t(x))-u(x)+ct.\]
	\begin{prop}\label{prop: jaco coho const} Let $(M,g)$ be a closed Anosov manifold. Assume that the unstable Jacobian is not cohomologous to a constant. Then,
		\[\forall q>1, \ \operatorname{Pr}(-q\psi^u)<-(q-1)\gamma_0.\]
	\end{prop}
	It is known that $-\psi^u$ is cohomologous to a constant when $(M,g)$ is a locally symmetric space of negative curvature. Katok's conjecture asserts that it is the only case where this happens. The conjecture was proved for negatively curved surfaces in the seminal paper \cite{katok1982} and is open in higher dimensions. 
	
	The following result can be found in \cite{ratner1973,Parrypol}.
		\begin{lem}[Ratner, Parry--Pollicott]\label{lem: ratner pp} Let $F$ be a H\"older continuous potential on $SM$. Then the pressure $t\mapsto \operatorname{Pr}(tF)$ is a convex real analytic function. Moreover, it is strictly convex unless $F$ is cohomologous to a constant.
	\end{lem}

Proposition \ref{prop: jaco coho const} will follow from this next Lemma.
	\begin{lem}Define $\beta(q):=\operatorname{Pr}(-q\psi^u)$. Assume $\psi^u$ is not cohomologous to a constant. Then $\beta'(q)<-\gamma_0$ for all $q\in \mathbb R$.
	\end{lem}
	\begin{proof}Assume that there is some $q_0\in \mathbb R$ such that $\beta'(q_0)\ge -\gamma_0$. Since $\psi^u$ is not cohomologous to a constant, we know by Lemma \ref{lem: ratner pp} that $\beta'$ is strictly increasing, hence we can find $\varepsilon>0$ and $q_1>q_0$ such that $\beta'(q_1)>-\gamma_0+\varepsilon$. It follows that there is a constant $C_\varepsilon>0$ such that
		\begin{equation} \label{eq: lower bound beta(q)}\beta(q)\ge (-\gamma_0+\varepsilon)q+C_\varepsilon,\end{equation}
		uniformly in $q\ge 1$. 	Let us now recall that for any fixed $q$, we have
		\[\lim_{T\to +\infty}\frac 1T\log \sum_{\ell(\gamma)\in [T,T+1]} \frac{1}{|I-P_\gamma|^q}=\beta(q),\]
		where $P_\gamma$ is the linearized Poincar\'e map of the periodic orbit and $|I-P_\gamma|:=|\det(I-P_\gamma)|$. By definition of $\gamma_0$, for any $\delta>0$ there is a constant $C_{\delta}>0$ such that for any closed geodesic of length $\ell(\gamma)$ large enough, we have
		\[|I-P_\gamma|\ge C_\delta \mathrm{e}^{(\gamma_0-\delta)\ell(\gamma)}.\]
		Since the number of periodic orbits of length $\le T$ grows at most exponentially, we have
		\begin{equation} \label{eq: upper bound sum per or}\sum_{\ell(\gamma)\in [T,T+1]} \frac{1}{|I-P_\gamma|^q}\le C_{\delta,q} \exp\big(T((-\gamma_0+\delta)q+C)),\end{equation}
		for some constant $C>0$. Taking the logarithms then dividing by $T$ on both sides, we obtain by letting $T\to +\infty$:
		\begin{equation} \label{eq: upper bound beta(q)} \beta(q)\le (-\gamma_0+\delta)q+C.\end{equation}
		Combining \eqref{eq: lower bound beta(q)} with \eqref{eq: upper bound beta(q)} and letting $q\to +\infty$, we deduce
		\[-\gamma_0+\varepsilon\le -\gamma_0+\delta.\]
		Since $\delta$ has been chosen $>0$ independently of $\varepsilon$, we reach a contradiction by taking $\delta=\varepsilon/2$. \end{proof}
	\begin{proof}[Proof of Proposition \ref{prop: jaco coho const}] By the Lemma, we have $\beta'(q)<-\gamma_0$ for all $q\ge 1$. Since $\beta(1)=\operatorname{Pr}(-\psi^u)=0$, we deduce that for any $q>1$:
		\[\operatorname{Pr}(-q\psi^u)=\beta(q)=\int_1^q \beta'(q')\mathrm{d}q'<-\gamma_0(q-1).\] 
		\end{proof}
		
	\section{Semiclassical quantization on unitary flat bundles} \label{appendix} Let $\mathcal M$ be a closed Riemannian manifold of dimension $d$ and $\mathcal E\to \mathcal M$ be a unitary flat vector bundle over $\mathcal M$. We define a semiclassical quantization procedure mapping symbols $a\in C^\infty(T^*\mathcal M)$ to linear operators $\operatorname{Op}_h^{\mathcal E}(a):C^\infty(\mathcal M)\to \mathcal D'(\mathcal M,\mathcal E)$. We do not aim at being exhaustive nor proving optimal results, but focus on the tools needed for our purposes. This is not the first instance where one has to use a semiclassical quantization procedure to deal with an infinite number of bundles at once, see \cite{MaMa,cekić2024semiclassicalanalysisprincipalbundles}. We use the notation $\langle \xi\rangle:=(1+|\xi|^2)^{1/2}$.
	
	\subsection{Order functions and spaces of symbols}	\begin{defi} In this Appendix, a bounded function $m\in S^0(T^*\mathcal M)$ that satisfies $m(x,\lambda \xi)=m(x,\xi)$ for all $\lambda>1$ and $|\xi|\ge 1$ is called an \emph{order function}. \end{defi}
	
	\begin{defi}Let $m$ is an order function and $\rho\in (1/2,1]$. A function $a\in C^\infty(T^*\mathcal M)$ belongs to the exotic symbol class $S^{m(\bullet)}_\rho(T^*\mathcal M)$ if for any multiindices $\alpha,\beta\in \mathbb N^d$, there exists a constant $C_{\alpha,\beta}>0$ such that
		\[\forall (x,\xi)\in T^*\mathcal M, \quad |\partial_\xi^{\alpha}\partial_x^{\beta} a|\le C_{\alpha,\beta} \langle \xi\rangle^{m(x,\xi)-\rho|\alpha|+(1-\rho)|\beta|}.\]
		where the inequalities are understood in charts.
		This is a Fr\'echet space with seminorms
		\[p_{\alpha,\beta}(a):=\sup_{(x,\xi)\in T^*\mathcal M} \big|\langle \xi\rangle^{-m(x,\xi)+\rho|\alpha|-(1-\rho)|\beta|} \partial_\xi^\alpha\partial_x^\beta a\big|.\]
		When $m$ is constant we recover the standard H\"ormander class of symbols which is simply denoted $S^m_\rho$.
	\end{defi}
	\begin{ex}If $m$ is an order function, then the symbol $\langle \xi\rangle^{m(x,\xi)}$ belongs to $S^{m(\bullet)}_\rho(T^*\mathcal M)$ for any $\rho<1$. When $m$ is nonconstant, the symbol $\langle \xi\rangle^m$ falls short of the nicer class $S_1^{m(\bullet)}(T^*\mathcal M)$ due to the appearance of logarithmic terms in $\langle \xi\rangle$ when performing differentiations. \end{ex}
	\begin{defi}If $m$ is an order function, we set
		\[S^{m(\bullet)}_{1^-}:=\bigcap_{\rho<1} S_\rho^{m(\bullet)}.\]\end{defi}
	\begin{lem}For any $s\in \mathbb R$ and $\varepsilon>0$, the inclusion $S_{1^-}^{s}\subset S^{s+\varepsilon}_1$ is continuous.
	\end{lem}
	
	Then, the following results holds:
	\begin{prop} Assume $m\le m'$ are two order functions. Then, $S^{m(\bullet)}_\rho\subset S^{m'(\bullet)}_{\rho}$ and the inclusion is continuous. In particular,
		\[S^{\inf m}_\rho(T^*\mathcal M)\subset S^{m(\bullet)}_\rho(T^*\mathcal M)\subset S^{\sup m}_\rho(T^*\mathcal M).\]
	\end{prop}
	
	\subsection{Semiclassical calculus on $C^\infty(\mathcal M)$} \label{subsec: semi cal on M} We begin with the trivial bundle $\mathbb{C}\to \mathcal M$, the case of general unitary flat bundles $\mathcal E\to \mathcal M$ is dealt with in the next section. A concise summary of the calculus for compactly supported symbols is given in the note of Guedes Bonthonneau \cite{bonthonneauWeyl}. Another reference for the $C^\infty_{\rm comp}$-calculus is the Appendix of \cite{DJN21} where very explicit estimates can be found---they use right-quantization but the proofs adapt seamlessly to the Weyl quantization. \medskip
	
	Consider an atlas $\{(\mathcal U_i,\varphi_i)\}$ of $\mathcal M$, with $\varphi_i:\mathcal U_i\to \mathcal V_i\subset \mathbb{R}^d$. We may assume, following \cite{bonthonneauWeyl}, that the charts are \emph{isochore}, \emph{i.e.} that $\varphi_i$ maps the Riemannian measure on $\mathcal M$ onto the Lebesgue measure on $\mathbb{R}^d$ (for the existence of such charts, see \cite{MoserIsochore}). 
	Let $\{\chi_i\},\{\eta_i\}$ be two families of smooth functions, with $\chi_i,\eta_i\in C^\infty_{\rm comp}(\mathcal U_i,[0,1])$, such that $\chi_i\prec \eta_i$ (\emph{i.e.} $\eta_i\equiv 1$ on $\operatorname{Supp}\chi_i$), $\sum \chi_i^2=1$ and $\sum \eta_i\le 2$. We introduce shorthands
	\[\chi^i:=\chi_i\circ \varphi_i^{-1}, \qquad \eta^i:=\eta_i\circ \varphi_i^{-1}.\]
	 We denote by $\widetilde \varphi_i:T^*\mathcal U_i\to T^*V_i$ the symplectic lift of $\varphi_i$.
	
	The Weyl quantization on $\mathbb{R}^d$ is denoted by $\mathrm{Op}_h^{\rm w}(\cdot)$. Recall that it is defined by
	\[\mathrm{Op}_h^{\rm w}(a) u(x)=\frac{1}{(2\pi h)^d}\int \mathrm{e}^{\frac{\mathrm{i}}{h}\langle x-y,\xi\rangle}a\big(\frac{x+y}{2},\xi\big)u(y)\mathrm{d}y\mathrm{d}\xi,\]
	for any suitable symbol $a\in C^\infty(T^*\mathbb{R}^d)$. We can also let $\mathrm{Op}_h^{\rm w}(a)$ act on $\mathbb C^n$-valued functions by acting as $\mathrm{Op}_h^{\rm w}(a)$ on each coordinate. Given two reasonable symbols $a,b$, there is a $h$-dependent symbol $a\#b$ known as the \emph{Moyal product} of $a$ and $b$, such that
	\[\mathrm{Op}_h^{\rm w}(a)\mathrm{Op}_h^{\rm w}(b)=\mathrm{Op}_h^{\rm w}(a\#b),\]
	see \cite[Theorem 4.11]{zworski2012semiclassical}.
	
	\begin{defi}[Weyl quantization on manifolds] Let $m$ be an order function, and $\rho\in (1/2,1]$. For a real symbol $a\in S^{m(\bullet)}_\rho(T^*\mathcal M)$, define $a_i:=(\eta_i a)\circ \widetilde \varphi_i^{-1}\in C^\infty_{\rm comp}(V_i)$, then set
		\[\mathrm{Op}_h^{\mathcal M}(a):=\sum_{i\in I} (\varphi_i)^* \chi^i \mathrm{Op}_h^{\rm w}(a_i) (\varphi_i)_* \chi_i.\]
		The operator $\mathrm{Op}_h^{\mathcal M}(a)$ is formally self-adjoint on $C^\infty({\mathcal M})$. \end{defi}
	\begin{rem}
		Consider the map 
		\begin{equation} \label{eq: def mapping T} T:C^\infty({\mathcal M})\to \bigoplus_{i\in I}C^\infty(V_i), \qquad u\mapsto \big((\varphi_i)_*(\chi_iu)\big)_{i\in I}.\end{equation}
		Because the charts are isochore, $T$ induces an isometry $T:L^2(\mathcal M)\to \bigoplus_{i\in I}L^2(V_i)$, with adjoint given by
		\[T^*(v_i)_{i\in I}=\sum_{i\in I} \chi_i \varphi_i^*v_i.\]
		Then, one has $T^*T=\operatorname{Id}_{L^2(\mathcal M)}$, 
		and we can write
		\begin{equation} \label{eq: op in T coord} \operatorname{Op}_h^{\mathcal M}(a)=T^*\operatorname{Diag}_{i\in I}\big(\mathrm{Op}_h^{\rm w}(a_i)\big)T,\end{equation}
		where $\operatorname{Diag}_{i\in I}(A_i)(v_i)_{i\in I}:=(A_i v_i)_{i\in I}$ for any family of operators $A_i:C^\infty(V_i)\to C^\infty(V_i)$. \end{rem}
	
	The following proposition expresses the action of a change of chart on pseudodifferential operators.
	
	\begin{prop}[Egorov theorem in $\mathbb R^d$]\label{prop: egorov rd}Let $f:U\subset \mathbb{R}^d\to V\subset \mathbb{R}^d$ be a smooth diffeomorphism. Let $\chi\in C^\infty_{\rm comp}(U)$ and $\psi\in C^\infty_{\rm comp}(V)$. Then, for all $a\in S^m_{\rho}(T^*\mathbb{R}^d)$ there is a symbol $b\in S^m_{\rho}(T^*\mathbb{R}^d)$ such that
		\[f^*\psi\mathrm{Op}_h^{\rm w}(a)f_*\chi=\mathrm{Op}_h^{\rm w}(b)+\mathcal O(h^\infty)_{L^2\to L^2}.\]
		Moreover, the symbol $b$ admits the following asymptotic expansion: for all $N$, we have
		\[b=\sum_{k\le N-1} h^k \mathbf D^{2k} (a\circ \tilde f) + h^{N} r_N,\]
		with $r_N\in S^{m-N(2\rho-1)}_\rho(T^*\mathbb{R}^d)$ with seminorms depending on the seminorms of $a$ and the diffeomorphism $f$ only. Here $\mathbf D^{2k}$ denotes a differential operator of order $2k$, with at least $k$ derivatives hitting the $\xi$-variable. The leading term is given by
		\[\mathbf D^0(a\circ \tilde f)=\psi(f(x))a(f(x),\mathrm{d}f(x)^{-\top}\xi)\chi(x).\]\end{prop}
	\begin{proof}This is a straightforward adaptation of \emph{e.g.} \cite[Lemma 5.2.7]{Lefeuvre} to the semiclassical setting.\end{proof}
	
	\begin{prop}\label{prop: list properties scalar quanti} The following properties for $\operatorname{Op}_h^{\mathcal M}$ hold. \begin{enumerate}
			
			\item \emph{\textbf{$L^2$-continuity for symbols in $S^0_1(T^*{\mathcal M})$.}} Let $a\in S^0_1(T^*{\mathcal M})$. Then,
			\[\big\|\operatorname{Op}_h^{\mathcal M}(a)\big\|_{L^2\to L^2}\le \|a\|_\infty+\mathcal O_a(h).\]
			\item \emph{\textbf{$L^2$-continuity for symbols in $S^0_\rho(T^*{\mathcal M})$.}} Let $a\in S^0_\rho(T^*{\mathcal M})$. Then,
			\begin{equation} \label{eq: L^2 cont scalar exotic}\big\|\operatorname{Op}_h^{\mathcal M}(a)\big\|_{L^2\to L^2}\le C(a),\end{equation}
			where the constant $C(a)$ is a continuous seminorm on $S^m_\rho(T^*\mathcal M)$.
			\item \emph{\textbf{Pseudolocality.}} Assume $a\in S^{m}_\rho(T^*\mathcal M)$. Then, for any $\chi,\psi\in C^\infty(\mathcal M)$ independent of $h$ with $\operatorname{Supp}(\chi)\cap \operatorname{Supp}(\psi)=\varnothing$, we have
			\[\chi \operatorname{Op}_h^{\mathcal M}(a)\psi =\mathcal O(h^\infty)_{\mathcal D'(\mathcal M)\to C^\infty(\mathcal M)}.\]
			\item \emph{\textbf{Composition.}} Let $m_1,m_2$ be two order functions, such that $m_1+m_2\le 0$ outside a compact set. For any $a\in S^{m_1(\bullet)}_{1^-}$ and $b\in S^{m_2(\bullet)}_{1^-}$, one has
			\[\operatorname{Op}_h^{\mathcal M}(a)\operatorname{Op}_h^{\mathcal M}(b)=\operatorname{Op}_h^{\mathcal M}(ab)+\mathcal O(h)_{L^2\to L^2}.\]
	\end{enumerate}\end{prop}
	\begin{proof} Recall that $\operatorname{Op}_h^{\mathcal M}(a)=T^*\operatorname{Diag}_{i\in I}\big(\mathrm{Op}_h^{\rm w}(a_i)\big)T$, and $T$ is an isometry $L^2(\mathcal M)\to \bigoplus L^2(V_i,\mathbb{C}^n)$. Because $\operatorname{Diag}_{i\in I}(\operatorname{Op}_h^{\rm w}(a_i))$ preserves the direct sum $\bigoplus_{i\in I} L^2(V_i,\mathbb{C}^n)$, one has
		\[\big\|\operatorname{Op}_h^{\mathcal M}(a)\big\|_{L^2(\mathcal M)}\le \sup_{i\in I}\big\|\operatorname{Op}_h^{\rm w}(a_i)\big\|_{L^2(\mathbb{R}^d)}.\]
		
		(1) Assume $a\in S^0_1(T^*\mathcal M)$. Then $a_i\in S^0_1(T^*\mathbb{R}^d)$. By \cite[Theorem 13.13]{zworski2012semiclassical}, we get that for each $i\in I$:
		\[\big\|\operatorname{Op}_h^{\rm w}(a_i)\big\|_{L^2(\mathbb{R}^d)}\le \| \eta_i a\|_{\infty}+\mathcal O(h)\le \|a\|_{\infty}+\mathcal O(h).\]
		
		(2) Assume $a\in S^0_\rho(T^*\mathcal M)$. We start by recalling a result for the $h=1$ quantization. For any $b\in S^0_\rho(T^*\mathbb{R}^d)$, by \cite[Theorem 18.6.3]{HorIII} applied to the metric from \cite[(18.4.1)']{HorIII} (see also \cite[Theorem 18.1.13]{HorIII} and the remarks around \cite[(18.1.1)"]{HorIII}), the operator $\operatorname{Op}^{\rm w}_1(b):L^2(\mathbb{R}^d)\to L^2(\mathbb{R}^d)$ is continuous, moreover, there is a constant $C>0$ depending on the dimension $d$ and $\rho$, such that
		\[\|\operatorname{Op}_1^{\rm w}(b)\|_{L^2\to L^2}\le C\sum_{|\alpha|+|\beta|\le C} p_{\alpha,\beta}(b).\]
		Now, for a general symbol $b\in S^0_\rho(T^*\mathbb{R}^d)$, it follows from a standard rescaling argument that for $b_h(x,\xi)=b(h^{\frac 12} x,h^{\frac 12}\xi)$ we have
		\[\|\operatorname{Op}_h^{\rm w}(b)\|_{L^2\to L^2}=\|\operatorname{Op}_1^{\rm w}(b_h)\|_{L^2\to L^2}.\] One can check that $p_{\alpha,\beta}(b_h)\le p_{\alpha,\beta}(b)$ for all multiindices $\alpha,\beta$, which gives
		\[\|\operatorname{Op}_h^{\rm w}(b)\|_{L^2\to L^2}\le C \sum_{|\alpha|+|\beta|\le C} p_{\alpha,\beta}(b).\]
		It suffices to apply this result to $b=a_i$ for each $i\in I$ to obtain \eqref{eq: L^2 cont scalar exotic}.
		
		(3) It follows from the result for Weyl quantization on $\mathbb{R}^d$.
		
		(4) We have
		\begin{equation*} \operatorname{Op}_h^{\mathcal M}(a)\operatorname{Op}_h^{\mathcal M}(b)=\sum_{i\in I} \varphi_i^*\chi_i \sum_{j\in I} \operatorname{Op}_h^{\rm w}(a_i)\chi_i(\varphi_i)_*\varphi_j^*\chi_j\operatorname{Op}_h^{\rm w}(b_j) (\varphi_j)_*\chi_j .\end{equation*}
		By pseudolocality, we may insert a factor $\eta_i$ on the right of the product, up to adding an error $\mathcal O(h^\infty)$:
		\[\operatorname{Op}_h^{\rm w}(a_i)\chi_i (\varphi_i)_*\varphi_j^* \chi_j\operatorname{Op}_h^{\rm w}(b_j)=\operatorname{Op}_h^{\rm w}(a_i)\chi_i (\varphi_i)_*\varphi_j^*\chi_j\operatorname{Op}_h^{\rm w}(b_j)\eta_i+\mathcal O(h^\infty)_{L^2\to L^2},\]
		and then rewrite, with $\varphi_{ij}=\varphi_j\circ \varphi_i^{-1}$:
		
		\begin{equation*} \operatorname{Op}_h^{\mathcal M}(a)\operatorname{Op}_h^{\mathcal M}(b)=\sum_{i\in I} \varphi_i^*\chi_i \Big(\sum_{j\in I}\operatorname{Op}_h^{\rm w}(a_i)\chi_i(\varphi_{ij})^*\chi_j\operatorname{Op}_h^{\rm w}(b_j) (\varphi_{ij})_*\chi_j \Big) (\varphi_i)_*\eta_i+\mathcal O(h^\infty)_{L^2\to L^2}.\end{equation*}
		Since $b_j\in S^{m_2\circ \widetilde \varphi_j^{-1}}_\rho(T^*\mathbb R^d)$, we first have by the change of variables formula
		\[(\varphi_{ij})^*\chi_j\operatorname{Op}_h^{\rm w}(b_j) (\varphi_{ij})_*\eta_i\chi_j=\mathrm{Op}_h^{\rm w}(\chi_j^2\eta_ib_j\circ \widetilde \varphi_{ij}+hr),\]
		with $r\in h S^{m_2\circ \widetilde \varphi_i^{-1}-(2\rho-1)}_\rho(T^*\mathbb R^d)$. Since $a_i\in S^{m_1\circ \widetilde \varphi_i^{-1}}_\rho(T^*\mathbb R^d)$ and $m_1+m_2\le 0$ outside a compact set, we have $a_i\#r\in h S^{-(2\rho-1)}_\rho$; since this holds for any $\rho\in (1/2,1)$ it follows that $a_i\#r\in h S^0_1$. By playing with the composition rule and the $L^2$-continuity theorem, we find that
		\[\operatorname{Op}_h^{\rm w}(a_i)\chi_i(\varphi_{ij})^*\chi_j\operatorname{Op}_h^{\rm w}(b_j) (\varphi_{ij})_*\eta_i\chi_j=\mathrm{Op}_h^{\rm w}(\chi_ia_i(\chi_j)^2b_i)+\mathcal O(h)_{L^2\to L^2},\]
		which implies by using the composition rule again:
		\[\operatorname{Op}_h^{\rm w}(a_i)\chi_i(\varphi_{ij})^*\chi_j\operatorname{Op}_h^{\rm w}(b_j) (\varphi_{ij})_*\eta_i\chi_j
		=\operatorname{Op}_h^{\rm w}((ab)_i\chi_j^2)\chi_i+\mathcal O(h)_{L^2\to L^2}.\]	
		Since $\sum_{j\in I} \chi_j^2=1$ this gives
		\begin{equation*} \operatorname{Op}_h^{\mathcal M}(a)\operatorname{Op}_h^{\mathcal M}(b)=\sum_{i\in I} \varphi_i^*\chi_i \operatorname{Op}_h^{\rm w}((ab)_i)\chi_i (\varphi_i)_*+\mathcal O(h)_{L^2\to L^2}.\end{equation*}
		From the definition of $\operatorname{Op}_h^{\mathcal M}$ this exactly means that
		\begin{equation*} \operatorname{Op}_h^{\mathcal M}(a)\operatorname{Op}_h^{\mathcal M}(b)=\operatorname{Op}_h^{\mathcal M}(ab)+\mathcal O(h)_{L^2\to L^2}.\end{equation*}
	\end{proof}
	
	We now turn to a Egorov-type Lemma, in a rather \emph{ad hoc} setting which is sufficient for our purposes.
	\begin{lem} \label{lem: egorov like scalar M} Let $a\in S^{m_1(\bullet)}_{1^-}(T^*\mathcal M)$, $b\in S^{m_2(\bullet)}_{1^-}(T^*\mathcal M)$ and let $f$ be a volume-preserving diffeomorphism of $\mathcal M$. Let $\tilde f(x,\xi)=(f(x),\mathrm{d}f(x)^{-\top}\xi)$ denote the symplectic lift of $f$ to $T^*\mathcal M$ and assume that $m_1\circ \tilde f+m_2\le 0$ outside a compact set. Then, the following holds. 
		\begin{itemize}
			\item The operator $\operatorname{Op}_h^{\mathcal M}(a)f_*\operatorname{Op}_h^{\mathcal M}(b)$ is continuous $L^2\to L^2$, i.e.
			\[\big\|\operatorname{Op}_h^{\mathcal M}(a)f_*\operatorname{Op}_h^{\mathcal M}(b)\big\|=\mathcal O_{a,b}(1).\]
			\item If $m_1\circ \widetilde f+m_2\le -c<0$ on $\operatorname{Supp}(b)\cap \widetilde f^{-1}(\operatorname{Supp}(a))$ then,
			\[\big\|\operatorname{Op}_h^{\mathcal M}(a)f_*\operatorname{Op}_h^{\mathcal M}(b)\big\|\lesssim \|(a\circ \widetilde f)b\|_{\infty}+\mathcal O_{a,b}(h).\]
		\end{itemize}

	\end{lem}
	\begin{proof}Since $f$ preserves the volume, we have
		\[\big\|\operatorname{Op}_h^{\mathcal M}(a)f_*\operatorname{Op}_h^{\mathcal M}(b)\big\|_{L^2\to L^2}=\big\| f^*\operatorname{Op}_h^{\mathcal M}(a)f_*\operatorname{Op}_h^{\mathcal M}(b)\big\|_{L^2\to L^2}.\]
		Thus, to obtain the result, it is enough to bound, for each chart $(\varphi_\ell,\chi_\ell)$ of the atlas, the norm
		\[\big\| f^*\operatorname{Op}_h^{\mathcal M}(a)f_*\operatorname{Op}_h^{\mathcal M}(b) \varphi_\ell^*\chi_\ell\big\|_{L^2\to L^2},\]
		and in turn by pseudolocality it is enough to bound (here we take a cutoff $\chi_\ell'$ such that $\chi_\ell\prec \chi_\ell'$)
		\begin{equation} \label{eq: by pseudoloc it is enough} \big\|\big[(\varphi_\ell)_*\chi_\ell' f^*\operatorname{Op}_h^{\mathcal M}(a)f_* \varphi_\ell^*(\chi_\ell')\big]\circ \big[ (\varphi_\ell)_*\chi_\ell'\operatorname{Op}_h^{\mathcal M}(b) (\varphi_\ell)^*\chi_\ell\big]\big\|_{L^2\to L^2}.\end{equation}
		Using Proposition~\ref{prop: egorov rd} twice, we find 
		\[(\varphi_\ell)_*\chi_\ell' f^*\operatorname{Op}_h^{\mathcal M}(a)f_* \varphi_\ell^*(\chi_\ell')=\mathrm{Op}_h^{\rm w}((\chi_\ell')^2a\circ \widetilde f\circ \widetilde \varphi_\ell^{-1}+hr_1),\]
		with $r_1\in S^{m_1\circ \widetilde f\circ \widetilde \varphi_\ell^{-1}(\bullet)-(2\rho-1)}_\rho(T^*\mathbb R^d)$ for all $\rho\in (1/2,1)$, and 
		\[(\varphi_\ell)_*\chi_\ell'\operatorname{Op}_h^{\mathcal M}(b) (\varphi_\ell)^*\chi_\ell=\mathrm{Op}_h^{\rm w}(\chi_\ell b\circ \varphi_\ell^{-1}+hr_2),\]
		with $r_2\in S^{m_2\circ \widetilde \varphi_\ell^{-1}(\bullet)-(2\rho-1)}_\rho(T^*\mathbb R^d)$ for all $\rho\in (1/2,1)$. Since $m_1\circ \widetilde f+m_2\le 0$ outside a compact set, by item (4) of Proposition \ref{prop: list properties scalar quanti}, the operator in \eqref{eq: by pseudoloc it is enough} is given by
		\[\mathrm{Op}_h^{\rm w}([\chi_\ell (a\circ \tilde f) b]\circ \widetilde \varphi_\ell^{-1})+\mathcal O(h)_{L^2\to L^2}.\]
		Here we used the fact that the Moyal product $(\varphi_\ell)_*a\#r_2$ belongs to $S^0_1(T^*\mathbb R^d)$, for $(\varphi_\ell)_* a\in S_\rho^{m_1\circ \tilde f\circ \widetilde \varphi_\ell^{-1}}(T^*\mathbb R^d)$ and $r_2\in S^{m_2\circ \widetilde \varphi_\ell^{-1}(\bullet)-(2\rho-1)}_\rho(T^*\mathbb R^d)$.

		We now invoke the $L^2$-continuity theorem:
		\begin{itemize}
			\item If $m_1\circ \widetilde f+m_2\le -c<0$ on the support of $(a\circ \tilde f)b$ then $[\chi_\ell (a\circ \tilde f) b]\circ \widetilde \varphi_\ell^{-1}$ belongs to $S^0_1(T^*\mathbb{R}^d)$ hence
			\[\|\mathrm{Op}_h^{\rm w}([\chi_\ell (a\circ \tilde f) b]\circ \widetilde \varphi_\ell^{-1})\|_{L^2(\mathbb R^d)}\le \|(a\circ \tilde f)b\|_\infty+\mathcal O_{a,b}(h).\]
			\item If we just have $m_1\circ \widetilde f+m_2\le 0$ on the support then
			\[\|\mathrm{Op}_h^{\rm w}([\chi_\ell (a\circ \tilde f) b]\circ \widetilde \varphi_\ell^{-1})\|_{L^2(\mathbb R^d)}=\mathcal O_{a,b}(1).\]
		\end{itemize}
	\end{proof}

	\subsection{Quantization on unitary flat bundles} We turn to the case of general unitary flat bundles $\mathcal E\to \mathcal M$. The fiber over $x$ is denoted by $\mathcal E_x$.
	
	\subsubsection{Local trivializations of $\mathcal E$} Fix a finite covering $\{\mathcal U_j\}$ of $\mathcal M$, that need not be the same as in the previous section---hopefully this causes no confusion. We consider flat bundles over $\mathcal M$ of the form 
	\begin{equation} \label{eq: expression mathcal E} \mathcal E=\bigsqcup \mathcal U_j\times \mathbb C^n\big/\sim,\end{equation}
	where for any $x\in \mathcal U_i\cap \mathcal U_j$ we identify $(x,v)\in \mathcal V_i\times \mathbb C^n$ with $(x,\kappa_{ij}(v))\in \mathcal U_j\times \mathbb C^n$, for some unitary matrix $\kappa_{ij}$. If the covering $\{\mathcal U_j\}$ is chosen in such a way that any finite intersection of the $\mathcal U_j$'s is either contractible or empty, then any flat bundle over $\mathcal M$ can be expressed in the form \eqref{eq: expression mathcal E}. We will describe another choice which is more suited to the case of lifts of bundles over $M$ in the next subsection. 
	
	We consider a quadratic partition of unity $\{\chi_j\}$ associated to the covering $\{\mathcal U_j\}$, that is $\operatorname{Supp}\chi_j\subset \mathcal U_j$ and $\sum \chi_j^2=1$. Using \eqref{eq: expression mathcal E}, we introduce a family of local trivializations
	\[R_i:C^\infty_{\rm comp}(\mathcal U_i,\mathcal E)\to C^\infty_{\rm comp}(\mathcal U_i,\mathbb{C}^n).
	\]
	Here $R_i$ extends to an isometry from $L^2(\mathcal U_i,\mathcal E)$ to $L^2(\mathcal U_i,\mathbb C^n)$. The adjoint of $R_i$ is denoted by $R_i^*$. Then, we define a global map
	\[R:\begin{array}{rl}C^\infty(M,E)&\longrightarrow \bigoplus_{i\in I} C^\infty(\mathcal U_i,\mathbb{C}^n) \\ s&\longmapsto (\chi_i R_i s)_{i\in I}\end{array}.\]
	Then, $R$ induces an isometry $L^2(\mathcal M,\mathcal E)\to \bigoplus_{i\in I} L^2(\mathcal U_i,\mathbb{C}^n)$, whose adjoint is given by
	\[R^*: (s_i)_{i\in I}\longmapsto \sum_{i\in I} \chi_iR_i^* s_i,\]
	Since $R_i^*R_i=\operatorname{Id}$ and $\sum \chi_i^2=1$ we have $R^*R=\operatorname{Id}_{C^\infty(\mathcal M,\mathcal E)}$. The operator $R_{ij}:=R_iR_j^*$ maps $C^\infty(\mathcal U_i\cap \mathcal U_j,\mathbb{C}^n)$ onto itself; it is a multiplication operator by a unitary matrix which is independent of the point $x\in \mathcal U_i\cap \mathcal U_j$. 

	\subsubsection{The case of lifted bundles} When $\mathcal E\to SM$ is the lift of a bundle $E\to M$, it is useful for our purposes to make a particular choice of covering of $SM$ by open sets. We let $\{U_j\}$ be a finite covering of $M$ such that any finite intersection of the $U_j$'s is either empty or contractible. Then, we have
	\[E=\bigsqcup U_j\times \mathbb C^n \big/  \sim,\]
	where for any $x\in U_i\cap U_j$ we identify $(x,z)\in U_i\times \mathbb C^n$ with $(x,\kappa_{ij}(z))\in U_j\times \mathbb C^n$. Since the bundle $E\to M$ is flat and unitary, one can make $\kappa_{ij}$ a constant unitary matrix.
	
	 We consider an associated quadratic partition of unity $\psi_j$. Letting $\pi:SM\to M$ denote the projection, we define
\[\mathcal U_j=\pi^{-1}(U_j), \qquad \chi_j=\pi^*\psi_j.\]
	Then, the bundle $\mathcal E=\pi^*E$ over $SM$ is given by
	\[\mathcal E=\bigsqcup \mathcal U_j\times \mathbb C^n\big/ \sim,\]
	where for any $(x,v)\in \mathcal U_i\cap \mathcal U_j$ we identify $\big((x,v),z\big)\in \mathcal U_i\times \mathbb C^n$ with $\big((x,v),\kappa_{ij}(z)\big)\in \mathcal U_j\times \mathbb C^n$.
	
	The interest of this choice is that since the restriction $\mathcal E_{|S_xM}\to S_xM$ is the trivial bundle $(S_xM\times E_x)\to S_xM$, and $\chi_j$ is constant along the vertical fibers we have
	\[  \big(g(\Delta_{\mathbb V})\otimes \operatorname{Id}_{\mathbb C^n}\big)R_j \chi_j=R_j\chi_j g(\Delta_{\mathbb V}^E).\]

	\subsubsection{Quantization operator}

	\begin{defi}Let $m$ be an order function. For $a\in S^{m(\bullet)}_\rho(T^*\mathcal M)$, we define
		\[\operatorname{Op}_h^{\mathcal E}(a):=R^*\circ \operatorname{Diag}_{i\in I}\big(\operatorname{Op}_h^{\mathcal M}(a)\big)\circ R.\]
		Here $\operatorname{Diag}_{i\in I}\big(\operatorname{Op}_h^{\mathcal M}(a)\big)$ acts diagonally on $\bigoplus_{i\in I} C^\infty_{\rm comp}(\mathcal U_i,\mathbb{C}^n)$, with action on the $i$-th factor given by $\operatorname{Op}_h^{\mathcal M}(a)$. \end{defi}
	
	Then, the following holds.
	\begin{prop}[$L^2$-continuity]\label{prop: L2 continuity} Let $a\in S^0_1(T^*\mathcal M)$. For any flat bundle $\mathcal E\to \mathcal M$, one has
		\[\big\|\operatorname{Op}_h^{\mathcal E}(a)\big\|_{L^2(\mathcal M,\mathcal E)}\le \|a\|_\infty+\mathcal O_a(h),\]
		with the error uniform in $\mathcal E$.
		
		Let $a\in S^0_\rho(T^*\mathcal M)$. For any flat bundle $\mathcal E\to \mathcal M$, one has
		\[\big\|\operatorname{Op}_h^{\mathcal E}(a)\big\|_{L^2(\mathcal M,\mathcal E)}\le C(a),\]
		where the error is independent of $\mathcal E$ and $C(a)$ is bounded if seminorms of $a$ are bounded.\end{prop}
 \begin{proof}
		Since $R$ is an isometry $L^2(\mathcal M,\mathcal E)\to \bigoplus L^2(\mathcal U_i,\mathbb{C}^n)$, and $\operatorname{Diag}(\operatorname{Op}_h^{\mathcal M}(\eta_ia))$ preserves the direct sum $\bigoplus L^2(\mathcal U_i,\mathbb{C}^n)$, we have
		\[\big\|\operatorname{Op}_h^{\mathcal E}(a)\big\|_{L^2(\mathcal M,\mathcal E)}\le \big\|\operatorname{Op}_h^{\mathcal M}(a)\big\|_{L^2(\mathcal M)}.\]
		We conclude by invoking Proposition \ref{prop: list properties scalar quanti}.
	\end{proof}


	\begin{prop}\label{prop: product formula quantization bundle} Let $a,b\in S^{m_1(\bullet)}_{1^-},S^{m_2(\bullet)}_{1^-}$ with $m_1+m_2\le 0$ outside a compact set. Then,
		\[\operatorname{Op}_h^{\mathcal E}(a)\operatorname{Op}_h^{\mathcal E}(b)=\operatorname{Op}_h^{\mathcal E}(ab)+\mathcal O_{a,b}(h)_{L^2\to L^2}.\]
		The remainder is uniform with respect to $\mathcal E$.\end{prop}
	\begin{proof} We have
		\begin{equation*} \operatorname{Op}_h^{\mathcal E}(a)\operatorname{Op}_h^{\mathcal E}(b) =\sum_{i\in I} R_i^*\chi_i \sum_{j\in I} \operatorname{Op}_h^{\mathcal M}(a)\chi_i\chi_j R_{ij}\operatorname{Op}_h^{\mathcal M}(b_j) \chi_j R_j.\end{equation*}
		Since $R_{ij}$ is a constant unitary matrix we can interchange $\operatorname{Op}_h^{\mathcal M}(a)R_{ij}=R_{ij}\operatorname{Op}_h^{\mathcal M}(a)$ and write
		\begin{equation*} \operatorname{Op}_h^{\mathcal E}(a)\operatorname{Op}_h^{\mathcal E}(b) =\sum_{i\in I} R_i^*\chi_i \sum_{j\in I} R_{ij}\operatorname{Op}_h^{\mathcal M}(a)\chi_i\chi_j\operatorname{Op}_h^{\mathcal M}(b) \chi_j R_j.\end{equation*}
		By playing with the composition formula from Proposition~\ref{prop: list properties scalar quanti} we find
		\[\operatorname{Op}_h^{\mathcal M}(a)\chi_i\chi_j\operatorname{Op}_h^{\mathcal M}(b)\chi_j=\operatorname{Op}_h^{\mathcal M}(\chi_j^2ab)\chi_i \eta_j+\mathcal O(h)_{L^2\to L^2}.\]
		Since the operators $R_i^*$, $R_j$, $R_{ij}$ are isometries, we infer
		\begin{equation*}\operatorname{Op}_h^{\mathcal E}(a)\operatorname{Op}_h^{\mathcal E}(b)=\sum_{i\in I} R_i^*\chi_i \sum_{j\in I} R_{ij}\operatorname{Op}_h^{\mathcal M}(\chi_j^2ab)\eta_j\chi_i R_j+\mathcal O(h)_{L^2\to L^2}.\end{equation*}
		We now use $\eta_j\chi_i R_j=\eta_j\chi_i R_{ji}R_i$ and $R_{ji}R_{ij}=\operatorname{Id}$ to obtain
		\begin{equation}\label{eq: Ope(a)Ope(b)} \operatorname{Op}_h^{\mathcal E}(a)\operatorname{Op}_h^{\mathcal E}(b)=\sum_{i\in I} R_i^*\chi_i \Big(\sum_{j\in I}\operatorname{Op}_h^{\mathcal M}(\chi_j^2ab)\eta_j\Big) \chi_iR_i+\mathcal O(h)_{L^2\to L^2}.\end{equation}
		Since $\sum_{j\in I} \chi_j^2=1$ and $\chi_j\prec \eta_j$, we find 
		\begin{equation} \label{eq: nanana} \sum_j \operatorname{Op}_h^{\mathcal M}(\chi_j^2ab)\eta_j=\operatorname{Op}_h^{\mathcal M}(ab)+\mathcal O(h)_{L^2\to L^2}.\end{equation}
		Injecting \eqref{eq: nanana} into \eqref{eq: Ope(a)Ope(b)} finally gives
		\begin{equation*} \operatorname{Op}_h^{\mathcal E}(a)\operatorname{Op}_h^{\mathcal E}(b)=\operatorname{Op}_h^{\mathcal E}(ab)+\mathcal O(h)_{L^2\to L^2},\end{equation*}
		as we wished.
	\end{proof}
	
	\begin{rem}In fact, one can show that the operator $R=\operatorname{Op}_h^{\mathcal E}(a)\operatorname{Op}_h^{\mathcal E}(b)-\operatorname{Op}_h^{\mathcal E}(ab)$ is continuous $H^{r}\to H^{r-(2\rho-1)}$ for all $r\in \mathbb{R}$, where $H^r=H^r(\mathcal M,\mathcal E)$ denotes the standard Sobolev space of order $r$. \end{rem}
	As a corollary, we get:
	\begin{cor}\label{cor: inverse op (e^sg)} Let $s\in \mathbb R$. For $0<h<h_0(s)$ small enough, for any unitary flat bundle $\mathcal E\to \mathcal M$, the operator $\operatorname{Op}_h^{\mathcal E}(\mathrm{e}^{sG_m}):C^\infty(\mathcal M,\mathcal E)\to C^\infty(\mathcal M,\mathcal E)$ is invertible and
		\[\operatorname{Op}_h^{\mathcal E}(\mathrm{e}^{sG_m})^{-1}=\operatorname{Op}_h^{\mathcal E}(\mathrm{e}^{-sG_m})(\operatorname{Id}+\mathcal O_{s,m}(h)_{L^2\to L^2}).\]
	\end{cor}
	\begin{proof} By Proposition~\ref{prop: product formula quantization bundle}, one has
		\[\operatorname{Op}_h^{\mathcal E}(\mathrm{e}^{-sG_m})\operatorname{Op}_h^{\mathcal E}(\mathrm{e}^{sG_m})=\operatorname{Id}+R,\]
		with $\|R\|_{L^2\to L^2}\lesssim h$ uniformly in $\mathcal E$ and $R:H^r\to H^{r-(2\rho-1)}$ for any $r$. Let $v\in C^\infty(\mathcal M,\mathcal E)$, and set
		\[u:=(\operatorname{Id}+R)^{-1}\operatorname{Op}_h^{\mathcal E}(\mathrm{e}^{-sG_m})v\in L^2(\mathcal M,\mathcal E).\]
		Then, $(\operatorname{Id}+R)u=\operatorname{Op}_h^{\mathcal E}(\mathrm{e}^{-sG_m})v\in C^\infty(\mathcal M,\mathcal E)$. Moreover, since $R:H^r\to H^{r-(2\rho-1)}$ is continuous, the identity $u=\operatorname{Op}_h^{\mathcal E}(\mathrm{e}^{-sG_m})v-Ru$ implies
		\begin{equation} \label{eq: gain regularity sobolev} u\in H^k\Longrightarrow u\in H^{k+(2\rho-1)}.\end{equation}
		Since $u\in L^2$, it follows by iterating \eqref{eq: gain regularity sobolev} and using the Sobolev embedding theorem that $u\in C^\infty(\mathcal M,\mathcal E)$, and
		$\operatorname{Op}_h^{\mathcal E}(\mathrm{e}^{sG_m})u=v$.
	\end{proof}
		
	Finally, we extend Lemma~\ref{lem: egorov like scalar M} to the case of unitary flat bundles $\mathcal E\to \mathcal M$, in the case where $f=\varphi_t$ is the geodesic flow, for small $t$.
	\begin{prop}[Egorov theorem]\label{prop: Egorov theorem} Let $m$ be the order function from Proposition~\ref{prop: construction escape}. Let $a,b\in S^{sm(\bullet)}, S^{-sm(\bullet)}$, respectively. Then,
		\[\big\|\operatorname{Op}_h^{\mathcal E}(a)\mathrm{e}^{-t\mathbf X_{\mathcal E}}\operatorname{Op}_h^{\mathcal E}(b)\big\|_{L^2\to L^2}=\mathcal O_{a,b}(1).\]
		If, in addition, $m\circ \widetilde \varphi_t-m\le -c<0$ on $\operatorname{Supp}(b)\cap \widetilde \varphi_t^{-1}(\operatorname{Supp}(a))$, then
		\[\big\|\operatorname{Op}_h^{\mathcal E}(a)\mathrm{e}^{-t\mathbf X_{\mathcal E}}\operatorname{Op}_h^{\mathcal E}(b)\big\|_{L^2\to L^2}\lesssim \|(a\circ \widetilde \varphi_t)b\|_{\infty}+\mathcal O(h).\]
	\end{prop}
	
	\begin{proof} Since $\mathrm{e}^{t\mathbf X_{\mathcal E}}$ is unitary on $L^2$, it is equivalent to study the operator
		\[P:=\mathrm{e}^{t\mathbf X_{\mathcal E}}\operatorname{Op}_h^{\mathcal E}(a)\mathrm{e}^{-t\mathbf X_{\mathcal E}}\operatorname{Op}_h^{\mathcal E}(b)\]
		For $t$ small enough, since the connection is flat, we have:
		\[R_i\chi_i \mathrm{e}^{-t\mathbf X_{\mathcal E}}=\chi_i\mathrm{e}^{-tX} R_i\eta_i, \qquad \mathrm{e}^{t\mathbf X_{\mathcal E}}R_i^*\chi_i=R_i^* \mathrm{e}^{tX}\chi_i.\]
		Thus, 
		\[P=\sum_i R_i^*\mathrm{e}^{tX}\chi_i \operatorname{Op}_h^{\mathcal M}(a) \chi_i \mathrm{e}^{-tX}R_i \eta_i\sum_j R_j^*\chi_j \operatorname{Op}_h^{\mathcal M}(b)R_j \chi_j.\]
		By pseudolocality, provided $t$ is small enough, we can insert a factor $\eta_i$ on the right, up to adding an error $\mathcal O(h^\infty)$, so that
		\[P=\sum_i R_i^*\mathrm{e}^{tX}\chi_i \operatorname{Op}_h^{\mathcal M}(a) \chi_i \mathrm{e}^{-tX}R_i \eta_i\sum_j R_j^*\chi_j \operatorname{Op}_h^{\mathcal M}(b)R_j \chi_j\eta_i+\mathcal O(h^\infty).\]
		Recall that $R_iR_j^*\eta_i\chi_j=R_{ij}\eta_i\chi_j$ with $R_{ij}$ a constant matrix, we can thus interchange $R_{ij}$ with $\operatorname{Op}_h^{\mathcal M}(b)$ and then use $R_{ij}R_j \chi_j\eta_i=R_i\chi_j\eta_i$ to obtain
		\[P=\sum_i R_i^*\Big(\sum_j \mathrm{e}^{tX}\chi_i \operatorname{Op}_h^{\mathcal M}(a) \chi_i \mathrm{e}^{-tX}\chi_j \operatorname{Op}_h^{\mathcal M}(b)\chi_j \Big)R_i \eta_i+\mathcal O(h^\infty)_{L^2\to L^2}.\]
		We now invoke Lemma~\ref{lem: egorov like scalar M} (the presence of the cutoffs $\chi_i,\chi_j$ is harmless) with $f=\varphi_t$ to obtain for each $i,j$:
		\[\big\|\mathrm{e}^{tX}\operatorname{Op}_h^{\mathcal M}(a) \chi_i \mathrm{e}^{-tX}\chi_j \operatorname{Op}_h^{\mathcal M}(b)\big\|_{L^2\to L^2}\le C(a,b),\]
		with $C(a,b)=\mathcal O_{a,b}(1)$ in general, which improves to $C(a,b)\lesssim \|(a\circ \widetilde \varphi_t)b\|_{\infty}+\mathcal O_{a,b}(h)$ if $m\circ \widetilde \varphi_t-m\le -c<0$ on $\operatorname{Supp}(b)\cap \widetilde \varphi_t^{-1}(\operatorname{Supp}(a))$. It remains to sum over $i,j$ to conclude.

	\end{proof}
	
	\printbibliography
	
	\email
	
\end{document}